\documentclass{article}
\usepackage{amssymb}
\usepackage{amsfonts}
\usepackage{url}
\usepackage{amsmath}
\usepackage{graphicx}

\newtheorem{theorem}{Theorem}

\newtheorem{corollary}{Corollary}

\newtheorem{definition}{Definition}

\newtheorem{lemma}{Lemma}

\newtheorem{proposition}{Proposition}
\newtheorem{remark}{Remark}

\newenvironment{proof}[1][Proof]{\textbf{#1.} }{\ \rule{0.5em}{0.5em}}

\begin{document}

\title{Compactness property of the linearized Boltzmann collision operator for a multicomponent polyatomic gas}

\author{Niclas Bernhoff \footnote{niclas.bernhoff@kau.se}\\Department of Mathematics and Computer Science\\ Karlstad University,  65188 Karlstad, Sweden }

\maketitle

\abstract{The linearized Boltzmann collision operator is
fundamental in many studies of the Boltzmann equation and its main
properties are of substantial importance. The decomposition into a sum of a
positive multiplication operator, the collision frequency, and an integral
operator is trivial. Compactness of the integral operator for monatomic
single species is a classical result, while corresponding results for
monatomic mixtures and polyatomic single species are more recently obtained.
This work concerns the compactness of the operator for a multicomponent
mixture of polyatomic species, where the polyatomicity is modeled by a
discrete internal energy variable. With a probabilistic formulation of the
collision operator as a starting point, compactness is obtained by proving
that the integral operator is a sum of Hilbert-Schmidt integral operators
and operators, which are uniform limits of Hilbert-Schmidt integral
operators, under some assumptions on the collision kernel. The assumptions
are essentially generalizations of the Grad's assumptions for monatomic
single species. Self-adjointness of the linearized collision operator
follows. Moreover, bounds on - including coercivity of - the collision
frequency are obtained for a hard sphere like model. Then it follows that
the linearized collision operator is a Fredholm operator, and its domain is
also obtained.}




\section{Introduction\label{S1}}

The Boltzmann equation is a fundamental equation of kinetic theory of gases,
e.g., for computations of the flow around a space shuttle in the upper
atmosphere during reentry \cite{BBBD-18}. Studies of the main properties of
the linearized collision operator are of great importance in gaining
increased knowledge about related problems, see, e.g., \cite{Cercignani-88}
and references therein, and for related half-space problems\ \cite{BGS-06,
Be-21, Be-22, BG-21}. The linearized collision operator is obtained, by
considering deviations of an equilibrium, or Maxwellian, distribution. It
can in a natural way be written as a sum of a positive multiplication
operator, the collision frequency, and an integral operator $-K$. Compact
properties of the integral operator $K$ (for angular cut-off kernels) are
extensively studied for monatomic single species, see, e.g., \cite{Gr-63,
Dr-75, Cercignani-88, LS-10}, and more recently for monatomic
multi-component mixtures \cite{BGPS-13, Be-21a} and polyatomic single
species, where the polyatomicity is modeled by either a discrete or a
continuous internal energy variable \cite{Be-21a,Be-21b}. See also \cite%
{BBS-22} for the case of molecules undergoing resonant collisions, i.e.,
collisions where internal energy is transferred to internal energy, and
correspondingly, translational energy to translational energy, during the
collisions. The integral operator can be written as the sum of a
Hilbert-Schmidt integral operator and an approximately Hilbert-Schmidt
integral operator (cf. Lemma $\ref{LGD}$ in Section $\ref{PT1}$) \cite%
{Glassey}, and so compactness of the integral operator $K$ can be obtained.
This work extends the results of \cite{Be-21a} for monatomic multicomponent
mixtures and polyatomic single species, where the polyatomicity is modeled
by a discrete internal energy variable \cite{ErnGio-94,GS-99}, to the case
of polyatomic multicomponent mixtures, where the polyatomicity is modeled by
discrete internal energy variables. To consider mixtures of monatomic and
polyatomic molecules are of highest relevance in, e.g., the upper atmosphere 
\cite{BBBD-18}.

Following the lines of \cite{Be-21a,Be-21b}, motivated by an approach by
Kogan in \cite[Sect. 2.8]{Kogan} for the monatomic single species case, a
probabilistic formulation of the collision operator is considered as the
starting point. With this approach, it is shown, based on\ slightly modified
arguments from the ones in \cite{Be-21a}, that the integral operator $K$ can
be written as a sum of compact operators in the form of Hilbert-Schmidt
integral operators and approximately Hilbert-Schmidt integral operators -
which are uniform limits of Hilbert-Schmidt integral operators - and so
compactness of the integral operator $K$ follows. The operator $K$ is
self-adjoint, as well as, the collision frequency, why the linearized
collision operator, as the sum of two self-adjoint operators of which one is
bounded, is also self-adjoint.

For models corresponding to hard sphere models in the monatomic case, bounds
on the collision frequency are obtained. Then the collision frequency is
coercive and becomes a Fredholm operator. The set of Fredholm operators is
closed under addition with compact operators. Therefore also the linearized
collision operator becomes a Fredholm operator by the compactness of the
integral operator $K$. For hard sphere like models the linearized collision
operator satisfies all the properties of the general linear operator in the
abstract half-space problem considered in \cite{Be-21}.

The rest of the paper is organized as follows. In Section $\ref{S2}$, the
model considered is presented. The probabilistic formulation of the
collision operators considered and its relations to more classical
formulations \cite{ErnGio-94,GS-99} are accounted for in Section $\ref{S2.1}$%
.\ Some classical results for the collision operators in Section $\ref{S2.2}$
and the linearized collision operator in Section $\ref{S2.3}$ are reviewed.
Section $\ref{S3}$ is devoted to the main results of this paper, while the
main proofs are addressed in Sections $\ref{PT1}-\ref{PT2}$; a proof of
compactness of the integral operator $K$ is presented in Section $\ref{PT1}$%
, while a proof of the bounds on the collision frequency appears in Section $%
\ref{PT2}$. Finally, the appendix concerns a proof of a crucial - for the
compactness - lemma, which is an extension of a corresponding lemma for the
monatomic mixture case \cite{BGPS-13,Be-21a}.

\section{Model\label{S2}}

This section concerns the model considered. A probabilistic formulation of
the collision operator is considered, whose relation to a more classical
formulation is accounted for. Known properties of the model and
corresponding linearized collision operator are also reviewed.

Consider a multicomponent mixture of $s$ polyatomic species $a_{1},...,a_{s}$%
, with masses $m_{1},...,m_{s}$, respectively. The polyatomicity is modeled
by $r_{\alpha }$ different internal energies $I_{1}^{\alpha
},...,I_{r_{\alpha }}^{\alpha }\ $for each $\alpha \in \left\{
1,...,s\right\} $. Here the internal energies $I_{1}^{\alpha
},...,I_{r_{\alpha }}^{\alpha }$ are assumed to be nonnegative real numbers; 
$\left\{ I_{1}^{\alpha },...,I_{r_{\alpha }}^{\alpha }\right\} \subset $ $%
\mathbb{R}_{+}$ for $\alpha \in \left\{ 1,...,s\right\} $. A monatomic
species $a_{\alpha }$ can also be considered by choosing $r_{\alpha }=1$,
while $s=1$ would correspond to the case of single species.

The distribution functions are of the vector form $f=\left(
f_{1},...,f_{s}\right) $, where $f_{\alpha }=\left( f_{\alpha
,1},...,f_{\alpha ,r_{\alpha }}\right) $\ is the distribution function for
particles of species $a_{\alpha }$ for $\alpha \in \left\{ 1,...,s\right\} $%
. Here $f_{\alpha ,i}=f_{\alpha ,i}\left( t,\mathbf{x},\boldsymbol{\xi }%
\right) =f_{\alpha }\left( t,\mathbf{x},\boldsymbol{\xi },I_{i}\right) $,
with temporal variable $t\in \mathbb{R}_{+}$, spatial variable $\mathbf{x}%
=\left( x,y,z\right) \in \mathbb{R}^{3}$, and molecular velocity variable $%
\boldsymbol{\xi }=\left( \xi _{x},\xi _{y},\xi _{z}\right) \in \mathbb{R}^{3}
$, is the distribution function for particles of species $a_{\alpha }$ with
internal energy $I_{i}$ for $i\in \left\{ 1,...,r_{\alpha }\right\} $ and $%
\alpha \in \left\{ 1,...,s\right\} $.

Denote by $\Omega \subset \mathbb{N}^{6}$, 
\begin{eqnarray*}
&&\Omega :=  \\
&&\left\{ \left( \alpha ,\beta ,i,j,k,l\right) :\left\{ \alpha ,\beta
\right\} \subseteq \left\{ 1,...,s\right\} ,\left\{ i,k\right\} \subseteq
\left\{ 1,...,r_{\alpha }\right\} ,\left\{ j,l\right\} \subseteq \left\{
1,...,r_{\beta }\right\} \right\} \text{.}
\end{eqnarray*}

Moreover, denote $r=\sum_{\alpha =1}^{s}r_{\alpha }$ and consider the real
Hilbert space 
\begin{equation*}
\mathcal{\mathfrak{h}}:=\left( L^{2}\left( d\boldsymbol{\xi }\right) \right)
^{r},
\end{equation*}
with inner product%
\begin{equation*}
\left( f,g\right) =\sum_{\alpha =1}^{s}\sum_{i=1}^{r_{\alpha }}\int_{\mathbb{%
R}^{3}}f_{\alpha ,i}g_{\alpha ,i}\,d\boldsymbol{\xi }\text{, }f,g\in \left(
L^{2}\left( d\boldsymbol{\xi }\right) \right) ^{r}\text{.}
\end{equation*}

The evolution of the distribution functions is (in the absence of external
forces) described by the (vector) Boltzmann equation%
\begin{equation}
\frac{\partial f}{\partial t}+\left( \boldsymbol{\xi }\cdot \nabla _{\mathbf{%
x}}\right) f=Q\left( f,f\right) \text{,}  \label{BE1}
\end{equation}%
where the (vector) collision operator $Q=\left(
Q_{1}^{1},...,Q_{r_{1}}^{1},...,Q_{1}^{s},...,Q_{r_{s}}^{s}\right) $ is a
quadratic bilinear operator that accounts for the change of velocities and
internal energies of particles due to binary collisions (assuming that the
gas is rarefied, such that other collisions are negligible), where the
component $Q_{i}^{\alpha }$ is the collision operator for particles of
species $a_{\alpha }$ with internal energy $I_{i}$ for $i\in \left\{
1,...,r_{\alpha }\right\} $ and $\alpha \in \left\{ 1,...,s\right\} $.

A collision can, given two particles of species $a_{\alpha }$ and $a_{\beta }
$, $\left\{ \alpha ,\beta \right\} \subset \left\{ 1,...,s\right\} $,
respectively, be represented by two pre-collisional pairs, each pair
consisting of a microscopic velocity and an internal energy, $\left( 
\boldsymbol{\xi },I_{i}^{\alpha }\right) $ and $(\boldsymbol{\xi }_{\ast
},I_{j}^{\beta })$, and two corresponding post-collisional pairs, $\left( 
\boldsymbol{\xi }^{\prime },I_{k}^{\alpha }\right) $ and $(\boldsymbol{\xi }%
_{\ast }^{\prime },I_{l}^{\beta })$, for some $\left\{ \left( i,j\right)
,\left( k,l\right) \right\} \subset \left\{ 1,...,r_{\alpha }\right\} \times
\left\{ 1,...,r_{\beta }\right\} $. The notation for pre- and
post-collisional pairs may be interchanged as well. Due to momentum and
total energy conservation, the following relations have to be satisfied by
the pairs%
\begin{eqnarray}
m_{\alpha }\boldsymbol{\xi }+m_{\beta }\boldsymbol{\xi }_{\ast }
&=&m_{\alpha }\boldsymbol{\xi }^{\prime }+m_{\beta }\boldsymbol{\xi }_{\ast
}^{\prime }  \notag \\
m_{\alpha }\left\vert \boldsymbol{\xi }\right\vert ^{2}+m_{\beta }\left\vert 
\boldsymbol{\xi }_{\ast }\right\vert ^{2}+I_{i}^{\alpha }+I_{j}^{\beta }
&=&m_{\alpha }\left\vert \boldsymbol{\xi }^{\prime }\right\vert
^{2}+m_{\beta }\left\vert \boldsymbol{\xi }_{\ast }^{\prime }\right\vert
^{2}+I_{k}^{\alpha }+I_{l}^{\beta }\text{.}  \label{CI}
\end{eqnarray}

\subsection{Collision operator\label{S2.1}}

The (vector) collision operator $Q=\left(
Q_{1}^{1},...,Q_{r_{1}}^{1},...,Q_{1}^{s},...,Q_{r_{s}}^{s}\right) $ has
components that can be written in the following form 
\begin{eqnarray}
Q_{i}^{\alpha }(f,f) &=&\sum_{\beta =1}^{s}\sum\limits_{k=1}^{r_{\alpha
}}\sum\limits_{j,l=1}^{r_{\beta }}\int_{\left( \mathbb{R}^{3}\right)
^{3}}W_{\alpha \beta }(\boldsymbol{\xi },\boldsymbol{\xi }_{\ast
},I_{i}^{\alpha },I_{j}^{\beta }\left\vert \boldsymbol{\xi }^{\prime },%
\boldsymbol{\xi }_{\ast }^{\prime },I_{k}^{\alpha },I_{l}^{\beta }\right. )\,
\notag \\
&&\times \left( \frac{f_{\alpha ,k}^{\prime }f_{\beta ,l\ast }^{\prime }}{%
\varphi _{k}^{\alpha }\varphi _{l}^{\beta }}-\frac{f_{\alpha ,i}f_{\beta
,j\ast }}{\varphi _{i}^{\alpha }\varphi _{j}^{\beta }}\right) \,d\boldsymbol{%
\xi }_{\ast }d\boldsymbol{\xi }^{\prime }d\boldsymbol{\xi }_{\ast }^{\prime }
\label{c1}
\end{eqnarray}%
for some constant $\left( \varphi _{1}^{1},...,\varphi
_{r_{1}}^{1},...,\varphi _{1}^{s},...,\varphi _{r_{s}}^{s}\right) \in 
\mathbb{R}^{r}$ . Here and below the abbreviations%
\begin{equation}
f_{\alpha ,i\ast }=f_{\alpha ,i}\left( t,\mathbf{x},\boldsymbol{\xi }_{\ast
}\right) \text{, }f_{\alpha ,i}^{\prime }=f_{\alpha ,i}\left( t,\mathbf{x},%
\boldsymbol{\xi }^{\prime }\right) \text{, and }f_{\alpha ,i\ast }^{\prime
}=f_{\alpha ,i}\left( t,\mathbf{x},\boldsymbol{\xi }_{\ast }^{\prime
}\right)   \label{a1}
\end{equation}%
are used. In the collision operator $\left( \ref{c1}\right) $ the gain term
- the term containing the product $f_{\alpha ,k}^{\prime }f_{\beta ,l\ast
}^{\prime }$ - accounts for the gain of particles of species $a_{\alpha }$\
with microscopic velocity $\boldsymbol{\xi }$ and internal energy $%
I_{i}^{\alpha }$ (at time $t$ and position $\mathbf{x}$) - here $\left( 
\boldsymbol{\xi },I_{i}^{\alpha }\right) $ and $\left( \boldsymbol{\xi }%
_{\ast },I_{j}^{\beta }\right) $ represent the post-collisional particles,
while the loss term - the term containing the product $f_{\alpha ,i}f_{\beta
,j\ast }$ - accounts for the loss of particles of species $a_{\alpha }$\
with microscopic velocity $\boldsymbol{\xi }$ and internal energy $%
I_{i}^{\alpha }$ - here $\left( \boldsymbol{\xi },I_{i}^{\alpha }\right) $
and $\left( \boldsymbol{\xi }_{\ast },I_{j}^{\beta }\right) $ represent the
pre-collisional particles. The corresponding (signed) internal energy gap is%
\begin{equation*}
\Delta I_{kl,ij}^{\alpha \beta }=I_{k}^{\alpha }+I_{l}^{\beta
}-I_{i}^{\alpha }-I_{j}^{\beta }\text{.}
\end{equation*}%
The transition probabilities 
\begin{equation*}
W_{\alpha \beta }:\left( \left( \mathbb{R}^{3}\right) ^{2}\times \left\{
1,...,r_{\alpha }\right\} \times \left\{ 1,...,r_{\beta }\right\} \right)
^{2}\rightarrow \mathbb{R}_{+}:=[0,\infty )\text{, }\left\{ \alpha ,\beta
\right\} \subset \left\{ 1,...,s\right\} ,
\end{equation*}%
are of the form, cf. \cite{Be-21a},%
\begin{eqnarray}
&&W_{\alpha \beta }(\boldsymbol{\xi },\boldsymbol{\xi }_{\ast
},I_{i}^{\alpha },I_{j}^{\beta }\left\vert \boldsymbol{\xi }^{\prime },%
\boldsymbol{\xi }_{\ast }^{\prime },I_{k}^{\alpha },I_{l}^{\beta }\right. ) 
\notag \\
&=&\left( m_{\alpha }+m_{\beta }\right) ^{2}m_{\alpha }m_{\beta }\sigma
_{kl,ij}^{\alpha \beta }\left( \left\vert \mathbf{g}^{\prime }\right\vert
,\cos \theta \right) \frac{\left\vert \mathbf{g}^{\prime }\right\vert }{%
\left\vert \mathbf{g}\right\vert }\delta _{3}\left( m_{\alpha }\boldsymbol{%
\xi }+m_{\beta }\boldsymbol{\xi }_{\ast }-m_{\alpha }\boldsymbol{\xi }%
^{\prime }-m_{\beta }\boldsymbol{\xi }_{\ast }^{\prime }\right)   \notag \\
&&\times \varphi _{k}^{\alpha }\varphi _{l}^{\beta }\delta _{1}\left( \frac{1%
}{2}\left( m_{\alpha }\left\vert \boldsymbol{\xi }\right\vert ^{2}+m_{\beta
}\left\vert \boldsymbol{\xi }_{\ast }\right\vert ^{2}-m_{\alpha }\left\vert 
\boldsymbol{\xi }^{\prime }\right\vert ^{2}-m_{\beta }\left\vert \boldsymbol{%
\xi }_{\ast }^{\prime }\right\vert ^{2}\right) -\Delta I_{kl,ij}^{\alpha
\beta }\right)   \notag \\
&=&\left( m_{\alpha }+m_{\beta }\right) ^{2}m_{\alpha }m_{\beta }\sigma
_{ij,kl}^{\alpha \beta }\left( \left\vert \mathbf{g}\right\vert ,\cos \theta
\right) \frac{\left\vert \mathbf{g}\right\vert }{\left\vert \mathbf{g}%
^{\prime }\right\vert }\delta _{3}\left( m_{\alpha }\boldsymbol{\xi }%
+m_{\beta }\boldsymbol{\xi }_{\ast }-m_{\alpha }\boldsymbol{\xi }^{\prime
}-m_{\beta }\boldsymbol{\xi }_{\ast }^{\prime }\right)   \notag \\
&&\times \varphi _{i}^{\alpha }\varphi _{j}^{\beta }\delta _{1}\left( \frac{1%
}{2}\left( m_{\alpha }\left\vert \boldsymbol{\xi }\right\vert ^{2}+m_{\beta
}\left\vert \boldsymbol{\xi }_{\ast }\right\vert ^{2}-m_{\alpha }\left\vert 
\boldsymbol{\xi }^{\prime }\right\vert ^{2}-m_{\beta }\left\vert \boldsymbol{%
\xi }_{\ast }^{\prime }\right\vert ^{2}\right) -\Delta I_{kl,ij}^{\alpha
\beta }\right) \text{,}  \notag \\
&&\text{with }\sigma _{ij,kl}^{\alpha \beta }=\sigma _{ij,kl}^{\alpha \beta
}\left( \left\vert \mathbf{g}\right\vert ,\cos \theta \right) >0\text{ a.e., 
}\cos \theta =\frac{\mathbf{g}\cdot \mathbf{g}^{\prime }}{\left\vert \mathbf{%
g}\right\vert \left\vert \mathbf{g}^{\prime }\right\vert }\text{, }\mathbf{g}%
=\boldsymbol{\xi }-\boldsymbol{\xi }_{\ast }\text{,}  \notag \\
&&\mathbf{g}^{\prime }=\boldsymbol{\xi }^{\prime }-\boldsymbol{\xi }_{\ast
}^{\prime }\text{, and }\Delta I_{kl,ij}^{\alpha \beta }=I_{k}^{\alpha
}+I_{l}^{\beta }-I_{i}^{\alpha }-I_{j}^{\beta }\text{,}  \label{tp}
\end{eqnarray}%
where $\delta _{3}$ and $\delta _{1}$ denote the Dirac's delta function in $%
\mathbb{R}^{3}$ and $\mathbb{R}$, respectively; taking the conservation of
momentum and total energy $\left( \ref{CI}\right) $ into account. Here and
below we use the (inconsistent) shorthanded expressions 
\begin{eqnarray*}
\sigma _{ij,kl}^{\alpha \beta } &=&\sigma _{\alpha \beta }(\boldsymbol{\xi },%
\boldsymbol{\xi }_{\ast },I_{i}^{\alpha },I_{j}^{\beta }\left\vert 
\boldsymbol{\xi }^{\prime },\boldsymbol{\xi }_{\ast }^{\prime
},I_{k}^{\alpha },I_{l}^{\beta }\right. )=\widetilde{\sigma }_{\alpha \beta
}(\left\vert \mathbf{g}\right\vert ,\cos \theta ,I_{i}^{\alpha
},I_{j}^{\beta },I_{k}^{\alpha },I_{l}^{\beta })\text{ and} \\
\sigma _{kl,ij}^{\alpha \beta } &=&\sigma _{\alpha \beta }(\boldsymbol{\xi }%
^{\prime },\boldsymbol{\xi }_{\ast }^{\prime },I_{k}^{\alpha },I_{l}^{\beta
}\left\vert \boldsymbol{\xi },\boldsymbol{\xi }_{\ast },I_{i}^{\alpha
},I_{j}^{\beta }\right. )=\widetilde{\sigma }_{\alpha \beta }(\left\vert 
\mathbf{g}^{\prime }\right\vert ,\cos \theta ,I_{k}^{\alpha },I_{l}^{\beta
},I_{i}^{\alpha },I_{j}^{\beta })\text{,}
\end{eqnarray*}%
for given scattering cross-section 
\begin{equation*}\sigma _{\alpha \beta }:\left( \left( 
\mathbb{R}^{3}\right) ^{2}\times \left\{ 1,...,r_{\alpha }\right\} \times
\left\{ 1,...,r_{\beta }\right\} \right) ^{2}\rightarrow \mathbb{R}_{+}\text{,}
\end{equation*}
or, of the form $\widetilde{\sigma }_{\alpha \beta }:\mathbb{R}_{+}\times %
\left[ -1,1\right] \times \left\{ 1,...,r_{\alpha }\right\} ^{2}\times
\left\{ 1,...,r_{\beta }\right\} ^{2}\rightarrow \mathbb{R}_{+}$; assuming
the pairs $\left( \boldsymbol{\xi },I_{i}^{\alpha }\right) $, $(\boldsymbol{%
\xi }_{\ast },I_{j}^{\beta })$, $\left( \boldsymbol{\xi }^{\prime
},I_{k}^{\alpha }\right) $, and $(\boldsymbol{\xi }_{\ast }^{\prime
},I_{l}^{\beta })$ being given - here, by the arguments of $W_{\alpha \beta }
$ for $\left\{ \alpha ,\beta \right\} \subset \left\{ 1,...,s\right\} $,

The scattering cross sections $\sigma _{kl,ij}^{\alpha \beta }$, with $%
\left( \alpha ,\beta ,i,j,k,l\right) \in \Omega $, are assumed to satisfy
the microreversibility conditions%
\begin{equation}
\varphi _{i}^{\alpha }\varphi _{j}^{\beta }\left\vert \mathbf{g}\right\vert
^{2}\sigma _{ij,kl}^{\alpha \beta }\left( \left\vert \mathbf{g}\right\vert
,\cos \theta \right) =\varphi _{k}^{\alpha }\varphi _{l}^{\beta }\left\vert 
\mathbf{g}^{\prime }\right\vert ^{2}\sigma _{kl,ij}^{\alpha \beta }\left(
\left\vert \mathbf{g}^{\prime }\right\vert ,\cos \theta \right) \text{.}
\label{mr}
\end{equation}%
Furthermore, to obtain invariance of\ change of particles in a collision, it
is assumed that the scattering cross sections $\sigma _{kl,ij}^{\alpha \beta
}$,\ with $\left( \alpha ,\beta ,i,j,k,l\right) \in \Omega $, satisfy the
symmetry relations (fixing the pairs $\left( \boldsymbol{\xi },I_{i}^{\alpha
}\right) $, $(\boldsymbol{\xi }_{\ast },I_{j}^{\beta })$, $\left( 
\boldsymbol{\xi }^{\prime },I_{k}^{\alpha }\right) $, and $(\boldsymbol{\xi }%
_{\ast }^{\prime },I_{l}^{\beta })$) 
\begin{equation}
\sigma _{kl,ij}^{\alpha \beta }=\sigma _{lk,ji}^{\beta \alpha }\text{,}
\label{sr}
\end{equation}%
while 
\begin{equation}
\sigma _{ij,kl}^{\alpha \alpha }\left( \left\vert \mathbf{g}\right\vert
,-\cos \theta \right) =\sigma _{ij,kl}^{\alpha \alpha }\left( \left\vert 
\mathbf{g}\right\vert ,\cos \theta \right) \text{ and }\sigma
_{ij,kl}^{\alpha \alpha }=\sigma _{ji,kl}^{\alpha \alpha }=\sigma
_{ji,lk}^{\alpha \alpha }\text{.}  \label{sr1}
\end{equation}%
The invariance under change of particles in a collision, which follows
directly by the definition of the transition probability $\left( \ref{tp}%
\right) $ and the symmetry relations $\left( \ref{sr}\right) ,\left( \ref%
{sr1}\right) $ for the collision frequency, and the microreversibility of
the collisions $\left( \ref{mr}\right) $, implies that the transition
probabilities $\left( \ref{tp}\right) $ for $\left\{ \alpha ,\beta \right\}
\subset \left\{ 1,...,s\right\} $ satisfy the relations

\begin{eqnarray}
W_{\alpha \beta }(\boldsymbol{\xi },\boldsymbol{\xi }_{\ast },I_{i}^{\alpha
},I_{j}^{\beta }\left\vert \boldsymbol{\xi }^{\prime },\boldsymbol{\xi }%
_{\ast }^{\prime },I_{k}^{\alpha },I_{l}^{\beta }\right. ) &=&W_{\beta
\alpha }(\boldsymbol{\xi }_{\ast },\boldsymbol{\xi },I_{j}^{\beta
},I_{i}^{\alpha }\left\vert \boldsymbol{\xi }_{\ast }^{\prime },\boldsymbol{%
\xi }^{\prime },I_{l}^{\beta },I_{k}^{\alpha }\right. )  \notag \\
W_{\alpha \beta }(\boldsymbol{\xi },\boldsymbol{\xi }_{\ast },I_{i}^{\alpha
},I_{j}^{\beta }\left\vert \boldsymbol{\xi }^{\prime },\boldsymbol{\xi }%
_{\ast }^{\prime },I_{k}^{\alpha },I_{l}^{\beta }\right. ) &=&W_{\alpha
\beta }(\boldsymbol{\xi }^{\prime },\boldsymbol{\xi }_{\ast }^{\prime
},I_{k}^{\alpha },I_{l}^{\beta }\left\vert \boldsymbol{\xi },\boldsymbol{\xi 
}_{\ast },I_{i}^{\alpha },I_{j}^{\beta }\right. )  \notag \\
W_{\alpha \alpha }(\boldsymbol{\xi },\boldsymbol{\xi }_{\ast },I_{i}^{\alpha
},I_{j}^{\beta }\left\vert \boldsymbol{\xi }^{\prime },\boldsymbol{\xi }%
_{\ast }^{\prime },I_{k}^{\alpha },I_{l}^{\beta }\right. ) &=&W_{\alpha
\alpha }(\boldsymbol{\xi },\boldsymbol{\xi }_{\ast },I_{i}^{\alpha
},I_{j}^{\alpha }\left\vert \boldsymbol{\xi }_{\ast }^{\prime },\boldsymbol{%
\xi }^{\prime },I_{l}^{\alpha },I_{k}^{\alpha }\right. )\text{.}
\label{rel1}
\end{eqnarray}

Applying known properties of Dirac's delta function, the transition
probabilities - aiming to obtain expressions for $\mathbf{G}_{\alpha \beta
}^{\prime }=\dfrac{m_{\alpha }\boldsymbol{\xi }^{\prime }+m_{\beta }%
\boldsymbol{\xi }_{\ast }^{\prime }}{m_{\alpha }+m_{\beta }}$ and $%
\left\vert \mathbf{g}^{\prime }\right\vert $ in the arguments of the
delta-functions - may be transformed to 
\begin{eqnarray*}
&&W_{\alpha \beta }(\boldsymbol{\xi },\boldsymbol{\xi }_{\ast
},I_{i}^{\alpha },I_{j}^{\beta }\left\vert \boldsymbol{\xi }^{\prime },%
\boldsymbol{\xi }_{\ast }^{\prime },I_{k}^{\alpha },I_{l}^{\beta }\right. )
\\
&=&\left( m_{\alpha }+m_{\beta }\right) ^{2}m_{\alpha }m_{\beta }\varphi
_{k}^{\alpha }\varphi _{l}^{\beta }\sigma _{kl,ij}^{\alpha \beta }\frac{%
\left\vert \mathbf{g}^{\prime }\right\vert }{\left\vert \mathbf{g}%
\right\vert }\delta _{3}\left( \left( m_{\alpha }+m_{\beta }\right) \left( 
\mathbf{G}_{\alpha \beta }-\mathbf{G}_{\alpha \beta }^{\prime }\right)
\right) \\
&&\times \delta _{1}\left( \dfrac{m_{\alpha }m_{\beta }}{2\left( m_{\alpha
}+m_{\beta }\right) }\left( \left\vert \mathbf{g}\right\vert ^{2}-\left\vert 
\mathbf{g}^{\prime }\right\vert ^{2}\right) -\Delta I_{kl,ij}^{\alpha \beta
}\right) \\
&=&2\varphi _{k}^{\alpha }\varphi _{l}^{\beta }\sigma _{kl,ij}^{\alpha \beta
}\frac{\left\vert \mathbf{g}^{\prime }\right\vert }{\left\vert \mathbf{g}%
\right\vert }\delta _{3}\left( \mathbf{G}_{\alpha \beta }-\mathbf{G}_{\alpha
\beta }^{\prime }\right) \delta _{1}\left( \left\vert \mathbf{g}\right\vert
^{2}-\left\vert \mathbf{g}^{\prime }\right\vert ^{2}-2\frac{m_{\alpha
}+m_{\beta }}{m_{\alpha }m_{\beta }}\Delta I_{kl,ij}^{\alpha \beta }\right)
\\
&=&\varphi _{k}^{\alpha }\varphi _{l}^{\beta }\sigma _{kl,ij}^{\alpha \beta }%
\frac{1}{\left\vert \mathbf{g}\right\vert }\mathbf{1}_{\left\vert \mathbf{g}%
\right\vert ^{2}>2\widetilde{\Delta }I_{kl,ij}^{\alpha \beta }}\delta
_{3}\!\left( \mathbf{G}_{\alpha \beta }-\mathbf{G}_{\alpha \beta }^{\prime
}\right) \delta _{1}\!\left( \sqrt{\left\vert \mathbf{g}\right\vert ^{2}-2%
\widetilde{\Delta }I_{kl,ij}^{\alpha \beta }}-\left\vert \mathbf{g}^{\prime
}\right\vert \right) \\
&=&\varphi _{i}^{\alpha }\varphi _{j}^{\beta }\sigma _{ij,kl}^{\alpha \beta }%
\frac{\left\vert \mathbf{g}\right\vert }{\left\vert \mathbf{g}^{\prime
}\right\vert ^{2}}\mathbf{1}_{\left\vert \mathbf{g}\right\vert ^{2}>2%
\widetilde{\Delta }I_{kl,ij}^{\alpha \beta }}\!\delta _{3}\!\left( \mathbf{G}%
_{\alpha \beta }-\mathbf{G}_{\alpha \beta }^{\prime }\right) \delta
_{1}\left( \sqrt{\left\vert \mathbf{g}\right\vert ^{2}-2\widetilde{\Delta }%
I_{kl,ij}^{\alpha \beta }}-\left\vert \mathbf{g}^{\prime }\right\vert
\right) \text{,} \\
\text{ } &&\text{with }\mathbf{G}_{\alpha \beta }=\frac{m_{\alpha }%
\boldsymbol{\xi }+m_{\beta }\boldsymbol{\xi }_{\ast }}{m_{\alpha }+m_{\beta }%
}\text{ and }\widetilde{\Delta }I_{kl,ij}^{\alpha \beta }=\frac{m_{\alpha
}+m_{\beta }}{m_{\alpha }m_{\beta }}\Delta I_{kl,ij}^{\alpha \beta }\text{.}
\end{eqnarray*}

\begin{remark}
Note that, cf. \cite{HD-69},%
\begin{equation*}
\delta _{1}\left( \dfrac{m_{\alpha }m_{\beta }}{2\left( m_{\alpha }+m_{\beta
}\right) }\left( \left\vert \mathbf{g}\right\vert ^{2}-\left\vert \mathbf{g}%
^{\prime }\right\vert ^{2}\right) -\Delta I_{kl,ij}^{\alpha \beta }\right)
=\delta _{1}\left( E_{ij}^{\alpha \beta }-E_{kl}^{\alpha \beta }\right) ,
\end{equation*}%
for $E_{ij}^{\alpha \beta }=\dfrac{m_{\alpha }m_{\beta }}{2\left( m_{\alpha
}+m_{\beta }\right) }\left\vert \mathbf{g}\right\vert ^{2}+I_{i}^{\alpha
}+I_{j}^{\beta }\ $and $E_{kl}^{\alpha \beta }=\dfrac{m_{\alpha }m_{\beta }}{%
2\left( m_{\alpha }+m_{\beta }\right) }\left\vert \mathbf{g}^{\prime
}\right\vert ^{2}+I_{k}^{\alpha }+I_{l}^{\beta }$.
\end{remark}

By a change of variables $\left\{ \mathbf{g}^{\prime }=\boldsymbol{\xi }%
^{\prime }-\boldsymbol{\xi }_{\ast }^{\prime },\mathbf{G}_{\alpha \beta
}^{\prime }=\dfrac{m_{\alpha }\boldsymbol{\xi }^{\prime }+m_{\beta }%
\boldsymbol{\xi }_{\ast }^{\prime }}{m_{\alpha }+m_{\beta }}\right\} $
followed by one to spherical coordinates, noting that%
\begin{equation}
d\boldsymbol{\xi }^{\prime }d\boldsymbol{\xi }_{\ast }^{\prime }=d\mathbf{G}%
_{\alpha \beta }^{\prime }d\mathbf{g}^{\prime }=\left\vert \mathbf{g}%
^{\prime }\right\vert ^{2}d\mathbf{G}_{\alpha \beta }^{\prime }d\left\vert 
\mathbf{g}^{\prime }\right\vert d\boldsymbol{\omega }\text{, with }%
\boldsymbol{\omega }=\frac{\mathbf{g}^{\prime }}{\left\vert \mathbf{g}%
^{\prime }\right\vert }\text{,}  \label{df1}
\end{equation}%
the observation that%
\begin{eqnarray*}
&&Q_{i}^{\alpha }(f,f) \\
&=&\sum_{\beta =1}^{s}\sum\limits_{k=1}^{r_{\alpha
}}\sum\limits_{j,l=1}^{r_{\beta }}\int_{\left( \mathbb{R}^{3}\right)
^{2}\times \mathbb{R}_{+}\mathbb{\times S}^{2}}W_{\alpha \beta }(\boldsymbol{%
\xi },\boldsymbol{\xi }_{\ast },I_{i}^{\alpha },I_{j}^{\beta }\left\vert 
\boldsymbol{\xi }^{\prime },\boldsymbol{\xi }_{\ast }^{\prime
},I_{k}^{\alpha },I_{l}^{\beta }\right. ) \\
&&\times \left( \frac{f_{\alpha ,k}^{\prime }f_{\beta ,l\ast }^{\prime }}{%
\varphi _{k}^{\alpha }\varphi _{l}^{\beta }}-\frac{f_{\alpha ,i}f_{\beta
,j\ast }}{\varphi _{i}^{\alpha }\varphi _{j}^{\beta }}\right) \,\left\vert 
\mathbf{g}^{\prime }\right\vert ^{2}\,d\boldsymbol{\xi }_{\ast }d\mathbf{G}%
_{\alpha \beta }^{\prime }d\left\vert \mathbf{g}^{\prime }\right\vert d%
\boldsymbol{\omega } \\
&=&\sum_{\beta =1}^{s}\sum\limits_{k=1}^{r_{\alpha
}}\sum\limits_{j,l=1}^{r_{\beta }}\int_{\mathbb{R}^{3}\mathbb{\times S}%
^{2}}\sigma _{ij,kl}^{\alpha \beta }\left\vert \mathbf{g}\right\vert \left(
f_{\alpha ,k}^{\prime }f_{\beta ,l\ast }^{\prime }\frac{\varphi _{i}^{\alpha
}\varphi _{j}^{\beta }}{\varphi _{k}^{\alpha }\varphi _{l}^{\beta }}%
-f_{\alpha ,i}f_{\beta ,j\ast }\right) \,d\boldsymbol{\xi }_{\ast }d%
\boldsymbol{\omega ,}
\end{eqnarray*}%
where%
\begin{equation*}
\left\{ 
\begin{array}{l}
\boldsymbol{\xi }^{\prime }=\mathbf{G}_{\alpha \beta }+\dfrac{m_{\beta }}{%
m_{\alpha }+m_{\beta }}\sqrt{\left\vert \mathbf{g}\right\vert ^{2}-2\dfrac{%
m_{\alpha }+m_{\beta }}{m_{\alpha }m_{\beta }}\Delta I_{kl,ij}^{\alpha \beta
}}\omega \medskip \\ 
\boldsymbol{\xi }_{\ast }^{\prime }=\mathbf{G}_{\alpha \beta }-\dfrac{%
m_{\alpha }}{m_{\alpha }+m_{\beta }}\sqrt{\left\vert \mathbf{g}\right\vert
^{2}-2\dfrac{m_{\alpha }+m_{\beta }}{m_{\alpha }m_{\beta }}\Delta
I_{kl,ij}^{\alpha \beta }}\omega%
\end{array}%
\right. \text{, }\omega \in S^{2}\text{,}
\end{equation*}%
can be made, resulting in a more familiar form of the Boltzmann collision
operator for mixtures with polyatomic molecules modeled with a discrete
energy variable, cf. e.g. \cite{ErnGio-94,GS-99}.

\begin{remark}
\label{Rem1}Note that, when considering spherical coordinates, we, maybe
unconventionally, often\ represent the direction by a vector in $\mathbb{S}%
^{2}$, rather than with azimuthal and polar angels, still referring to it as
spherical coordinates. By representing the direction by a unit vector, the
sine of the polar angle will not appear as a factor in the Jacobian,
resulting in the Jacobian to be the square of the radial length.
\end{remark}

\subsection{Collision invariants and Maxwellian distributions\label{S2.2}}

The following lemma follows directly by the relations $\left( \ref{rel1}%
\right) $.

\begin{lemma}
\label{L0}For any $\left( \alpha ,\beta ,i,j,k,l\right) \in \Omega $, the
measure 
\begin{equation*}
dA_{ij,kl}^{\alpha \beta }=W_{\alpha \beta }(\boldsymbol{\xi },\boldsymbol{%
\xi }_{\ast },I_{i}^{\alpha },I_{j}^{\beta }\left\vert \boldsymbol{\xi }%
^{\prime },\boldsymbol{\xi }_{\ast }^{\prime },I_{k}^{\alpha },I_{l}^{\beta
}\right. )d\boldsymbol{\xi \,}d\boldsymbol{\xi }_{\ast }d\boldsymbol{\xi }%
^{\prime }d\boldsymbol{\xi }_{\ast }^{\prime }
\end{equation*}%
is invariant under the (ordered) interchange%
\begin{equation}
\left( \boldsymbol{\xi },\boldsymbol{\xi }_{\ast },I_{i}^{\alpha
},I_{j}^{\beta }\right) \leftrightarrow \left( \boldsymbol{\xi }^{\prime },%
\boldsymbol{\xi }_{\ast }^{\prime },I_{k}^{\alpha },I_{l}^{\beta }\right)
\label{tr}
\end{equation}%
of variables, while%
\begin{equation*}
dA_{ij,kl}^{\alpha \beta }+dA_{ji,lk}^{\beta \alpha }
\end{equation*}%
is invariant under the (ordered) interchange of variables%
\begin{equation}
\left( \boldsymbol{\xi },\boldsymbol{\xi }^{\prime },I_{i}^{\alpha
},I_{k}^{\alpha }\right) \leftrightarrow \left( \boldsymbol{\xi }_{\ast },%
\boldsymbol{\xi }_{\ast }^{\prime },I_{j}^{\beta },I_{l}^{\beta }\right) 
\text{.}  \label{tr1}
\end{equation}
\end{lemma}

The weak form of the collision operator $Q(f,f)$ reads%
\begin{eqnarray*}
\left( Q(f,f),g\right) &=&\sum_{\alpha ,\beta
=1}^{s}\sum\limits_{i,k=1}^{r_{\alpha }}\sum\limits_{j,l=1}^{r_{\beta
}}\int_{\left( \mathbb{R}^{3}\right) ^{4}}\left( \frac{f_{\alpha ,k}^{\prime
}f_{\beta ,l\ast }^{\prime }}{\varphi _{k}^{\alpha }\varphi _{l}^{\beta }}-%
\frac{f_{\alpha ,i}f_{\beta ,j\ast }}{\varphi _{i}^{\alpha }\varphi
_{j}^{\beta }}\right) g_{\alpha ,i}\,dA_{ij,kl}^{\alpha \beta } \\
&=&\sum_{\alpha ,\beta =1}^{s}\sum\limits_{i,k=1}^{r_{\alpha
}}\sum\limits_{j,l=1}^{r_{\beta }}\int_{\left( \mathbb{R}^{3}\right)
^{4}}\left( \frac{f_{\alpha ,k}^{\prime }f_{\beta ,l\ast }^{\prime }}{%
\varphi _{k}^{\alpha }\varphi _{l}^{\beta }}-\frac{f_{\alpha ,i}f_{\beta
,j\ast }}{\varphi _{i}^{\alpha }\varphi _{j}^{\beta }}\right) g_{_{\beta
,j\ast }}\,dA_{ij,kl}^{\alpha \beta } \\
&=&-\sum_{\alpha ,\beta =1}^{s}\sum\limits_{i,k=1}^{r_{\alpha
}}\sum\limits_{j,l=1}^{r_{\beta }}\int_{\left( \mathbb{R}^{3}\right)
^{4}}\left( \frac{f_{\alpha ,k}^{\prime }f_{\beta ,l\ast }^{\prime }}{%
\varphi _{k}^{\alpha }\varphi _{l}^{\beta }}-\frac{f_{\alpha ,i}f_{\beta
,j\ast }}{\varphi _{i}^{\alpha }\varphi _{j}^{\beta }}\right) g_{\alpha
,k}^{\prime }\,dA_{ij,kl}^{\alpha \beta } \\
&=&-\sum_{\alpha ,\beta =1}^{s}\sum\limits_{i,k=1}^{r_{\alpha
}}\sum\limits_{j,l=1}^{r_{\beta }}\int_{\left( \mathbb{R}^{3}\right)
^{4}}\left( \frac{f_{\alpha ,k}^{\prime }f_{\beta ,l\ast }^{\prime }}{%
\varphi _{k}^{\alpha }\varphi _{l}^{\beta }}-\frac{f_{\alpha ,i}f_{\beta
,j\ast }}{\varphi _{i}^{\alpha }\varphi _{j}^{\beta }}\right) g_{\beta
,l\ast }^{\prime }\,dA_{ij,kl}^{\alpha \beta }
\end{eqnarray*}%
for any function $g=\left( g_{1},...,g_{s}\right) $, with $g_{\alpha
}=\left( g_{\alpha ,1},...,g_{\alpha ,r_{\alpha }}\right) $, such that the
first integrals are defined for all $\left( \alpha ,\beta ,i,j,k,l\right)
\in \Omega $, while the following equalities are obtained by applying Lemma $%
\ref{L0}$.

Denote for any function $g=\left( g_{1},...,g_{s}\right) $, with $g_{\alpha
}=\left( g_{\alpha ,1},...,g_{\alpha ,r_{\alpha }}\right) $, 
\begin{equation*}
\Delta _{ij,kl}^{\alpha \beta }\left( g\right) =g_{\alpha ,i}+g_{\beta
,j\ast }-g_{\alpha ,k}^{\prime }-g_{\beta ,l\ast }^{\prime }\text{, }\left(
\alpha ,\beta ,i,j,k,l\right) \in \Omega \text{.}
\end{equation*}

We have the following proposition.

\begin{proposition}
\label{P1}Let $g=\left( g_{1},...,g_{s}\right) $, with $g_{\alpha }=\left(
g_{\alpha ,1},...,g_{\alpha ,r_{\alpha }}\right) $, be such that for all $%
\left( \alpha ,\beta ,i,j,k,l\right) \in \Omega $ 
\begin{equation*}
\int_{\left( \mathbb{R}^{3}\right) ^{4}}\frac{f_{\alpha ,k}^{\prime
}f_{\beta ,l\ast }^{\prime }}{\varphi _{k}^{\alpha }\varphi _{l}^{\beta }}-%
\frac{f_{\alpha ,i}f_{\beta ,j\ast }}{\varphi _{i}^{\alpha }\varphi
_{j}^{\beta }}g_{\alpha ,i}\,dA_{ij,kl}^{\alpha \beta }
\end{equation*}%
is defined. Then%
\begin{equation*}
\left( Q(f,f),g\right) =\frac{1}{4}\sum_{\alpha ,\beta
=1}^{s}\sum\limits_{i,k=1}^{r_{\alpha }}\sum\limits_{j,l=1}^{r_{\beta
}}\int_{\left( \mathbb{R}^{3}\right) ^{4}}\left( \frac{f_{\alpha ,k}^{\prime
}f_{\beta ,l\ast }^{\prime }}{\varphi _{k}^{\alpha }\varphi _{l}^{\beta }}-%
\frac{f_{\alpha ,i}f_{\beta ,j\ast }}{\varphi _{i}^{\alpha }\varphi
_{j}^{\beta }}\right) \Delta _{ij,kl}^{\alpha \beta }\left( g\right)
\,dA_{ij}^{kl}.
\end{equation*}
\end{proposition}

\begin{definition}
A function $g=\left( g_{1},...,g_{s}\right) $, with $g_{\alpha }=\left(
g_{\alpha ,1},...,g_{\alpha ,r_{\alpha }}\right) $,$\ $is a collision
invariant if 
\begin{equation*}
\Delta _{ij,kl}^{\alpha \beta }\left( g\right) \,W_{\alpha \beta }(%
\boldsymbol{\xi },\boldsymbol{\xi }_{\ast },I_{i}^{\alpha },I_{j}^{\beta
}\left\vert \boldsymbol{\xi }^{\prime },\boldsymbol{\xi }_{\ast }^{\prime
},I_{k}^{\alpha },I_{l}^{\beta }\right. )=0\text{ a.e.}
\end{equation*}%
for all $\left( \alpha ,\beta ,i,j,k,l\right) \in \Omega $.
\end{definition}

Denote%
\begin{eqnarray*}
I &=&(I_{1}^{1},...,I_{r_{1}}^{1},...,I_{1}^{s},...,I_{r_{s}}^{s})\text{ and}
\\
e_{\alpha } &=&\left( 0_{r_{1}},...,0_{r_{\alpha -1}},1_{r_{\alpha
}},0_{r_{\alpha +1}},...,0_{r_{s}}\right) \in \mathbb{R}^{r}\text{ for }%
\alpha \in \left\{ 1,...,s\right\} \text{,}
\end{eqnarray*}%
where $0_{r_{\alpha }}=(0,...,0)\in \mathbb{R}^{r_{\alpha }}$ and $%
1_{r_{\alpha }}=(1,...,1)\in \mathbb{R}^{r_{\alpha }}$ for $\alpha \in
\left\{ 1,...,s\right\} $. It is clear that $e_{1},...,e_{s},m\xi _{x},m\xi
_{y},m\xi _{z},$ and $m\left\vert \boldsymbol{\xi }\right\vert ^{2}+2I$,
with $m=\sum_{\alpha =1}^{s}m_{\alpha }e_{\alpha }$, are collision
invariants - corresponding to conservation of mass(es), momentum, and total
energy.

In fact, we have the following proposition, cf. \cite{GS-99, Cercignani-88}.

\begin{proposition}
\label{P2}The vector space of collision invariants is generated by 
\begin{equation*}
\left\{ e_{1},...,e_{s},m\xi _{x},m\xi _{y},m\xi _{z},m\left\vert 
\boldsymbol{\xi }\right\vert ^{2}+2I\right\} \text{, with }m=\sum_{\alpha
=1}^{s}m_{\alpha }e_{\alpha }.
\end{equation*}
\end{proposition}

Define%
\begin{equation*}
\mathcal{W}\left[ f\right] :=\left( Q(f,f),\log \left( \varphi ^{-1}f\right)
\right) ,
\end{equation*}%
where $\varphi =\mathrm{diag}\left( \varphi _{1}^{1},...,\varphi
_{r_{1}}^{1},...,\varphi _{1}^{s},...,\varphi _{r_{s}}^{s}\right) $. It
follows by Proposition $\ref{P1}$ that%
\begin{eqnarray*}
\mathcal{W}\left[ f\right] &=&-\frac{1}{4}\sum_{\alpha ,\beta
=1}^{s}\sum\limits_{i,k=1}^{r_{\alpha }}\sum\limits_{j,l=1}^{r_{\beta
}}\int_{\left( \mathbb{R}^{3}\right) ^{4}}\frac{f_{\alpha ,i}f_{\beta ,j\ast
}}{\varphi _{i}^{\alpha }\varphi _{j}^{\beta }}\left( \frac{\varphi
_{i}^{\alpha }\varphi _{j}^{\beta }f_{\alpha ,k}^{\prime }f_{\beta ,l\ast
}^{\prime }}{f_{\alpha ,i}f_{\beta ,j\ast }\varphi _{k}^{\alpha }\varphi
_{l}^{\beta }}-1\right) \\
&&\times \log \left( \frac{\varphi _{i}^{\alpha }\varphi _{j}^{\beta
}f_{\alpha ,k}^{\prime }f_{\beta ,l\ast }^{\prime }}{f_{\alpha ,i}f_{\beta
,j\ast }\varphi _{k}^{\alpha }\varphi _{l}^{\beta }}\right)
\,dA_{ij,kl}^{\alpha \beta }\text{.}
\end{eqnarray*}%
Since $\left( x-1\right) \mathrm{log}\left( x\right) \geq 0$ for $x>0$, with
equality if and only if $x=1$,%
\begin{equation*}
\mathcal{W}\left[ f\right] \leq 0\text{,}
\end{equation*}%
with equality if and only if 
\begin{equation}
\left( \frac{f_{\alpha ,k}^{\prime }f_{\beta ,l\ast }^{\prime }}{\varphi
_{k}^{\alpha }\varphi _{l}^{\beta }}-\frac{f_{\alpha ,i}f_{\beta ,j\ast }}{%
\varphi _{i}^{\alpha }\varphi _{j}^{\beta }}\right) W_{\alpha \beta }(%
\boldsymbol{\xi },\boldsymbol{\xi }_{\ast },I_{i}^{\alpha },I_{j}^{\beta
}\left\vert \boldsymbol{\xi }^{\prime },\boldsymbol{\xi }_{\ast }^{\prime
},I_{k}^{\alpha },I_{l}^{\beta }\right. )=0\text{ a.e.}  \label{m1}
\end{equation}%
for all $\left( \alpha ,\beta ,i,j,k,l\right) \in \Omega $, or,
equivalently, if and only if%
\begin{equation*}
Q(f,f)\equiv 0\text{.}
\end{equation*}

For any equilibrium, or, Maxwellian, distribution $M=(M_{1},...,M_{s})$,
with $M_{\alpha }=\left( M_{\alpha ,1},...,M_{\alpha ,r_{\alpha }}\right) $,
it follows by equation $\left( \ref{m1}\right) $, since $Q(M,M)\equiv 0$,
that for any $\left( \alpha ,\beta ,i,j,k,l\right) \in \Omega $%
\begin{eqnarray*}
&&\left( \log \frac{M_{\alpha ,i}}{\varphi _{i}^{\alpha }}+\log \frac{%
M_{\beta ,j\ast }}{\varphi _{j}^{\beta }}-\log \frac{M_{\alpha ,k}^{\prime }%
}{\varphi _{k}^{\alpha }}-\log \frac{M_{\beta ,l\ast }^{\prime }}{\varphi
_{l}^{\beta }}\right)  \\
\times  &&W_{\alpha \beta }(\boldsymbol{\xi },\boldsymbol{\xi }_{\ast
},I_{i}^{\alpha },I_{j}^{\beta }\left\vert \boldsymbol{\xi }^{\prime },%
\boldsymbol{\xi }_{\ast }^{\prime },I_{k}^{\alpha },I_{l}^{\beta }\right. )=0%
\text{ a.e. .}
\end{eqnarray*}%
Hence, $\log \left( \varphi ^{-1}M\right) =\left( \log \dfrac{M_{1,1}}{%
\varphi _{1}^{1}},...,\log \dfrac{M_{s,r_{s}}}{\varphi _{s}^{r_{s}}}\right) $
is a collision invariant, and the components of the Maxwellian distributions 
$M=(M_{1},...,M_{s})$ are of the form 
\begin{equation*}
M_{\alpha ,i}=\dfrac{n_{\alpha }\varphi _{i}^{\alpha }m_{\alpha }^{3/2}}{%
\left( 2\pi T\right) ^{3/2}q_{\alpha }}e^{-\left( m_{\alpha }\left\vert 
\boldsymbol{\xi }-\mathbf{u}\right\vert ^{2}+2I_{i}^{\alpha }\right) /\left(
2T\right) }\text{,}
\end{equation*}%
where $n_{\alpha }=\left( M,e_{_{\alpha }}\right) $, $\mathbf{u}=\dfrac{1}{%
\rho }\left( M,m\boldsymbol{\xi }\right) $, and $T=\dfrac{1}{3n}\left(
M,m\left\vert \boldsymbol{\xi }-\mathbf{u}\right\vert ^{2}\right) $, with $%
n=\sum\limits_{\alpha =1}^{s}n_{\alpha }$, $\rho =\sum\limits_{\alpha
=1}^{s}m_{\alpha }n_{\alpha }$, and $m=\sum\limits_{\alpha =1}^{s}m_{\alpha
}e_{\alpha }$, while $q_{\alpha }=\sum\limits_{i=1}^{r_{\alpha }}\varphi
_{i}^{\alpha }e^{-I_{i}^{\alpha }/T}$, for $i\in \left\{ 1,...,r_{\alpha
}\right\} $ and $\alpha \in \left\{ 1,...,s\right\} $.

Note that, by equation $\left( \ref{m1}\right) $, any Maxwellian
distribution, or, just Maxwellian, $M=(M_{1},...,M_{s})$, with $M_{\alpha
}=\left( M_{\alpha ,1},...,M_{\alpha ,r_{\alpha }}\right) $, for any $\left(
\alpha ,\beta ,i,j,k,l\right) \in \Omega $ satisfies the relation 
\begin{equation}
\left( \frac{M_{\alpha ,k}^{\prime }M_{\beta ,l\ast }^{\prime }}{\varphi
_{k}^{\alpha }\varphi _{l}^{\beta }}-\frac{M_{\alpha ,i}M_{\beta ,j\ast }}{%
\varphi _{i}^{\alpha }\varphi _{j}^{\beta }}\right) W_{\alpha \beta }(%
\boldsymbol{\xi },\boldsymbol{\xi }_{\ast },I_{i}^{\alpha },I_{j}^{\beta
}\left\vert \boldsymbol{\xi }^{\prime },\boldsymbol{\xi }_{\ast }^{\prime
},I_{k}^{\alpha },I_{l}^{\beta }\right. )=0\text{ a.e. }.  \label{M1}
\end{equation}

\begin{remark}
Introducing the $\mathcal{H}$-functional%
\begin{equation*}
\mathcal{H}\left[ f\right] =\left( f,\log f\right) \text{,}
\end{equation*}%
an $\mathcal{H}$-theorem can be obtained.
\end{remark}

\subsection{Linearized collision operator\label{S2.3}}

Considering a deviation of a Maxwellian distribution $M=(M_{1},...,M_{s})$,
with $M_{\alpha }=\left( M_{\alpha ,1},...,M_{\alpha ,r_{\alpha }}\right) $,
where $M_{\alpha ,i}=\dfrac{n_{\alpha }\varphi _{i}^{\alpha }m_{\alpha
}^{3/2}}{\left( 2\pi \right) ^{3/2}q_{\alpha }}e^{-m_{\alpha }\left\vert 
\boldsymbol{\xi }\right\vert ^{2}/2}e^{-I_{i}^{\alpha }}$, of the form%
\begin{equation}
f=M+\mathcal{M}^{1/2}h  \label{s1}
\end{equation}%
where $\mathcal{M}=\mathrm{diag}\left(
M_{1,1},...,M_{1,r_{1}},...,M_{s,1},...,M_{s,r_{s}}\right) $, results, by
insertion in the Boltzmann equation $\left( \ref{BE1}\right) $, in the system%
\begin{equation}
\frac{\partial h}{\partial t}+\left( \boldsymbol{\xi }\cdot \nabla _{\mathbf{%
x}}\right) h+\mathcal{L}h=\Gamma \left( h,h\right) \text{,}  \label{LBE}
\end{equation}%
where the components of the linearized collision operator $\mathcal{L}%
=\left( \mathcal{L}_{1},...,\mathcal{L}_{s}\right) $, with $\mathcal{L}%
_{\alpha }=\left( \mathcal{L}_{\alpha ,1},...,\mathcal{L}_{\alpha ,r_{\alpha
}}\right) $, are given by \ 
\begin{eqnarray}
\mathcal{L}_{\alpha ,i}h &=&-M_{\alpha ,i}^{-1/2}\left( Q_{i}^{\alpha }(M,%
\mathcal{M}^{1/2}h)+Q_{i}^{\alpha }(\mathcal{M}^{1/2}h,M)\right)   \notag \\
&=&\sum\limits_{\beta =1}^{s}\sum\limits_{k=1}^{r_{\alpha
}}\sum\limits_{j,l=1}^{r_{\beta }}\int_{\left( \mathbb{R}^{3}\right)
^{3}}\left( \frac{M_{\beta ,j\ast }M_{\alpha ,k}^{\prime }M_{\beta ,l\ast
}^{\prime }}{\varphi _{i}^{\alpha }\varphi _{j}^{\beta }\varphi _{k}^{\alpha
}\varphi _{l}^{\beta }}\right) ^{1/2}  \notag \\
&&\times W_{\alpha \beta }(\boldsymbol{\xi },\boldsymbol{\xi }_{\ast
},I_{i}^{\alpha },I_{j}^{\beta }\left\vert \boldsymbol{\xi }^{\prime },%
\boldsymbol{\xi }_{\ast }^{\prime },I_{k}^{\alpha },I_{l}^{\beta }\right.
)\Delta _{ij,kl}^{\alpha \beta }\left( \mathcal{M}^{-1/2}h\right) \,d%
\boldsymbol{\xi }_{\ast }d\boldsymbol{\xi }^{\prime }d\boldsymbol{\xi }%
_{\ast }^{\prime }  \notag \\
&=&\nu _{\alpha ,i}h_{\alpha ,i}-K_{\alpha ,i}\left( h\right) \text{,}
\label{dec2}
\end{eqnarray}%
with 
\begin{eqnarray}
\nu _{\alpha ,i} &=&\sum\limits_{\beta =1}^{s}\sum\limits_{k=1}^{r_{\alpha
}}\sum\limits_{j,l=1}^{r_{\beta }}\int_{\left( \mathbb{R}^{3}\right) ^{3}}%
\frac{M_{\beta ,j\ast }}{\varphi _{i}^{\alpha }\varphi _{j}^{\beta }}%
W_{\alpha \beta }(\boldsymbol{\xi },\boldsymbol{\xi }_{\ast },I_{i}^{\alpha
},I_{j}^{\beta }\left\vert \boldsymbol{\xi }^{\prime },\boldsymbol{\xi }%
_{\ast }^{\prime },I_{k}^{\alpha },I_{l}^{\beta }\right. )d\boldsymbol{\xi }%
_{\ast }d\boldsymbol{\xi }^{\prime }d\boldsymbol{\xi }_{\ast }^{\prime }%
\text{,}  \notag \\
K_{\alpha ,i} &=&\sum\limits_{\beta =1}^{s}\sum\limits_{k=1}^{r_{\alpha
}}\sum\limits_{j,l=1}^{r_{\beta }}\int_{\left( \mathbb{R}^{3}\right)
^{3}}\left( \frac{h_{\alpha ,k}^{\prime }}{\left( M_{\alpha ,k}^{\prime
}\right) ^{1/2}}+\frac{h_{\beta ,l\ast }^{\prime }}{\left( M_{\beta ,l\ast
}^{\prime }\right) ^{1/2}}-\frac{h_{\beta ,j\ast }}{M_{\beta ,j\ast }^{1/2}}%
\right)   \notag \\
&&\times \left( \frac{M_{\beta ,j\ast }M_{\alpha ,k}^{\prime }M_{\beta
,l\ast }^{\prime }}{\varphi _{i}^{\alpha }\varphi _{j}^{\beta }\varphi
_{k}^{\alpha }\varphi _{l}^{\beta }}\right) ^{1/2}W_{\alpha \beta }(%
\boldsymbol{\xi },\boldsymbol{\xi }_{\ast },I_{i}^{\alpha },I_{j}^{\beta
}\left\vert \boldsymbol{\xi }^{\prime },\boldsymbol{\xi }_{\ast }^{\prime
},I_{k}^{\alpha },I_{l}^{\beta }\right. )d\boldsymbol{\xi }_{\ast }d%
\boldsymbol{\xi }^{\prime }d\boldsymbol{\xi }_{\ast }^{\prime }  \label{dec1}
\end{eqnarray}%
for all $i\in \left\{ 1,...,r_{\alpha }\right\} $ and $\alpha \in \left\{
1,...,s\right\} $, while the components of the quadratic term $\Gamma
=\left( \Gamma _{1},...,\Gamma _{s}\right) $, with $\Gamma _{\alpha }=\left(
\Gamma _{\alpha ,1},...,\Gamma _{\alpha ,r_{\alpha }}\right) $, are given by%
\begin{equation}
\Gamma _{\alpha ,i}\left( h,h\right) =M_{\alpha ,i}^{-1/2}Q_{i}^{\alpha }(%
\mathcal{M}^{1/2}h,\mathcal{M}^{1/2}h)  \label{nl1}
\end{equation}%
for all $i\in \left\{ 1,...,r_{\alpha }\right\} $ and $\alpha \in \left\{
1,...,s\right\} $.

The multiplication operator $\Lambda $ defined by 
\begin{equation*}
\Lambda (f)=\nu f\text{, where }\nu =\mathrm{diag}\left( \nu _{1,1},...,\nu
_{1,r_{1}},...,\nu _{s,1},...,\nu _{s,r_{s}}\right) \text{,}
\end{equation*}%
is a closed, densely defined, self-adjoint operator on $\left( L^{2}\left( d%
\boldsymbol{\xi }\right) \right) ^{r}$. It is Fredholm, as well, if and only
if $\Lambda $ is coercive.

The following lemma follows immediately by Lemma $\ref{L0}$.

\begin{lemma}
\label{L1}For any $\left( \alpha ,\beta ,i,j,k,l\right) \in \Omega $ the
measure 
\begin{equation*}
d\widetilde{A}_{ij,kl}^{\alpha \beta }=\left( \frac{M_{\alpha ,i}M_{\beta
,j\ast }M_{\alpha ,k}^{\prime }M_{\beta ,l\ast }^{\prime }}{\varphi
_{i}^{\alpha }\varphi _{j}^{\beta }\varphi _{k}^{\alpha }\varphi _{l}^{\beta
}}\right) ^{1/2}dA_{ij,kl}^{\alpha \beta }
\end{equation*}%
is invariant under the (ordered) interchange $\left( \ref{tr}\right) $ of
variables, while%
\begin{equation*}
d\widetilde{A}_{ij,kl}^{\alpha \beta }+d\widetilde{A}_{ji,lk}^{\beta \alpha }
\end{equation*}%
is invariant under the (ordered) interchange $\left( \ref{tr1}\right) $ of
variables
\end{lemma}

The weak form of the linearized collision operator $\mathcal{L}$ reads

\begin{eqnarray*}
\left( \mathcal{L}h,g\right)  &=&\sum_{\alpha ,\beta
=1}^{s}\sum\limits_{i,k=1}^{r_{\alpha }}\sum\limits_{j,l=1}^{r_{\beta
}}\int_{\left( \mathbb{R}^{3}\right) ^{4}}\Delta _{ij,kl}^{\alpha \beta
}\left( \mathcal{M}^{-1/2}h\right) \frac{g_{\alpha ,i}}{M_{\alpha ,i}^{1/2}}%
\,d\widetilde{A}_{ij,kl}^{\alpha \beta } \\
&=&\sum_{\alpha ,\beta =1}^{s}\sum\limits_{i,k=1}^{r_{\alpha
}}\sum\limits_{j,l=1}^{r_{\beta }}\int_{\left( \mathbb{R}^{3}\right)
^{4}}\Delta _{ij,kl}^{\alpha \beta }\left( \mathcal{M}^{-1/2}h\right) \frac{%
g_{\beta ,j\ast }}{M_{\beta ,j\ast }^{1/2}}\,d\widetilde{A}_{ij,kl}^{\alpha
\beta } \\
&=&-\sum_{\alpha ,\beta =1}^{s}\sum\limits_{i,k=1}^{r_{\alpha
}}\sum\limits_{j,l=1}^{r_{\beta }}\int_{\left( \mathbb{R}^{3}\right)
^{4}}\Delta _{ij,kl}^{\alpha \beta }\left( \mathcal{M}^{-1/2}h\right) \frac{%
g_{\alpha ,k}^{\prime }}{\left( M_{\alpha ,k}^{\prime }\right) ^{1/2}}\,d%
\widetilde{A}_{ij,kl}^{\alpha \beta } \\
&=&-\sum_{\alpha ,\beta =1}^{s}\sum\limits_{i,k=1}^{r_{\alpha
}}\sum\limits_{j,l=1}^{r_{\beta }}\int_{\left( \mathbb{R}^{3}\right)
^{4}}\Delta _{ij,kl}^{\alpha \beta }\left( \mathcal{M}^{-1/2}h\right) \frac{%
g_{\beta ,l\ast }^{\prime }}{\left( M_{\beta ,l\ast }^{\prime }\right) ^{1/2}%
}\,d\widetilde{A}_{ij,kl}^{\alpha \beta }\text{,}
\end{eqnarray*}%
for any function $g=\left( g_{1},...,g_{s}\right) $, with $g_{\alpha
}=\left( g_{\alpha ,1},...,g_{\alpha ,r_{\alpha }}\right) $, such that the
first integrals are defined for all $\left( \alpha ,\beta ,i,j,k,l\right)
\in \Omega $, while the following equalities are obtained by applying Lemma $%
\ref{L1}$. We have the following lemma.

\begin{lemma}
\label{L2}Let $g=\left( g_{1},...,g_{s}\right) $, with $g_{\alpha }=\left(
g_{\alpha ,1},...,g_{\alpha ,r_{\alpha }}\right) $, be such that%
\begin{equation*}
\int_{\left( \mathbb{R}^{3}\right) ^{4}}\Delta _{ij,kl}^{\alpha \beta
}\left( \mathcal{M}^{-1/2}h\right) \frac{g_{\alpha ,i}}{M_{\alpha ,i}^{1/2}}%
\,d\widetilde{A}_{ij,kl}^{\alpha \beta }
\end{equation*}%
is defined for any $\left( \alpha ,\beta ,i,j,k,l\right) \in \Omega $. Then%
\begin{equation*}
\left( \mathcal{L}h,g\right) =\frac{1}{4}\sum_{\alpha ,\beta
=1}^{s}\sum\limits_{i,k=1}^{r_{\alpha }}\sum\limits_{j,l=1}^{r_{\beta
}}\int_{\left( \mathbb{R}^{3}\right) ^{4}}\Delta _{ij,kl}^{\alpha \beta
}\left( \mathcal{M}^{-1/2}h\right) \,\Delta _{ij,kl}^{\alpha \beta }\left( 
\mathcal{M}^{-1/2}g\right) \,d\widetilde{A}_{ij,kl}^{\alpha \beta }.
\end{equation*}
\end{lemma}

\begin{proposition}
\label{Prop1}The linearized collision operator is symmetric and nonnegative,%
\begin{equation*}
\left( \mathcal{L}h,g\right) =\left( h,\mathcal{L}g\right) \text{ and }%
\left( \mathcal{L}h,h\right) \geq 0\text{,}
\end{equation*}%
and\ the kernel of $\mathcal{L}$, $\ker \mathcal{L}$, is generated by%
\begin{equation*}
\left\{ \mathcal{M}^{1/2}e_{1},...,\mathcal{M}^{1/2}e_{s},\mathcal{M}%
^{1/2}m\xi _{x},\mathcal{M}^{1/2}m\xi _{y},\mathcal{M}^{1/2}m\xi _{z},%
\mathcal{M}^{1/2}\left( m\left\vert \boldsymbol{\xi }\right\vert
^{2}+2I\right) \right\} \text{,}
\end{equation*}%
where $m=\sum_{\alpha =1}^{s}m_{\alpha }e_{\alpha }$ and $\mathcal{M}=%
\mathrm{diag}\left(
M_{1,1},...,M_{1,r_{1}},...,M_{s,1},...,M_{s,r_{s}}\right) $.
\end{proposition}

\begin{proof}
By Lemma $\ref{L2}$, it is immediate that $\left( \mathcal{L}h,g\right)
=\left( h,\mathcal{L}g\right) $, and%
\begin{equation*}
\left( \mathcal{L}h,h\right) =\frac{1}{4}\sum_{\alpha ,\beta
=1}^{s}\sum\limits_{i,k=1}^{r_{\alpha }}\sum\limits_{j,l=1}^{r_{\beta
}}\int_{\left( \mathbb{R}^{3}\right) ^{4}}\left( \Delta _{ij,kl}^{\alpha
\beta }\left( \mathcal{M}^{-1/2}h\right) \right) ^{2}d\widetilde{A}%
_{ij,kl}^{\alpha \beta }\geq 0.
\end{equation*}%
Furthermore, $h\in \ker \mathcal{L}$ if and only if $\left( \mathcal{L}%
h,h\right) =0$, which will be fulfilled\ if and only if for all $\left(
\alpha ,\beta ,i,j,k,l\right) \in \Omega $%
\begin{equation*}
\Delta _{ij,kl}^{\alpha \beta }\left( \mathcal{M}^{-1/2}h\right) W_{\alpha
\beta }(\boldsymbol{\xi },\boldsymbol{\xi }_{\ast },I_{i}^{\alpha
},I_{j}^{\beta }\left\vert \boldsymbol{\xi }^{\prime },\boldsymbol{\xi }%
_{\ast }^{\prime },I_{k}^{\alpha },I_{l}^{\beta }\right. )=0\text{ a.e.,}
\end{equation*}%
i.e. if and only if $\mathcal{M}^{-1/2}h$ is a collision invariant. The last
part of the lemma now follows by Proposition $\ref{P2}.$
\end{proof}

\begin{remark}
A property of the nonlinear term, although of no relevance to the studies
here, is that it is orthogonal to the kernel of $\mathcal{L}$, i.e. $\Gamma
\left( h,h\right) \in \left( \ker \mathcal{L}\right) ^{\perp _{\mathcal{%
\mathfrak{h}}}}$.

This follows, since any element in $\ker \mathcal{L}$ is of the form $%
\mathcal{M}^{1/2}g$ for some collision invariant $g$, while for any
collision invariant $g$%
\begin{eqnarray*}
\left( \Gamma \left( h,h\right) ,\mathcal{M}^{1/2}g\right) &=&\left( 
\mathcal{M}^{-1/2}Q(\mathcal{M}^{1/2}h,\mathcal{M}^{1/2}h),\mathcal{M}%
^{1/2}g\right) \\
&=&\left( Q(\mathcal{M}^{1/2}h,\mathcal{M}^{1/2}h),g\right) =0\text{.}
\end{eqnarray*}
\end{remark}

\section{Main Results\label{S3}}

This section is devoted to the main results, concerning a compactness
property in Theorem \ref{Thm1} and bounds of the collision frequencies in
Theorem \ref{Thm2}.

Assume that for some positive number $\gamma $, such that $0<\gamma <1$,
there for all $\left( \alpha ,\beta ,i,j,k,l\right) \in \Omega $ is a bound 
\begin{eqnarray}
0 &\leq &\sigma _{ij,kl}^{\alpha \beta }\left( \left\vert \mathbf{g}%
\right\vert ,\cos \theta \right) \leq \frac{C}{\left\vert \mathbf{g}%
\right\vert ^{2}}\left( \Psi _{ij,kl}^{\alpha \beta }+\left( \Psi
_{ij,kl}^{\alpha \beta }\right) ^{\gamma /2}\right) \text{, where }  \notag
\\
&&\Psi _{ij,kl}^{\alpha \beta }=\left\vert \mathbf{g}\right\vert \sqrt{%
\left\vert \mathbf{g}\right\vert ^{2}-2\frac{m_{\alpha }+m_{\beta }}{%
m_{\alpha }m_{\beta }}\Delta I_{kl,ij}^{\alpha \beta }}\text{,}  \label{est1}
\end{eqnarray}%
for $\left\vert \mathbf{g}\right\vert ^{2}>2\left(m_{\alpha }+m_{\beta }\right) %
 \Delta I_{kl,ij}^{\alpha \beta }/\left(m_{\alpha }m_{\beta }\right)$, on the scattering
cross sections $\sigma _{ij,kl}^{\alpha \beta }$. Note that assumption $%
\left( \ref{est1}\right) $ reduces to Grad's assumption \cite{Gr-63} in the
case of vanishing internal energy gap, cf. \cite{Be-21a} for the case of
monatomic multicomponent mixtures. 

The following result may be obtained.

\begin{theorem}
\label{Thm1}Assume that for all $\left( \alpha ,\beta ,i,j,k,l\right) \in
\Omega $ the scattering cross sections $\sigma _{ij,kl}^{\alpha \beta }$
satisfy the bound $\left( \ref{est1}\right) $ for some positive number $%
\gamma $, such that $0<\gamma <1$. Then the operator $K=\left(
K_{1,1},...,K_{\alpha ,r_{1}},...,K_{s,1},...,K_{s,r_{s}}\right) $, with the
components $K_{\alpha ,i}$ given by expressions $\left( \ref{dec1}\right) $
is a self-adjoint compact operator on $\left( L^{2}\left( d\boldsymbol{\xi }%
\right) \right) ^{r}$.
\end{theorem}

Theorem \ref{Thm1} will be proven in Section $\ref{PT1}$.

\begin{corollary}
\label{Cor1}The linearized collision operator $\mathcal{L}$, with scattering
cross sections satisfying $\left( \ref{est1}\right) $, is a closed, densely
defined, self-adjoint operator on $\left( L^{2}\left( d\boldsymbol{\xi }%
\right) \right) ^{r}$.
\end{corollary}

\begin{proof}
By Theorem \ref{Thm1}, the linear operator $\mathcal{L}=\Lambda -K$ is
closed as the sum of a closed and a bounded operator, and densely defined,
since the domains of the linear operators $\mathcal{L}$ and $\Lambda $ are
equal; $D(\mathcal{L})=D(\Lambda )$. Furthermore, it is a self-adjoint
operator, since the set of self-adjoint operators is closed under addition
of bounded self-adjoint operators, see Theorem 4.3 of Chapter V in \cite%
{Kato}. 
\end{proof}

Now consider the scattering cross sections - cf. hard sphere models - 
\begin{eqnarray}
\sigma _{ij,kl}^{\alpha \beta } &=&C_{\alpha \beta }\dfrac{\sqrt{\left\vert 
\mathbf{g}\right\vert ^{2}-2\widetilde{\Delta }I_{kl,ij}^{\alpha \beta }}}{%
\left\vert \mathbf{g}\right\vert \varphi _{i}^{\alpha }\varphi _{j}^{\beta }}%
\text{ if }\left\vert \mathbf{g}\right\vert ^{2}>2\widetilde{\Delta }%
I_{kl,ij}^{\alpha \beta }\text{, where }  \notag \\
\widetilde{\Delta }I_{kl,ij}^{\alpha \beta } &=&\frac{m_{\alpha }+m_{\beta }%
}{m_{\alpha }m_{\beta }}\Delta I_{kl,ij}^{\alpha \beta }\text{,}  \label{e1}
\end{eqnarray}%
for some positive constants $C_{\alpha \beta }>0$ for all $\left( \alpha
,\beta ,i,j,k,l\right) \in \Omega $. Note that assumption $\left( \ref{e1}%
\right) $ reduces to the hard sphere model for monatomic multicomponent
mixtures \cite{Be-21a} in the case of vanishing internal energy gap and unit
weigths $\varphi _{1}^{1}=...=\varphi _{r_{1}}^{1}=...=\varphi
_{1}^{s}=...=\varphi _{r_{s}}^{s}=1$.

In fact, it would be enough with the bounds 
\begin{equation}
C_{-}\dfrac{\sqrt{\left\vert \mathbf{g}\right\vert ^{2}-2\widetilde{\Delta }%
I_{kl,ij}^{\alpha \beta }}}{\left\vert \mathbf{g}\right\vert \varphi
_{i}^{\alpha }\varphi _{j}^{\beta }}\leq \sigma _{ij,kl}^{\alpha \beta }\leq
C_{+}\dfrac{\sqrt{\left\vert \mathbf{g}\right\vert ^{2}-2\widetilde{\Delta }%
I_{kl,ij}^{\alpha \beta }}}{\left\vert \mathbf{g}\right\vert \varphi
_{i}^{\alpha }\varphi _{j}^{\beta }}\text{ if }\left\vert \mathbf{g}%
\right\vert ^{2}>2\widetilde{\Delta }I_{kl,ij}^{\alpha \beta }\text{,}
\label{ie1}
\end{equation}%
for some positive constants $C_{\pm }>0$, on the scattering cross sections.

\begin{theorem}
\label{Thm2} The linearized collision operator $\mathcal{L}$, with
scattering cross sections $\left( \ref{e1}\right) $ (or $\left( \ref{ie1}%
\right) $), can be split into a positive multiplication operator $\Lambda $,
defined by $\Lambda f=\nu f$, where $\nu =\nu (\left\vert \boldsymbol{\xi }%
\right\vert )=\mathrm{diag}\left( \nu _{1,1},...,\nu _{1,r_{1}},...,\nu
_{s,1},...,\nu _{s,r_{s}}\right) $, minus a compact operator $K$ on $\left(
L^{2}\left( d\boldsymbol{\xi }\right) \right) ^{r}$%
\begin{equation}
\mathcal{L}=\Lambda -K,  \label{dec3}
\end{equation}%
where there exist positive numbers $\nu _{-}$ and $\nu _{+}$, with $0<\nu
_{-}<\nu _{+}$, such that for any $\left( \alpha ,i\right) \in \left\{
1,...,s\right\} \times \left\{ 1,...,r\right\} $ 
\begin{equation}
\nu _{-}\left( 1+\left\vert \boldsymbol{\xi }\right\vert \right) \leq \nu
_{\alpha ,i}(\left\vert \boldsymbol{\xi }\right\vert )\leq \nu _{+}\left(
1+\left\vert \boldsymbol{\xi }\right\vert \right) \text{ for all }%
\boldsymbol{\xi }\in \mathbb{R}^{3}\text{.}  \label{ine1}
\end{equation}
\end{theorem}

The decomposition $\left( \ref{dec3}\right) $ follows by the decomposition $%
\left( \ref{dec2}\right) ,\left( \ref{dec1}\right) $ and Theorem \ref{Thm1},
while the bounds $\left( \ref{ine1}\right) $ are proven in Section $\ref{PT2}
$.

\begin{corollary}
\label{Cor2}The linearized collision operator $\mathcal{L}$, with scattering
cross sections $\left( \ref{e1}\right) $ (or $\left( \ref{ie1}\right) $), is
a Fredholm operator, with domain%
\begin{equation*}
D(\mathcal{L})=\left( L^{2}\left( \left( 1+\left\vert \boldsymbol{\xi }%
\right\vert \right) d\boldsymbol{\xi }\right) \right) ^{r}\text{.}
\end{equation*}
\end{corollary}

\begin{proof}
By Theorem \ref{Thm2} the multiplication operator $\Lambda $ is coercive
and, hence, a Fredholm operator. The set of Fredholm operators is closed
under addition of compact operators, see Theorem 5.26 of Chapter IV in \cite%
{Kato} and its proof, so, by Theorem \ref{Thm2}, $\mathcal{L}$ is a Fredholm
operator. 

Moreover, by Theorem \ref{Thm2}, $D(\mathcal{L})=D(\Lambda )=\left(
L^{2}\left( \left( 1+\left\vert \boldsymbol{\xi }\right\vert \right) d%
\boldsymbol{\xi }\right) \right) ^{r}$.
\end{proof}

\begin{corollary}
\label{Cor3}For the linearized collision operator $\mathcal{L}$, with
scattering cross sections $\left( \ref{e1}\right) $ (or $\left( \ref{ie1}%
\right) $), there exists a positive number $\lambda $, $0<\lambda <1$, such
that 
\begin{equation*}
\left( h,\mathcal{L}h\right) \geq \lambda \left( h,\nu (\left\vert 
\boldsymbol{\xi }\right\vert )h\right) \geq \lambda \nu _{-}\left( h,\left(
1+\left\vert \boldsymbol{\xi }\right\vert \right) h\right)
\end{equation*}%
for any $h\in \left( L^{2}(\left( 1+\left\vert \boldsymbol{\xi }\right\vert
\right) d\boldsymbol{\xi })\right) ^{r}\cap \mathrm{\mathrm{Im}}\mathcal{L}$.
\end{corollary}

\begin{proof}
Let $h\in \left( L^{2}(\left( 1+\left\vert \boldsymbol{\xi }\right\vert
\right) d\boldsymbol{\xi })\right) ^{r}\cap \left( \mathrm{ker}\mathcal{L}%
\right) ^{\perp }=\left( L^{2}(\left( 1+\left\vert \boldsymbol{\xi }%
\right\vert \right) d\boldsymbol{\xi })\right) ^{r}\cap \mathrm{\mathrm{Im}}%
\mathcal{L}$. As a Fredholm operator, $\mathcal{L}$ is closed with a closed
range, and as a compact operator, $K$ is bounded, and so there are positive
constants $\nu _{0}>0$ and $c_{K}>0$, such that%
\begin{equation*}
(h,\mathcal{L}h)\geq \nu _{0}(h,h)\text{ and }(h,Kh)\leq c_{K}(h,h).
\end{equation*}%
Let $\lambda =\dfrac{\nu _{0}}{\nu _{0}+c_{K}}$. Then%
\begin{eqnarray*}
(h,\mathcal{L}h) &=&(1-\lambda )(h,\mathcal{L}h)+\lambda (h,(\nu (\left\vert 
\boldsymbol{\xi }\right\vert )-K)h) \\
&\geq &(1-\lambda )\nu _{0}(h,h)+\lambda (h,\nu (\left\vert \boldsymbol{\xi }%
\right\vert )h)-\lambda c_{K}(h,h) \\
&=&(\nu _{0}-\lambda (\nu _{0}+c_{K}))(h,h)+\lambda (h,\nu (\left\vert 
\boldsymbol{\xi }\right\vert )h)=\lambda (h,\nu (\left\vert \boldsymbol{\xi }%
\right\vert )h)\text{.}
\end{eqnarray*}
\end{proof}

\begin{remark}
By Proposition $\ref{Prop1}$ and Corollary $\ref{Cor1}-\ref{Cor3}$, for hard
sphere like models the linearized operator $\mathcal{L}$ fulfills the
properties assumed on the linear operators in \cite{Be-21}, and hence, the
results therein can be applied for hard sphere like models.
\end{remark}

\section{Compactness \label{PT1}}

This section concerns the proof of Theorem \ref{Thm1}.

Note that in the proof the kernels are rewritten in such a way that $%
\boldsymbol{\xi }_{\ast }$ - and not $\boldsymbol{\xi }^{\prime }$ and $%
\boldsymbol{\xi }_{\ast }^{\prime }$ - always will be an argument of the
distribution functions. As for single species, either $\boldsymbol{\xi }%
_{\ast }$ is an argument in the loss term (like $\boldsymbol{\xi }$) or in
the gain term (unlike $\boldsymbol{\xi }$) of the collision operator.
However,\ in the latter case, unlike for single species, for mixtures one
have to differ between two different cases (considering interspecies
collision operators); either $\boldsymbol{\xi }_{\ast }$ is the velocity of
particles of the same species as the particles with velocity $\boldsymbol{%
\xi }$, or not. The kernels of the terms from the loss part of the collision
operator will be shown\ to be Hilbert-Schmidt in a quite direct way. Some of
the terms - for which $\boldsymbol{\xi }_{\ast }$ is the velocity of
particles of the same species as the particles with velocity $\boldsymbol{%
\xi }$ - of the gain parts of the collision operators will be shown to be
uniform limits of Hilbert-Schmidt integral operators, i.e. approximately
Hilbert-Schmidt integral operators in the sense of Lemma \ref{LGD}. By
applying the following lemma, Lemma \ref{L3}, (for disparate masses), which
is a generalization of corresponding lemma for monatomic mixtures by Boudin
et al in \cite{BGPS-13}, see also \cite{Be-21a}, it will be shown that the
kernels of the remaining terms - for which $\boldsymbol{\xi }_{\ast }$ is
the velocity of particles of a species different to the species of the
particles with velocity $\boldsymbol{\xi }$ - from the gain parts of the
collision operators, are Hilbert-Schmidt.

Denote, for any (non-zero) natural number $N$,%
\begin{equation*}
\mathfrak{h}_{N}:=\left\{ (\boldsymbol{\xi },\boldsymbol{\xi }_{\ast })\in
\left( \mathbb{R}^{3}\right) ^{2}:\left\vert \boldsymbol{\xi }-\boldsymbol{%
\xi }_{\ast }\right\vert \geq \frac{1}{N}\text{; }\left\vert \boldsymbol{\xi 
}\right\vert \leq N\right\}
\end{equation*}%
and%
\begin{equation*}
b^{(N)}=b^{(N)}(\boldsymbol{\xi },\boldsymbol{\xi }_{\ast }):=b(\boldsymbol{%
\xi },\boldsymbol{\xi }_{\ast })\mathbf{1}_{\mathfrak{h}_{N}}\text{.}
\end{equation*}%
Then we have the following lemma from \cite{Glassey}, that will be of
practical use for us to obtain compactness in this section.

\begin{lemma}
\label{LGD} (Glassey \cite[Lemma 3.5.1]{Glassey}, Drange \cite{Dr-75})

Assume that $b(\boldsymbol{\xi },\boldsymbol{\xi }_{\ast })\geq 0$ and let $%
Tf\left( \boldsymbol{\xi }\right) =\int_{\mathbb{R}^{3}}b(\boldsymbol{\xi },%
\boldsymbol{\xi }_{\ast })f\left( \boldsymbol{\xi }_{\ast }\right) \,d%
\boldsymbol{\xi }_{\ast }$.

Then $T$ is compact on $L^{2}\left( d\boldsymbol{\xi \,}\right) $ if

(i) $\int_{\mathbb{R}^{3}}b(\boldsymbol{\xi },\boldsymbol{\xi }_{\ast })\,d%
\boldsymbol{\xi }$ is bounded in $\boldsymbol{\xi }_{\ast }$;

(ii) $b^{(N)}\in L^{2}\left( d\boldsymbol{\xi \,}d\boldsymbol{\xi }_{\ast
}\right) $ for any (non-zero) natural number $N$;

(iii) $\underset{\boldsymbol{\xi }\in \mathbb{R}^{3}}{\sup }\int_{\mathbb{R}%
^{3}}b(\boldsymbol{\xi },\boldsymbol{\xi }_{\ast })-b^{(N)}(\boldsymbol{\xi }%
,\boldsymbol{\xi }_{\ast })\,d\boldsymbol{\xi }_{\ast }\rightarrow 0$ as $%
N\rightarrow \infty $.
\end{lemma}

Then the operator $T$ is the uniform limit of Hilbert-Schmidt integral
operators, and we say that the kernel $b(\boldsymbol{\xi },\boldsymbol{\xi }%
_{\ast })$ is approximately Hilbert-Schmidt, while $T$ is an approximately
Hilbert-Schmidt integral operator. The reader is referred to Lemma 3.5.1 in 
\cite{Glassey} for a proof.

\begin{lemma}
\label{L3} \cite{BGPS-13} For $\left( \alpha ,\beta ,i,j,k,l\right) \in
\Omega $, assume that $m_{\alpha }\neq m_{\beta }$, 
\begin{equation}
\left\{ 
\begin{array}{l}
\boldsymbol{\xi }^{\prime }=\boldsymbol{\xi }-\left\vert \boldsymbol{\xi }-%
\boldsymbol{\xi }^{\prime }\right\vert \boldsymbol{\eta } \\ 
\boldsymbol{\xi }_{\ast }^{\prime }=\boldsymbol{\xi }_{\ast }-\dfrac{%
m_{\alpha }}{m_{\beta }}\left\vert \boldsymbol{\xi }-\boldsymbol{\xi }%
^{\prime }\right\vert \boldsymbol{\eta }%
\end{array}%
\right. \text{, where }\boldsymbol{\eta }\in \mathbb{S}^{2}\text{,}
\label{vrel2}
\end{equation}%
and%
\begin{eqnarray}
m_{\alpha }\frac{\left\vert \boldsymbol{\xi }\right\vert ^{2}}{2}+m_{\beta }%
\frac{\left\vert \boldsymbol{\xi }_{\ast }^{\prime }\right\vert ^{2}}{2}
&=&m_{\alpha }\frac{\left\vert \boldsymbol{\xi }^{\prime }\right\vert ^{2}}{2%
}+m_{\beta }\frac{\left\vert \boldsymbol{\xi }_{\ast }\right\vert ^{2}}{2}%
+\Delta I_{kj,il}^{\alpha \beta }\text{, with}  \notag \\
\Delta I_{kj,il}^{\alpha \beta } &=&I_{k}^{\alpha }+I_{j}^{\beta
}-I_{i}^{\alpha }-I_{l}^{\beta }\text{.}  \label{vrel3}
\end{eqnarray}%
Then there exists a positive number $\rho $, $0<\rho <1$, such that%
\begin{eqnarray*}
m_{\alpha }\left\vert \boldsymbol{\xi }^{\prime }\right\vert ^{2}+m_{\beta
}\left\vert \boldsymbol{\xi }_{\ast }^{\prime }\right\vert ^{2} &\geq &\rho
\left( m_{\alpha }\left\vert \boldsymbol{\xi }\right\vert ^{2}+m_{\beta
}\left\vert \boldsymbol{\xi }_{\ast }\right\vert ^{2}\right) +\left( 1+\rho
\right) \frac{m_{\alpha }-m_{\beta }}{m_{\alpha }+m_{\beta }}\Delta
I_{kj,il}^{\alpha \beta } \\
&\geq &\rho \left( m_{\alpha }\left\vert \boldsymbol{\xi }\right\vert
^{2}+m_{\beta }\left\vert \boldsymbol{\xi }_{\ast }\right\vert ^{2}\right)
-2\left\vert \Delta I_{kj,il}^{\alpha \beta }\right\vert \text{.}
\end{eqnarray*}
\end{lemma}

A proof of Lemma \ref{L3}, based on the proof of the corresponding lemma 
\cite{BGPS-13} for monatomic mixtures in \cite{Be-21a}, is accounted for in
the appendix. The proof is constructive, in the way that an explicit value
of such a number $\rho $, namely%
\begin{equation*}
\rho =\left( \frac{\sqrt{m_{\alpha }}-\sqrt{m_{\beta }}}{\sqrt{m_{\alpha }}+%
\sqrt{m_{\beta }}}\right) ^{2}\text{,}
\end{equation*}%
is produced in the proof.

Now we turn to the proof of Theorem \ref{Thm1}. Note that throughout the
proof $C$ will denote a generic positive constant.

\begin{proof}
For $i\in \left\{ 1,...,r_{\alpha }\right\} $ and $\alpha \in \left\{
1,...,s\right\} $, rewrite expression $\left( \ref{dec1}\right) $ as%
\begin{eqnarray*}
K_{\alpha ,i} &=&\left( M_{\alpha ,i}\right) ^{-1/2}\sum\limits_{\beta
=1}^{s}\sum\limits_{k=1}^{r_{\alpha }}\sum\limits_{j,l=1}^{r_{\beta
}}\int_{\left( \mathbb{R}^{3}\right) ^{3}}w_{\alpha \beta }(\boldsymbol{\xi }%
,\boldsymbol{\xi }_{\ast },I_{i}^{\alpha },I_{j}^{\beta }\left\vert 
\boldsymbol{\xi }^{\prime },\boldsymbol{\xi }_{\ast }^{\prime
},I_{k}^{\alpha },I_{l}^{\beta }\right. ) \\
&&\times \left( \frac{h_{\alpha ,k}^{\prime }}{\left( M_{\alpha ,k}^{\prime
}\right) ^{1/2}}+\frac{h_{\beta ,l\ast }^{\prime }}{\left( M_{\beta ,l\ast
}^{\prime }\right) ^{1/2}}-\frac{h_{\beta ,j\ast }}{M_{\beta ,j\ast }^{1/2}}%
\right) \,d\boldsymbol{\xi }_{\ast }d\boldsymbol{\xi }^{\prime }d\boldsymbol{%
\xi }_{\ast }^{\prime }\text{,}
\end{eqnarray*}%
with 
\begin{eqnarray*}
&&w_{\alpha \beta }(\boldsymbol{\xi },\boldsymbol{\xi }_{\ast
},I_{i}^{\alpha },I_{j}^{\beta }\left\vert \boldsymbol{\xi }^{\prime },%
\boldsymbol{\xi }_{\ast }^{\prime },I_{k}^{\alpha },I_{l}^{\beta }\right. )
\\
&=&\left( \frac{M_{\alpha ,i}M_{\beta ,j\ast }M_{\alpha ,k}^{\prime
}M_{\beta ,l\ast }^{\prime }}{\varphi _{i}^{\alpha }\varphi _{j}^{\beta
}\varphi _{k}^{\alpha }\varphi _{l}^{\beta }}\right) ^{1/2}W_{\alpha \beta }(%
\boldsymbol{\xi },\boldsymbol{\xi }_{\ast },I_{i}^{\alpha },I_{j}^{\beta
}\left\vert \boldsymbol{\xi }^{\prime },\boldsymbol{\xi }_{\ast }^{\prime
},I_{k}^{\alpha },I_{l}^{\beta }\right. )\text{.}
\end{eqnarray*}%
Due to relations $\left( \ref{rel1}\right) $, the relations%
\begin{eqnarray}
w_{\alpha \beta }(\boldsymbol{\xi },\boldsymbol{\xi }_{\ast },I_{i}^{\alpha
},I_{j}^{\beta }\left\vert \boldsymbol{\xi }^{\prime },\boldsymbol{\xi }%
_{\ast }^{\prime },I_{k}^{\alpha },I_{l}^{\beta }\right. ) &=&w_{\beta
\alpha }(\boldsymbol{\xi }_{\ast },\boldsymbol{\xi },I_{j}^{\beta
},I_{i}^{\alpha }\left\vert \boldsymbol{\xi }_{\ast }^{\prime },\boldsymbol{%
\xi }^{\prime },I_{l}^{\beta },I_{k}^{\alpha }\right. )  \notag \\
w_{\alpha \beta }(\boldsymbol{\xi },\boldsymbol{\xi }_{\ast },I_{i}^{\alpha
},I_{j}^{\beta }\left\vert \boldsymbol{\xi }^{\prime },\boldsymbol{\xi }%
_{\ast }^{\prime },I_{k}^{\alpha },I_{l}^{\beta }\right. ) &=&w_{\alpha
\beta }(\boldsymbol{\xi }^{\prime },\boldsymbol{\xi }_{\ast }^{\prime
},I_{k}^{\alpha },I_{l}^{\beta }\left\vert \boldsymbol{\xi },\boldsymbol{\xi 
}_{\ast },I_{i}^{\alpha },I_{j}^{\beta }\right. )  \notag \\
w_{\alpha \alpha }(\boldsymbol{\xi },\boldsymbol{\xi }_{\ast },I_{i}^{\alpha
},I_{j}^{\beta }\left\vert \boldsymbol{\xi }^{\prime },\boldsymbol{\xi }%
_{\ast }^{\prime },I_{k}^{\alpha },I_{l}^{\beta }\right. ) &=&w_{\alpha
\alpha }(\boldsymbol{\xi },\boldsymbol{\xi }_{\ast },I_{i}^{\alpha
},I_{j}^{\alpha }\left\vert \boldsymbol{\xi }_{\ast }^{\prime },\boldsymbol{%
\xi }^{\prime },I_{l}^{\alpha },I_{k}^{\alpha }\right. )  \label{rel2}
\end{eqnarray}%
are satisfied for $\left( \alpha ,\beta ,i,j,k,l\right) \in \Omega $.

By renaming $\left\{ \boldsymbol{\xi }_{\ast },j\right\} \leftrightarrows
\left\{ \boldsymbol{\xi }_{\ast }^{\prime },l\right\} $, for $i\in \left\{
1,...,r_{\alpha }\right\} $ and $\left\{ \alpha ,\beta \right\} \subseteq
\left\{ 1,...,s\right\} $%
\begin{eqnarray*}
&&\sum\limits_{k=1}^{r_{\alpha }}\sum\limits_{j,l=1}^{r_{\beta
}}\int_{\left( \mathbb{R}^{3}\right) ^{3}}w_{\alpha \beta }(\boldsymbol{\xi }%
,\boldsymbol{\xi }_{\ast },I_{i}^{\alpha },I_{j}^{\beta }\left\vert 
\boldsymbol{\xi }^{\prime },\boldsymbol{\xi }_{\ast }^{\prime
},I_{k}^{\alpha },I_{l}^{\beta }\right. )\,\frac{h_{\beta ,l\ast }^{\prime }%
}{\left( M_{\beta ,l\ast }^{\prime }\right) ^{1/2}}\,d\boldsymbol{\xi }%
_{\ast }d\boldsymbol{\xi }^{\prime }d\boldsymbol{\xi }_{\ast }^{\prime } \\
&=&\sum\limits_{k=1}^{r_{\alpha }}\sum\limits_{j,l=1}^{r_{\beta
}}\int_{\left( \mathbb{R}^{3}\right) ^{3}}w_{\alpha \beta }(\boldsymbol{\xi }%
,\boldsymbol{\xi }_{\ast }^{\prime },I_{i}^{\alpha },I_{l}^{\beta
}\left\vert \boldsymbol{\xi }^{\prime },\boldsymbol{\xi }_{\ast
},I_{k}^{\alpha },I_{j}^{\beta }\right. )\,\frac{h_{\beta ,j\ast }}{M_{\beta
,j\ast }^{1/2}}\,\,d\boldsymbol{\xi }_{\ast }d\boldsymbol{\xi }^{\prime }d%
\boldsymbol{\xi }_{\ast }^{\prime }\text{.}
\end{eqnarray*}%
Moreover, by renaming $\left\{ \boldsymbol{\xi }_{\ast },j\right\}
\leftrightarrows \left\{ \boldsymbol{\xi }^{\prime },k\right\} $,%
\begin{eqnarray*}
&&\sum\limits_{k=1}^{r_{\alpha }}\sum\limits_{j,l=1}^{r_{\beta
}}\int_{\left( \mathbb{R}^{3}\right) ^{3}}w_{\alpha \beta }(\boldsymbol{\xi }%
,\boldsymbol{\xi }_{\ast },I_{i}^{\alpha },I_{j}^{\beta }\left\vert 
\boldsymbol{\xi }^{\prime },\boldsymbol{\xi }_{\ast }^{\prime
},I_{k}^{\alpha },I_{l}^{\beta }\right. )\,\frac{h_{\alpha ,k}^{\prime }}{%
\left( M_{\alpha ,k}^{\prime }\right) ^{1/2}}\,d\boldsymbol{\xi }_{\ast }d%
\boldsymbol{\xi }^{\prime }d\boldsymbol{\xi }_{\ast }^{\prime } \\
&=&\sum\limits_{j=1}^{r_{\alpha }}\sum\limits_{k,l=1}^{r_{\beta
}}\int_{\left( \mathbb{R}^{3}\right) ^{3}}w_{\alpha \beta }(\boldsymbol{\xi }%
,\boldsymbol{\xi }^{\prime },I_{i}^{\alpha },I_{k}^{\beta }\left\vert 
\boldsymbol{\xi }_{\ast },\boldsymbol{\xi }_{\ast }^{\prime },I_{j}^{\alpha
},I_{l}^{\beta }\right. )\,\frac{h_{\alpha ,j\ast }}{M_{\alpha ,j\ast }^{1/2}%
}\,d\boldsymbol{\xi }_{\ast }d\boldsymbol{\xi }^{\prime }d\boldsymbol{\xi }%
_{\ast }^{\prime }
\end{eqnarray*}%
for $i\in \left\{ 1,...,r_{\alpha }\right\} $ and $\left\{ \alpha ,\beta
\right\} \subseteq \left\{ 1,...,s\right\} $. It follows that  
\begin{eqnarray}
K_{\alpha ,i}\left( h\right)  &=&\sum\limits_{\beta =1}^{s}\int_{\mathbb{R}%
^{3}}k_{\alpha \beta ,i}\left( \boldsymbol{\xi },\boldsymbol{\xi }_{\ast
}\right) \,h_{\ast }\,d\boldsymbol{\xi }_{\ast }\text{, where }  \notag \\
k_{\alpha \beta ,i}h_{\ast } &=&\sum\limits_{j=1}^{r_{\alpha }}k_{\alpha
\beta ,ij}^{\left( \alpha \right) }h_{\alpha \ast
}+\sum\limits_{j=1}^{r_{\beta }}k_{\alpha \beta ,ij}^{\left( \beta \right)
}h_{\beta \ast }  \notag \\
&=&\sum\limits_{j=1}^{r_{\alpha }}k_{\alpha \beta ,ij}^{\left( \alpha
\right) }h_{\alpha \ast }+\sum\limits_{j=1}^{r_{\beta }}\left( k_{\alpha
\beta ,ij}^{\left( \beta ,2\right) }-k_{\alpha \beta ,ij}^{\left( \beta
,1\right) }\right) h_{\beta \ast }\text{, with}  \notag \\
k_{\alpha \beta ,ij}^{\left( \alpha \right) }(\boldsymbol{\xi },\boldsymbol{%
\xi }_{\ast }) &=&\sum\limits_{k,l=1}^{r_{\beta }}\int_{\left( \mathbb{R}%
^{3}\right) ^{2}}\frac{w_{\alpha \beta }(\boldsymbol{\xi },\boldsymbol{\xi }%
^{\prime },I_{i}^{\alpha },I_{k}^{\beta }\left\vert \boldsymbol{\xi }_{\ast
},\boldsymbol{\xi }_{\ast }^{\prime },I_{j}^{\alpha },I_{l}^{\beta }\right. )%
}{\left( M_{\alpha ,i}M_{\alpha ,j\ast }\right) ^{1/2}}\,d\boldsymbol{\xi }%
^{\prime }d\boldsymbol{\xi }_{\ast }^{\prime }\text{,}  \notag \\
k_{\alpha \beta ,ij}^{\left( \beta ,1\right) }(\boldsymbol{\xi },\boldsymbol{%
\xi }_{\ast }) &=&\sum\limits_{k=1}^{r_{\alpha }}\sum\limits_{l=1}^{r_{\beta
}}\int_{\left( \mathbb{R}^{3}\right) ^{2}}\frac{w_{\alpha \beta }(%
\boldsymbol{\xi },\boldsymbol{\xi }_{\ast },I_{i}^{\alpha },I_{j}^{\beta
}\left\vert \boldsymbol{\xi }^{\prime },\boldsymbol{\xi }_{\ast }^{\prime
},I_{k}^{\alpha },I_{l}^{\beta }\right. )}{\left( M_{\alpha ,i}M_{\beta
,j\ast }\right) ^{1/2}}\,d\boldsymbol{\xi }^{\prime }d\boldsymbol{\xi }%
_{\ast }^{\prime }\text{, and}  \notag \\
k_{\alpha \beta ,ij}^{\left( \beta ,2\right) }(\boldsymbol{\xi },\boldsymbol{%
\xi }_{\ast }) &=&\sum\limits_{k=1}^{r_{\alpha }}\sum\limits_{l=1}^{r_{\beta
}}\int_{\left( \mathbb{R}^{3}\right) ^{2}}\frac{w_{\alpha \beta }(%
\boldsymbol{\xi },\boldsymbol{\xi }_{\ast }^{\prime },I_{i}^{\alpha
},I_{l}^{\beta }\left\vert \boldsymbol{\xi }^{\prime },\boldsymbol{\xi }%
_{\ast },I_{k}^{\alpha },I_{j}^{\beta }\right. )}{\left( M_{\alpha
,i}M_{\beta ,j\ast }\right) ^{1/2}}\,d\boldsymbol{\xi }^{\prime }d%
\boldsymbol{\xi }_{\ast }^{\prime }  \label{k1}
\end{eqnarray}%
for $i\in \left\{ 1,...,r_{\alpha }\right\} $ and $\alpha \in \left\{
1,...,s\right\} $.

Next we obtain some symmetry relations that will help to yield
self-adjointness of the operator $K$ below. Indeed, by applying the second
relation in $\left( \ref{rel2}\right) $ and renaming $\left\{ \boldsymbol{%
\xi }^{\prime },k\right\} \leftrightarrows \left\{ \boldsymbol{\xi }_{\ast
}^{\prime },l\right\} $, 
\begin{eqnarray}
k_{\alpha \beta ,ij}^{\left( \alpha \right) }(\boldsymbol{\xi },\boldsymbol{%
\xi }_{\ast }) &=&\sum\limits_{k,l=1}^{r_{\beta }}\int_{\left( \mathbb{R}%
^{3}\right) ^{2}}\frac{w_{\alpha \beta }(\boldsymbol{\xi }_{\ast },%
\boldsymbol{\xi }_{\ast }^{\prime },I_{j}^{\alpha },I_{l}^{\beta }\left\vert 
\boldsymbol{\xi },\boldsymbol{\xi }^{\prime },I_{i}^{\alpha },I_{k}^{\beta
}\right. )}{\left( M_{\alpha ,i}M_{\alpha ,j\ast }\right) ^{1/2}}\,d%
\boldsymbol{\xi }^{\prime }d\boldsymbol{\xi }_{\ast }^{\prime }  \notag \\
&=&\sum\limits_{k,l=1}^{r_{\beta }}\int_{\left( \mathbb{R}^{3}\right) ^{2}}%
\frac{w_{\alpha \beta }(\boldsymbol{\xi }_{\ast },\boldsymbol{\xi }^{\prime
},I_{j}^{\alpha },I_{k}^{\beta }\left\vert \boldsymbol{\xi },\boldsymbol{\xi 
}_{\ast }^{\prime },I_{i}^{\alpha },I_{l}^{\beta }\right. )}{\left(
M_{\alpha ,i}M_{\alpha ,j\ast }\right) ^{1/2}}\,d\boldsymbol{\xi }^{\prime }d%
\boldsymbol{\xi }_{\ast }^{\prime }\,  \notag \\
&=&k_{\alpha \beta ,ji}^{\left( \alpha \right) }(\boldsymbol{\xi }_{\ast },%
\boldsymbol{\xi })  \label{sa1}
\end{eqnarray}%
for $\left\{ i,j\right\} \subseteq \left\{ 1,...,r_{\alpha }\right\} $ and $%
\left\{ \alpha ,\beta \right\} \subseteq \left\{ 1,...,s\right\} $. 

Moreover, for $\left( i,j\right) \in \left\{ 1,...,r_{\alpha }\right\}
\times \left\{ 1,...,r_{\beta }\right\} $ and $\left\{ \alpha ,\beta
\right\} \subseteq \left\{ 1,...,s\right\} $%
\begin{equation}
k_{\alpha \beta ,ij}^{\left( \beta \right) }(\boldsymbol{\xi },\boldsymbol{%
\xi }_{\ast })=k_{\beta \alpha ,ji}^{\left( \alpha ,1\right) }(\boldsymbol{%
\xi }_{\ast },\boldsymbol{\xi })-k_{\beta \alpha ,ji}^{\left( \alpha
,2\right) }(\boldsymbol{\xi }_{\ast },\boldsymbol{\xi })=k_{\beta \alpha
,ji}^{\left( \alpha \right) }(\boldsymbol{\xi }_{\ast },\boldsymbol{\xi }),
\label{sa2}
\end{equation}%
since, by applying the first relation in $\left( \ref{rel2}\right) $ and
renaming $\left\{ \boldsymbol{\xi }^{\prime },k\right\} \leftrightarrows
\left\{ \boldsymbol{\xi }_{\ast }^{\prime },l\right\} $, 
\begin{eqnarray*}
k_{\alpha \beta ,ij}^{\left( \beta ,1\right) }(\boldsymbol{\xi },\boldsymbol{%
\xi }_{\ast }) &=&\sum\limits_{k=1}^{r_{\alpha }}\sum\limits_{l=1}^{r_{\beta
}}\int_{\left( \mathbb{R}^{3}\right) ^{2}}\frac{w_{\beta \alpha }(%
\boldsymbol{\xi }_{\ast },\boldsymbol{\xi },I_{j}^{\beta },I_{i}^{\alpha
}\left\vert \boldsymbol{\xi }_{\ast }^{\prime },\boldsymbol{\xi }^{\prime
},I_{l}^{\beta },I_{k}^{\alpha }\right. )}{\left( M_{\alpha ,i}M_{\beta
,j\ast }\right) ^{1/2}}\,d\boldsymbol{\xi }^{\prime }d\boldsymbol{\xi }%
_{\ast }^{\prime } \\
&=&\sum\limits_{k=1}^{r_{\beta }}\sum\limits_{l=1}^{r_{\alpha }}\int_{\left( 
\mathbb{R}^{3}\right) ^{2}}\frac{w_{\beta \alpha }(\boldsymbol{\xi }_{\ast },%
\boldsymbol{\xi },I_{j}^{\beta },I_{i}^{\alpha }\left\vert \boldsymbol{\xi }%
^{\prime },\boldsymbol{\xi }_{\ast }^{\prime },I_{k}^{\beta },I_{l}^{\alpha
}\right. )}{\left( M_{\alpha ,i}M_{\beta ,j\ast }\right) ^{1/2}}\,d%
\boldsymbol{\xi }^{\prime }d\boldsymbol{\xi }_{\ast }^{\prime } \\
&=&k_{\beta \alpha ,ji}^{\left( \alpha ,1\right) }(\boldsymbol{\xi }_{\ast },%
\boldsymbol{\xi })\text{,}
\end{eqnarray*}%
while, by applying the two first relations in $\left( \ref{rel2}\right) $
and renaming $\left\{ \boldsymbol{\xi }^{\prime },k\right\} \leftrightarrows
\left\{ \boldsymbol{\xi }_{\ast }^{\prime },l\right\} $,%
\begin{eqnarray*}
k_{\alpha \beta ,ij}^{\left( \beta ,2\right) }(\boldsymbol{\xi },\boldsymbol{%
\xi }_{\ast }) &=&\sum\limits_{k=1}^{r_{\alpha }}\sum\limits_{l=1}^{r_{\beta
}}\int_{\left( \mathbb{R}^{3}\right) ^{2}}\frac{w_{\beta \alpha }(%
\boldsymbol{\xi }_{\ast }^{\prime },\boldsymbol{\xi },I_{l}^{\beta
},I_{i}^{\alpha }\left\vert \boldsymbol{\xi }_{\ast },\boldsymbol{\xi }%
^{\prime },I_{j}^{\beta },I_{k}^{\alpha }\right. )}{\left( M_{\alpha
,i}M_{\beta ,j\ast }\right) ^{1/2}}\,d\boldsymbol{\xi }^{\prime }d%
\boldsymbol{\xi }_{\ast }^{\prime } \\
&=&\sum\limits_{k=1}^{r_{\alpha }}\sum\limits_{l=1}^{r_{\beta }}\int_{\left( 
\mathbb{R}^{3}\right) ^{2}}\frac{w_{\beta \alpha }(\boldsymbol{\xi }_{\ast },%
\boldsymbol{\xi }^{\prime },I_{j}^{\beta },I_{k}^{\alpha }\left\vert 
\boldsymbol{\xi }_{\ast }^{\prime },\boldsymbol{\xi },I_{l}^{\beta
},I_{i}^{\alpha }\right. )}{\left( M_{\alpha ,i}M_{\beta ,j\ast }\right)
^{1/2}}\,d\boldsymbol{\xi }^{\prime }d\boldsymbol{\xi }_{\ast }^{\prime } \\
&=&\sum\limits_{k=1}^{r_{\beta }}\sum\limits_{l=1}^{r_{\alpha }}\int_{\left( 
\mathbb{R}^{3}\right) ^{2}}\frac{w_{\beta \alpha }(\boldsymbol{\xi }_{\ast },%
\boldsymbol{\xi }_{\ast }^{\prime },I_{j}^{\beta },I_{l}^{\alpha }\left\vert 
\boldsymbol{\xi }^{\prime },\boldsymbol{\xi },I_{k}^{\beta },I_{i}^{\alpha
}\right. )}{\left( M_{\alpha ,i}M_{\beta ,j\ast }\right) ^{1/2}}\,d%
\boldsymbol{\xi }^{\prime }d\boldsymbol{\xi }_{\ast }^{\prime } \\
&=&k_{\beta \alpha ,ji}^{\left( \alpha ,2\right) }(\boldsymbol{\xi }_{\ast },%
\boldsymbol{\xi })\text{.}
\end{eqnarray*}

We now continue by proving the compactness for the three different types of
collision kernel separately. Note that,  if $\alpha =\beta $, by applying
the last relation in $\left( \ref{rel2}\right) $, $k_{\alpha \beta
,ij}^{\left( \beta ,2\right) }(\boldsymbol{\xi },\boldsymbol{\xi }_{\ast
})=k_{\alpha \beta ,ij}^{\left( \alpha \right) }(\boldsymbol{\xi },%
\boldsymbol{\xi }_{\ast })$, and we will remain with only two cases - the
first two below. Even if $m_{\alpha }=m_{\beta }$, the kernels $k_{\alpha
\beta ,ij}^{\left( \alpha \right) }(\boldsymbol{\xi },\boldsymbol{\xi }%
_{\ast })$ and $k_{\alpha \beta ,ij}^{\left( \beta ,2\right) }(\boldsymbol{%
\xi },\boldsymbol{\xi }_{\ast })$ are structurally equal, why we (in
principle) remain with (first) two cases (the second one twice).

\begin{figure}[h]
\centering
\includegraphics[width=0.6\textwidth]{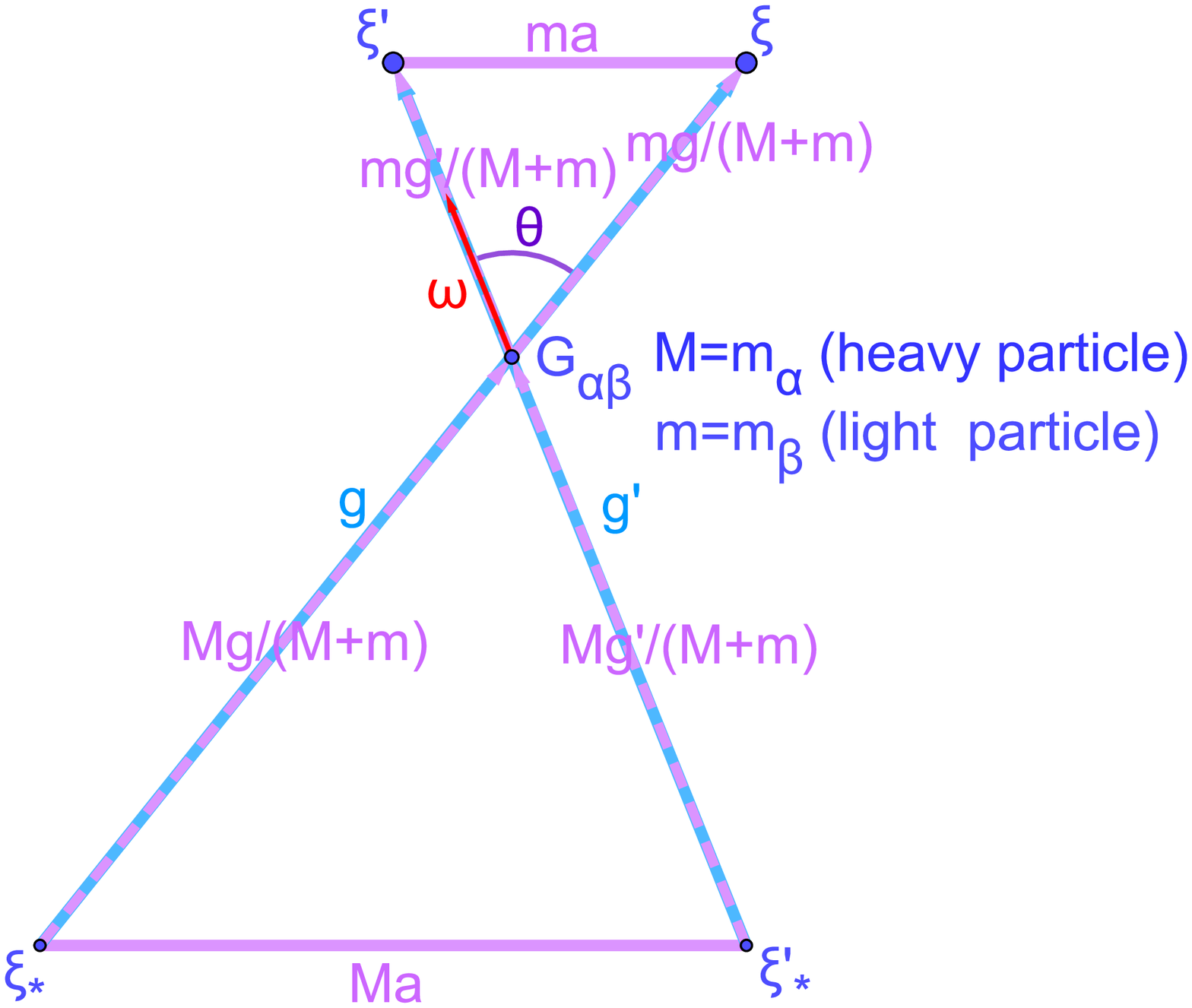}
\caption{Typical collision of $K_{\protect\alpha \protect\beta ,ij}^{(1)}$.}
\label{fig1}
\end{figure}

\textbf{I. Compactness of }$K_{\alpha \beta ,ij}^{(1)}=\int_{\mathbb{R}%
^{3}}k_{\alpha \beta ,ij}^{\left( \beta ,1\right) }(\boldsymbol{\xi },%
\boldsymbol{\xi }_{\ast })\,h_{\beta ,j\ast }\,d\boldsymbol{\xi }_{\ast }$
for $\left\{ \alpha ,\beta \right\} \subseteq \left\{ 1,...,s\right\} $ and $%
\left( i,j\right) \in \left\{ 1,...,r_{\alpha }\right\} \times \left\{
1,...,r_{\beta }\right\} $.

Assume the internal energy gap $\Delta I_{kl,ij}^{\alpha \beta
}=I_{k}^{\alpha }+I_{l}^{\beta }-I_{i}^{\alpha }-I_{j}^{\beta }$, as well
as, the velocities $\boldsymbol{\xi }$ and $\boldsymbol{\xi }_{\ast }$, to
be given. Then a collision will be uniquely determined by the unit vector $%
\boldsymbol{\omega }=\mathbf{g}^{\prime }/\left\vert \mathbf{g}^{\prime
}\right\vert $, with $\mathbf{g}^{\prime }=\boldsymbol{\xi }^{\prime }-%
\boldsymbol{\xi }_{\ast }^{\prime }$. This follows, since, by conservation
of momentum and total energy $\left( \ref{CI}\right) $, $m_{\alpha }\left( 
\boldsymbol{\xi }-\boldsymbol{\xi }^{\prime }\right) =m_{\beta }\left( 
\boldsymbol{\xi }_{\ast }^{\prime }-\boldsymbol{\xi }_{\ast }\right) $,
while also $\left\vert \mathbf{g}^{\prime }\right\vert $ can be obtained,
cf. Figure $\ref{fig1}$.

Indeed, expression $\left( \ref{k1}\right) $ of $k_{\alpha \beta
,ij}^{\left( \beta ,1\right) }$ may be
transformed - by a change of variables $\left\{ 
\boldsymbol{\xi }^{\prime },\boldsymbol{\xi }_{\ast }^{\prime }\right\}
\rightarrow \left\{ \left\vert \mathbf{g}^{\prime }\right\vert ,\boldsymbol{%
\omega }=\dfrac{\mathbf{g}^{\prime }}{\left\vert \mathbf{g}^{\prime
}\right\vert },\mathbf{G}_{\alpha \beta }^{\prime }=\dfrac{m_{\alpha }%
\boldsymbol{\xi }^{\prime }+m_{\beta }\boldsymbol{\xi }_{\ast }^{\prime }}{%
m_{\alpha }+m_{\beta }}\right\} $, cf. Figure $\ref{fig1}$, noting that $%
\left( \ref{df1}\right) $, and using relation $\left( \ref{M1}\right) $ - to%
\begin{eqnarray*}
k_{\alpha \beta ,ij}^{\left( \beta ,1\right) }(\boldsymbol{\xi },\boldsymbol{%
\xi }_{\ast }) &=&\left( M_{\alpha ,i}M_{\beta ,j\ast }\right)
^{1/2}\sum\limits_{k=1}^{r_{\alpha }}\sum\limits_{l=1}^{r_{\beta }}\int_{%
\mathbb{R}^{3}\times \mathbb{R}_{+}\times \mathbb{S}^{2}}\left\vert \mathbf{g%
}\right\vert \sigma _{ij,kl}^{\alpha \beta }\mathbf{1}_{\left\vert \mathbf{g}%
\right\vert ^{2}>2\widetilde{\Delta }I_{kl,ij}^{\alpha \beta }} \\
&&\times \delta _{1}\left( \sqrt{\left\vert \mathbf{g}\right\vert ^{2}-2%
\widetilde{\Delta }I_{kl,ij}^{\alpha \beta }}-\left\vert \mathbf{g}^{\prime
}\right\vert \right) \delta _{3}\left( \mathbf{G}_{\alpha \beta }-\mathbf{G}%
_{\alpha \beta }^{\prime }\right) d\mathbf{G}_{\alpha \beta }^{\prime
}d\left\vert \mathbf{g}^{\prime }\right\vert d\boldsymbol{\omega } \\
&=&\left( M_{\alpha ,i}M_{\beta ,j\ast }\right) ^{1/2}\left\vert \mathbf{g}%
\right\vert \sum\limits_{k=1}^{r_{\alpha }}\sum\limits_{l=1}^{r_{\beta
}}\int_{\mathbb{S}^{2}}\sigma _{ij,kl}^{\alpha \beta }\left( \left\vert 
\mathbf{g}\right\vert ,\cos \theta \right) \mathbf{1}_{\left\vert \mathbf{g}%
\right\vert ^{2}>2\widetilde{\Delta }I_{kl,ij}^{\alpha \beta }}\,d%
\boldsymbol{\omega }\text{,} \\
&&\text{ with }\cos \theta =\boldsymbol{\omega }\cdot \frac{\mathbf{g}}{%
\left\vert \mathbf{g}\right\vert }\text{, }\mathbf{g}=\boldsymbol{\xi }-%
\boldsymbol{\xi }_{\ast }\text{, }\mathbf{G}_{\alpha \beta }=\dfrac{%
m_{\alpha }\boldsymbol{\xi }+m_{\beta }\boldsymbol{\xi }_{\ast }}{m_{\alpha
}+m_{\beta }}\text{,} \\
\text{ } &&\text{and }\widetilde{\Delta }I_{kl,ij}^{\alpha \beta }=\frac{%
m_{\alpha }+m_{\beta }}{m_{\alpha }m_{\beta }}\Delta I_{kl,ij}^{\alpha \beta
}\text{.}
\end{eqnarray*}%
By assumption $\left( \ref{est1}\right) $ and the expression\textbf{\ } 
\begin{eqnarray}
m_{\alpha }\frac{\left\vert \boldsymbol{\xi }\right\vert ^{2}}{2}+m_{\beta }%
\frac{\left\vert \boldsymbol{\xi }_{\ast }\right\vert ^{2}}{2}+I_{i}^{\alpha
}+I_{j}^{\beta } &=&\frac{m_{\alpha }+m_{\beta }}{2}\left\vert \mathbf{G}%
_{\alpha \beta }\right\vert ^{2}+E_{ij}^{\alpha \beta }\text{, where}  \notag
\\
E_{ij}^{\alpha \beta } &=&\dfrac{m_{\alpha }m_{\beta }}{2\left( m_{\alpha
}+m_{\beta }\right) }\left\vert \mathbf{g}\right\vert ^{2}+I_{i}^{\alpha
}+I_{j}^{\beta }\text{,}  \label{m2}
\end{eqnarray}%
for the exponent of the product $M_{\alpha ,i}M_{\beta ,j\ast }$, the bound%
\begin{eqnarray}
&&\left( k_{\alpha \beta ,ij}^{\left( \beta ,1\right) }(\boldsymbol{\xi },%
\boldsymbol{\xi }_{\ast })\right) ^{2}  \notag \\
&\leq &\frac{C}{\left\vert \mathbf{g}\right\vert ^{2}}M_{\alpha ,i}M_{\beta
,j\ast }\left( \int_{\mathbb{S}^{2}}\,d\boldsymbol{\omega }\!\right)
^{2}\left( \sum\limits_{k=1}^{r_{\alpha }}\sum\limits_{l=1}^{r_{\beta
}}\left( \Psi _{ij,kl}^{\alpha \beta }+\left( \Psi _{ij,kl}^{\alpha \beta
}\right) ^{\gamma /2}\right) \mathbf{1}_{\left\vert \mathbf{g}\right\vert
^{2}>2\widetilde{\Delta }I_{kl,ij}^{\alpha \beta }}\right) ^{2}  \notag \\
&\leq &\frac{C}{\left\vert \mathbf{g}\right\vert ^{2}}e^{-\left( m_{\alpha
}+m_{\beta }\right) \left\vert \mathbf{G}_{\alpha \beta }\right\vert
^{2}/2-E_{ij}^{\alpha \beta }}\left( \sum\limits_{k=1}^{r_{\alpha
}}\sum\limits_{l=1}^{r_{\beta }}\left( 1+\left\vert \mathbf{g}\right\vert
^{2}\right) \right) ^{2}  \notag \\
&=&C\frac{\left( 1+\left\vert \mathbf{g}\right\vert ^{2}\right) ^{2}}{%
\left\vert \mathbf{g}\right\vert ^{2}}e^{-\left( m_{\alpha }+m_{\beta
}\right) \left\vert \mathbf{G}_{\alpha \beta }\right\vert
^{2}/2-E_{ij}^{\alpha \beta }}  \label{b1}
\end{eqnarray}%
may be obtained. Then, by applying the bound $\left( \ref{b1}\right) $ and
first changing variables of integration $\left\{ \boldsymbol{\xi },%
\boldsymbol{\xi }_{\ast }\right\} \rightarrow \left\{ \mathbf{g},\mathbf{G}%
_{\alpha \beta }\right\} $, with unitary Jacobian, and then to spherical
coordinates,%
\begin{eqnarray*}
&&\int_{\left( \mathbb{R}^{3}\right) ^{2}}\left( k_{\alpha \beta
,ij}^{\left( \beta ,1\right) }(\boldsymbol{\xi },\boldsymbol{\xi }_{\ast
})\right) ^{2}d\boldsymbol{\xi }d\boldsymbol{\xi }_{\ast } \\
&\leq &C\int_{\left( \mathbb{R}^{3}\right) ^{2}}e^{-\left( m_{\alpha
}+m_{\beta }\right) \left\vert \mathbf{G}_{\alpha \beta }\right\vert
^{2}/2-E_{ij}^{\alpha \beta }}\frac{\left( 1+\left\vert \mathbf{g}%
\right\vert ^{2}\right) ^{2}}{\left\vert \mathbf{g}\right\vert ^{2}}d\mathbf{%
g}\boldsymbol{\,}d\mathbf{G}_{\alpha \beta } \\
&\leq &C\int_{0}^{\infty }R^{2}e^{-\left( m_{\alpha }+m_{\beta }\right)
R^{2}/2}dR\int_{0}^{\infty }e^{-m_{\alpha }m_{\beta }s^{2}/(2\left(
m_{\alpha }+m_{\beta }\right) )}\left( 1+s^{2}\right) ^{2}ds=C\text{.}
\end{eqnarray*}%
Note that, here and below, we will, in general, not indicate an integration
over a directional vector in $\mathbb{S}^{2}$ of the form%
\begin{equation*}
\int_{\mathbb{S}^{2}}\,d\boldsymbol{\omega }=4\pi \text{,}
\end{equation*}%
but just integrate it in the generic constant $C$.

Hence,%
\begin{equation*}
K_{\alpha \beta ,ij}^{(1)}=\int_{\mathbb{R}^{3}}k_{\alpha \beta ,ij}^{\left(
\beta ,1\right) }(\boldsymbol{\xi },\boldsymbol{\xi }_{\ast })\,h_{\beta
,j\ast }\,d\boldsymbol{\xi }_{\ast }
\end{equation*}%
are Hilbert-Schmidt integral operators and as such continuous and compact on 
$L^{2}\left( d\boldsymbol{\xi }\right) $, see Theorem 7.83 in \cite%
{RenardyRogers}, for $\left( i,j\right) \in \left\{ 1,...,r_{\alpha
}\right\} \times \left\{ 1,...,r_{\beta }\right\} $ and $\left\{ \alpha
,\beta \right\} \subseteq \left\{ 1,...,s\right\} $.

\begin{figure}[h]
\centering
\includegraphics[width=0.5\textwidth]{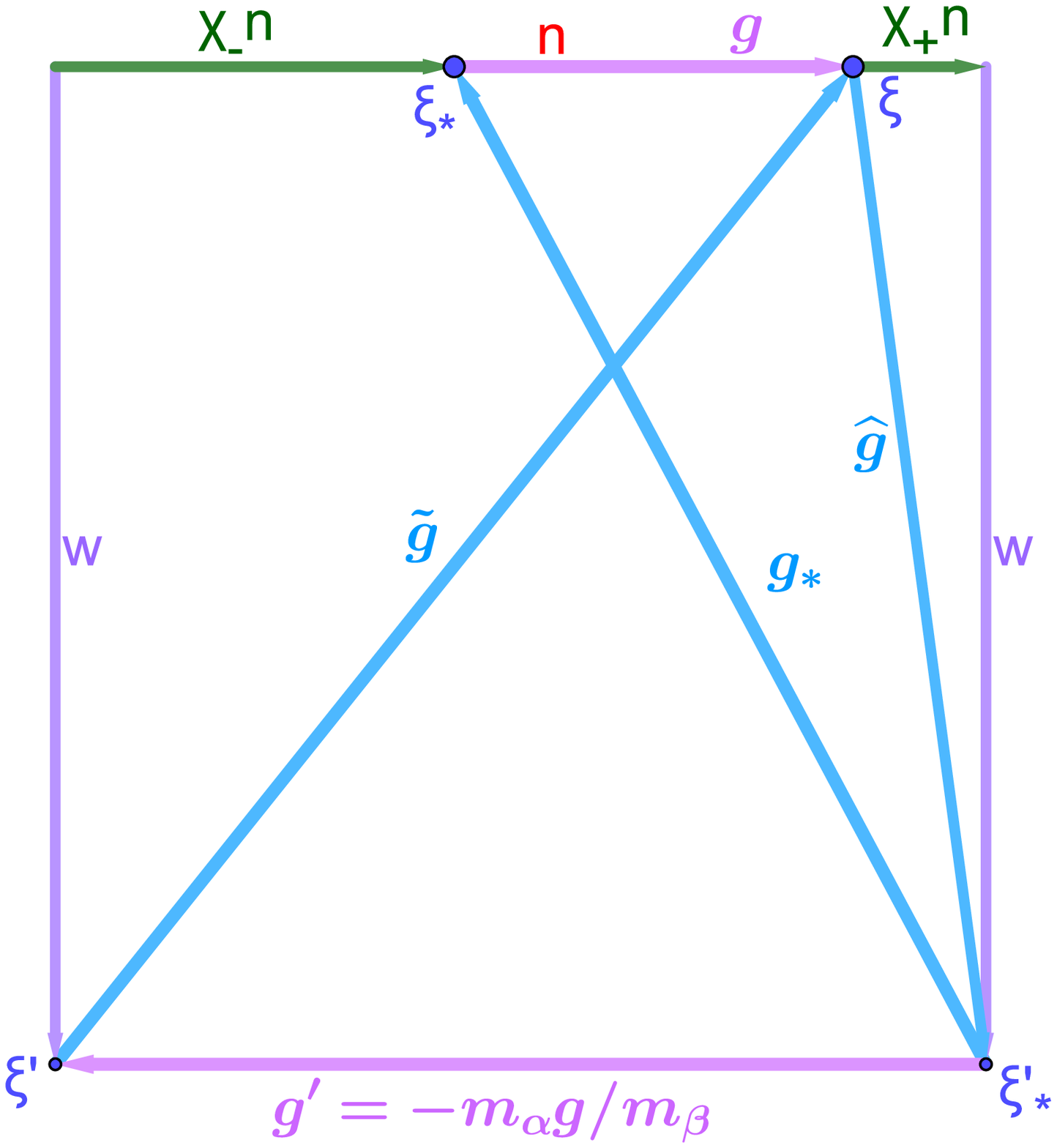}
\caption{Typical collision of $K_{\protect\alpha \protect\beta ,ij}^{(3)}$.}
\label{fig2}
\end{figure}

\textbf{II. Compactness of }$K_{\alpha \beta ,ij}^{(3)}=\int_{\mathbb{R}%
^{3}}k_{\alpha \beta ,ij}^{\left( \alpha \right) }(\boldsymbol{\xi },%
\boldsymbol{\xi }_{\ast })\,h_{\alpha ,j\ast }\,d\boldsymbol{\xi }_{\ast }$
for $\left\{ i,j\right\} \subseteq \left\{ 1,...,r_{\alpha }\right\} $ and $%
\left\{ \alpha ,\beta \right\} \subseteq \left\{ 1,...,s\right\} $.

Assume the internal energy gap $\Delta I_{ik,jl}^{\alpha \beta
}=I_{i}^{\alpha }+I_{k}^{\beta }-I_{j}^{\alpha }-I_{l}^{\beta }$, as well
as, the velocities $\boldsymbol{\xi }$ and $\boldsymbol{\xi }_{\ast }$, to
be given. Then a collision will be uniquely determined by a vector $\mathbf{w%
}$ orthogonal to $\mathbf{g}=\boldsymbol{\xi }-\boldsymbol{\xi }\mathbf{%
_{\ast }}$. This follows, since, by conservation of momentum and total
energy $\left( \ref{CI}\right) $ (reminding the relabeling of the velocities
and internal energies), the relation between $\left\vert \boldsymbol{\xi }-%
\boldsymbol{\xi }^{\prime }\right\vert $ and $\left\vert \boldsymbol{\xi }%
_{\ast }^{\prime }-\boldsymbol{\xi }_{\ast }\right\vert $ can be obtained,
while also $m_{\beta }\mathbf{g}^{\prime }=m_{\beta }\left( \boldsymbol{\xi }%
_{\ast }^{\prime }-\boldsymbol{\xi }^{\prime }\right) =m_{\alpha }\mathbf{g}$%
, cf. Figure $\ref{fig2}$. Indeed, note that - aiming to obtain expressions
for $\mathbf{g}^{\prime }$ and $\chi _{+}$ in the arguments of the
delta-functions,%
\begin{eqnarray*}
&&W_{\alpha \beta }(\boldsymbol{\xi },\boldsymbol{\xi }^{\prime
},I_{i}^{\alpha },I_{k}^{\beta }\left\vert \boldsymbol{\xi }_{\ast },%
\boldsymbol{\xi }_{\ast }^{\prime },I_{j}^{\alpha },I_{l}^{\beta }\right. )
\\
&=&\left( m_{\alpha }+m_{\beta }\right) ^{2}m_{\alpha }m_{\beta }\varphi
_{i}^{\alpha }\varphi _{k}^{\beta }\sigma _{ik,jl}^{\alpha \beta }\frac{%
\left\vert \widetilde{\mathbf{g}}\right\vert }{\left\vert \mathbf{g}_{\ast
}\right\vert }\delta _{3}\left( m_{\alpha }\mathbf{g}+m_{\beta }\mathbf{g}%
^{\prime }\right)  \\
&&\times \delta _{1}\left( m_{\alpha }\left\vert \mathbf{g}\right\vert
\left( \chi -\frac{m_{\alpha }-m_{\beta }}{2m_{\beta }}\left\vert \mathbf{g}%
\right\vert \right) -\Delta I_{ik,jl}^{\alpha \beta }\right)  \\
&=&\frac{\left( m_{\alpha }+m_{\beta }\right) ^{2}\left\vert \widetilde{%
\mathbf{g}}\right\vert }{\left\vert \mathbf{g}_{\ast }\right\vert \left\vert 
\mathbf{g}\right\vert m_{\beta }^{2}}\varphi _{i}^{\alpha }\varphi
_{k}^{\beta }\sigma _{ik,jl}^{\alpha \beta }\delta _{3}\left( \frac{%
m_{\alpha }}{m_{\beta }}\mathbf{g}+\mathbf{g}^{\prime }\right) \delta
_{1}\left( \chi _{+}-\frac{m_{\alpha }-m_{\beta }}{2m_{\beta }}\left\vert 
\mathbf{g}\right\vert -\frac{\Delta I_{ik,jl}^{\alpha \beta }}{m_{\alpha
}\left\vert \mathbf{g}\right\vert }\right) \text{,}
\end{eqnarray*}%
where $\mathbf{g}=\boldsymbol{\xi }-\boldsymbol{\xi }_{\ast }$, $\mathbf{g}%
^{\prime }=\boldsymbol{\xi }^{\prime }-\boldsymbol{\xi }_{\ast }^{\prime }$, 
$\widetilde{\mathbf{g}}=\boldsymbol{\xi }-\boldsymbol{\xi }^{\prime }$, $%
\mathbf{g}_{\ast }=\boldsymbol{\xi }_{\ast }-\boldsymbol{\xi }_{\ast
}^{\prime }$, $\Delta I_{ik,jl}^{\alpha \beta }=I_{i}^{\alpha }+I_{k}^{\beta
}-I_{j}^{\alpha }-I_{l}^{\beta }$, and $\chi _{+}=\left( \boldsymbol{\xi }%
_{\ast }^{\prime }-\boldsymbol{\xi }\right) \cdot \mathbf{n}$, with $\mathbf{%
n}=\dfrac{\mathbf{g}}{\left\vert \mathbf{g}\right\vert }$. Then by
performing a change of variables $\left\{ \boldsymbol{\xi }^{\prime },%
\boldsymbol{\xi }_{\ast }^{\prime }\right\} \rightarrow \left\{ \mathbf{g}%
^{\prime }=\boldsymbol{\xi }^{\prime }-\boldsymbol{\xi }_{\ast }^{\prime },~%
\widehat{\mathbf{g}}=\boldsymbol{\xi }_{\ast }^{\prime }-\boldsymbol{\xi }%
\right\} $, where%
\begin{equation*}
d\boldsymbol{\xi }^{\prime }d\boldsymbol{\xi }_{\ast }^{\prime }=d\mathbf{g}%
^{\prime }d\widehat{\mathbf{g}}=d\mathbf{g}^{\prime }d\chi _{+}d\mathbf{w}%
\text{, with }\mathbf{w}=\boldsymbol{\xi }_{\ast }^{\prime }-\boldsymbol{\xi 
}-\chi _{+}\mathbf{n}\text{.}
\end{equation*}%
the expression $\left( \ref{k1}\right) $ of $k_{\alpha \beta ,ij}^{\left(
\alpha \right) }$ may be rewritten in the following way%
\begin{eqnarray*}
&&k_{\alpha \beta ,ij}^{\left( \alpha \right) }(\boldsymbol{\xi },%
\boldsymbol{\xi }_{\ast }) \\
&=&\sum\limits_{k,l=1}^{r_{\beta }}\int_{\left( \mathbb{R}^{3}\right)
^{2}}\left( M_{\beta ,k}^{\prime }M_{\beta ,l\ast }^{\prime }\right) ^{1/2}%
\frac{W_{\alpha \beta }(\boldsymbol{\xi },\boldsymbol{\xi }^{\prime
},I_{i}^{\alpha },I_{k}^{\beta }\left\vert \boldsymbol{\xi }_{\ast },%
\boldsymbol{\xi }_{\ast }^{\prime },I_{j}^{\alpha },I_{l}^{\beta }\right. )}{%
\left( \varphi _{i}^{\alpha }\varphi _{k}^{\beta }\varphi _{j}^{\alpha
}\varphi _{l}^{\beta }\right) ^{1/2}}d\mathbf{g}^{\prime }d\widehat{\mathbf{g%
}} \\
&=&\sum\limits_{k,l=1}^{r_{\beta }}\int_{\left( \mathbb{R}^{3}\right)
^{\perp _{\mathbf{n}}}}\frac{\left( m_{\alpha }+m_{\beta }\right) ^{2}}{%
m_{\beta }^{2}}\frac{\left\vert \widetilde{\mathbf{g}}\right\vert \left(
M_{\beta ,k}^{\prime }M_{\beta ,l\ast }^{\prime }\right) ^{1/2}}{\left\vert 
\mathbf{g}_{\ast }\right\vert \left\vert \mathbf{g}\right\vert }\left( \frac{%
\varphi _{i}^{\alpha }\varphi _{k}^{\beta }}{\varphi _{j}^{\alpha }\varphi
_{l}^{\beta }}\right) ^{1/2} \\
&&\times \mathbf{1}_{\left\vert \widetilde{\mathbf{g}}\right\vert ^{2}>2%
\widetilde{\Delta }I_{jl,ik}^{\alpha \beta }}\sigma _{ik,jl}^{\alpha \beta
}\left( \left\vert \widetilde{\mathbf{g}}\right\vert ,\frac{\widetilde{%
\mathbf{g}}\cdot \mathbf{g}_{\ast }}{\left\vert \widetilde{\mathbf{g}}%
\right\vert \left\vert \mathbf{g}_{\ast }\right\vert }\right) d\mathbf{w}%
\text{, with }\widetilde{\Delta }I_{jl,ik}^{\alpha \beta }=\frac{m_{\alpha
}+m_{\beta }}{m_{\alpha }m_{\beta }}\Delta I_{jl,ik}^{\alpha \beta }\text{,}
\end{eqnarray*}%
where 
\begin{equation*}
\left( \mathbb{R}^{3}\right) ^{\perp _{\mathbf{n}}}=\left\{ \mathbf{w}\in 
\mathbb{R}^{3}:\mathbf{w}\perp \mathbf{n}\right\} .
\end{equation*}%
Here, see Figure $\ref{fig2}$,%
\begin{equation*}
\left\{ 
\begin{array}{l}
\boldsymbol{\xi }^{\prime }=\boldsymbol{\xi }_{\ast }+\mathbf{w}+\chi _{-}%
\mathbf{n} \\ 
\boldsymbol{\xi }_{\ast }^{\prime }=\boldsymbol{\xi }+\mathbf{w}+\chi _{+}%
\mathbf{n}%
\end{array}%
\right. \text{, with }\chi _{\pm }=\frac{\Delta I_{ik,jl}^{\alpha \beta }}{%
m_{\alpha }\left\vert \mathbf{g}\right\vert }\pm \frac{m_{\alpha }-m_{\beta }%
}{2m_{\beta }}\left\vert \mathbf{g}\right\vert \text{,}
\end{equation*}%
implying that the kinetic energy part of the exponent of the product $%
M_{\beta ,k}^{\prime }M_{\beta ,l\ast }^{\prime }$ equals 
\begin{eqnarray*}
&&m_{\beta }\frac{\left\vert \boldsymbol{\xi }^{\prime }\right\vert ^{2}}{2}%
+m_{\beta }\frac{\left\vert \boldsymbol{\xi }_{\ast }^{\prime }\right\vert
^{2}}{2} \\
&=&m_{\beta }\left\vert \frac{\boldsymbol{\xi +\xi }_{\ast }}{2}-\frac{%
\Delta I_{ik,jl}^{\alpha \beta }}{m_{\alpha }\left\vert \mathbf{g}%
\right\vert }\mathbf{n}+\mathbf{w}\right\vert ^{2}+\frac{m_{\alpha }^{2}}{%
4m_{\beta }}\left\vert \mathbf{g}\right\vert ^{2} \\
&=&m_{\beta }\left\vert \frac{\left( \boldsymbol{\xi +\xi }_{\ast }\right)
_{\perp _{\boldsymbol{n}}}}{2}+\mathbf{w}\right\vert ^{2}+m_{\beta }\left( 
\frac{\left( \boldsymbol{\xi +\xi }_{\ast }\right) _{\mathbf{n}}}{2}-\frac{%
\Delta I_{ik,jl}^{\alpha \beta }}{m_{\alpha }\left\vert \mathbf{g}%
\right\vert }\right) ^{2}+\frac{m_{\alpha }^{2}}{4m_{\beta }}\left\vert 
\mathbf{g}\right\vert ^{2} \\
&=&m_{\beta }\left\vert \frac{\left( \boldsymbol{\xi +\xi }_{\ast }\right)
_{\perp _{\boldsymbol{n}}}}{2}+\mathbf{w}\right\vert ^{2}+\frac{m_{\beta
}\left( m_{\alpha }\left( \left\vert \boldsymbol{\xi }_{\ast }\right\vert
^{2}-\left\vert \boldsymbol{\xi }\right\vert ^{2}\right) +2\Delta
I_{ik,jl}^{\alpha \beta }\right) ^{2}}{4m_{\alpha }^{2}\left\vert \mathbf{g}%
\right\vert ^{2}}+\frac{m_{\alpha }^{2}}{4m_{\beta }}\left\vert \mathbf{g}%
\right\vert ^{2},
\end{eqnarray*}%
where%
\begin{eqnarray*}
\left( \boldsymbol{\xi +\xi }_{\ast }\right) _{\mathbf{n}} &=&\left( 
\boldsymbol{\xi +\xi }_{\ast }\right) \cdot \mathbf{n}=\frac{\left\vert 
\boldsymbol{\xi }\right\vert ^{2}-\left\vert \boldsymbol{\xi }_{\ast
}\right\vert ^{2}}{\left\vert \boldsymbol{\xi }-\boldsymbol{\xi }_{\ast
}\right\vert },\text{\ and} \\
\left( \boldsymbol{\xi +\xi }_{\ast }\right) _{\perp _{\boldsymbol{n}}} &=&%
\boldsymbol{\xi +\xi }_{\ast }-\left( \boldsymbol{\xi +\xi }_{\ast }\right)
_{\mathbf{n}}\mathbf{n}.
\end{eqnarray*}%
Hence, by assumption $\left( \ref{est1}\right) $ and the Cauchy-Schwarz
inequality,%
\begin{eqnarray}
&&\left( k_{\alpha \beta ,ij}^{\left( \alpha \right) }(\boldsymbol{\xi },%
\boldsymbol{\xi }_{\ast })\right) ^{2}  \notag \\
&\leq &\frac{C}{\left\vert \mathbf{g}\right\vert ^{2}}\left(
\sum\limits_{k,l=1}^{r_{\beta }}\frac{1}{e^{\frac{I_{k}^{\beta
}+I_{l}^{\beta }}{2}}}\exp \left( -m_{\beta }\frac{\left( m_{\alpha }\left(
\left\vert \boldsymbol{\xi }_{\ast }\right\vert ^{2}-\left\vert \boldsymbol{%
\xi }\right\vert ^{2}\right) +2\Delta I_{ik,jl}^{\alpha \beta }\right) ^{2}}{%
8m_{\alpha }^{2}\left\vert \mathbf{g}\right\vert ^{2}}-\frac{m_{\alpha }^{2}%
}{8m_{\beta }}\left\vert \mathbf{g}\right\vert ^{2}\right) \right.   \notag
\\
&&\times \left. \sum\limits_{k,l=1}^{r_{\beta }}\int\limits_{\left( \mathbb{R%
}^{3}\right) ^{\perp _{\mathbf{n}}}}\left( 1+\frac{\mathbf{1}_{\left\vert 
\widetilde{\mathbf{g}}\right\vert ^{2}>2\widetilde{\Delta }I_{jl,ik}^{\alpha
\beta }}}{\left( \widetilde{\Psi }_{ik,jl}^{\alpha \beta }\right) ^{1-\gamma
/2}}\right) \exp \left( -\frac{m_{\beta }}{2}\left\vert \frac{\left( 
\boldsymbol{\xi +\xi }_{\ast }\right) _{\perp _{\boldsymbol{n}}}}{2}+\mathbf{%
w}\right\vert ^{2}\right) d\mathbf{w}\right) ^{2}  \notag \\
&\leq &\frac{C}{\left\vert \mathbf{g}\right\vert ^{2}}\left(
\sum\limits_{k,l=1}^{r_{\beta }}\exp \left( -m_{\beta }\frac{\left(
m_{\alpha }\!\left( \left\vert \boldsymbol{\xi }_{\ast }\right\vert
^{2}\!-\left\vert \boldsymbol{\xi }\right\vert ^{2}\right) +2\Delta
I_{ik,jl}^{\alpha \beta }\right) ^{2}}{8m_{\alpha }^{2}\left\vert \mathbf{g}%
\right\vert ^{2}}-\frac{m_{\alpha }^{2}}{8m_{\beta }}\left\vert \mathbf{g}%
\right\vert ^{2}\right) \right) ^{2}  \notag \\
&=&\frac{C}{\left\vert \mathbf{g}\right\vert ^{2}}\left(
\sum\limits_{k,l=1}^{r_{\beta }}\exp \left( -\frac{m_{\beta }}{8}\left(
\left\vert \mathbf{g}\right\vert +2\left\vert \boldsymbol{\xi }\right\vert
\cos \varphi +2\chi _{ik,jl}^{\alpha \beta }\right) ^{2}-\frac{m_{\alpha
}^{2}}{8m_{\beta }}\left\vert \mathbf{g}\right\vert ^{2}\right) \right) ^{2}
\notag \\
&\leq &\frac{C}{\left\vert \mathbf{g}\right\vert ^{2}}\sum%
\limits_{k,l=1}^{r_{\beta }}\exp \left( -m_{\beta }\left( \dfrac{\left\vert 
\mathbf{g}\right\vert }{2}+\left\vert \boldsymbol{\xi }\right\vert \cos
\varphi +\chi _{ik,jl}^{\alpha \beta }\right) ^{2}-\frac{m_{\alpha }^{2}}{%
4m_{\beta }}\left\vert \mathbf{g}\right\vert ^{2}\right) \text{, with } 
\notag \\
&&\chi _{ik,jl}^{\alpha \beta }=\chi _{ik,jl}^{\alpha \beta }\left(
\left\vert \mathbf{g}\right\vert \right) =\frac{\Delta I_{ik,jl}^{\alpha
\beta }}{m_{\alpha }\left\vert \mathbf{g}\right\vert }\text{, }\cos \varphi =%
\mathbf{n}\cdot \frac{\boldsymbol{\xi }}{\left\vert \boldsymbol{\xi }%
\right\vert }\text{, }  \notag \\
&&\widetilde{\Psi }_{ik,jl}^{\alpha \beta }=\left\vert \widetilde{\mathbf{g}}%
\right\vert \left\vert \mathbf{g}_{\ast }\right\vert \text{, and }\left\vert 
\mathbf{g}_{\ast }\right\vert ^{2}=\left\vert \widetilde{\mathbf{g}}%
\right\vert ^{2}-2\frac{m_{\alpha }+m_{\beta }}{m_{\alpha }m_{\beta }}\Delta
I_{ik,jl}^{\alpha \beta }\text{.}  \label{b2}
\end{eqnarray}%
Here, the second inequality follows by the following bound, which can be
obtained by noting that $\min \left( \left\vert \widetilde{\mathbf{g}}%
\right\vert ,\left\vert \mathbf{g}_{\ast }\right\vert \right) \geq
\left\vert \mathbf{w}\right\vert $, cf. Figure $\ref{fig2}$, and making a
change of variables $\mathbf{w}\rightarrow \widetilde{\mathbf{w}}=\left( 
\boldsymbol{\xi +\xi }_{\ast }\right) _{\perp _{\boldsymbol{n}}}/2+\mathbf{w}
$ followed by one to polar coordinates, 
\begin{eqnarray*}
&&\int_{\left( \mathbb{R}^{3}\right) ^{\perp _{\mathbf{n}}}}\left( 1+\frac{%
\mathbf{1}_{\left\vert \widetilde{\mathbf{g}}\right\vert ^{2}>2\widetilde{%
\Delta }I_{jl,ik}^{\alpha \beta }}}{\left( \widetilde{\Psi }_{ik,jl}^{\alpha
\beta }\right) ^{1-\gamma /2}}\right) \exp \left( -\frac{m_{\beta }}{2}%
\left\vert \frac{\left( \boldsymbol{\xi +\xi }_{\ast }\right) _{\perp _{%
\boldsymbol{n}}}}{2}+\mathbf{w}\right\vert ^{2}\right) d\mathbf{w} \\
&\leq &\int_{\left\vert \mathbf{w}\right\vert \leq 1}1+\left\vert \mathbf{w}%
\right\vert ^{\gamma -2}\,d\mathbf{w}+2\int_{\left\vert \mathbf{w}%
\right\vert \geq 1}\exp \left( -\frac{m_{\beta }}{2}\left\vert \frac{\left( 
\boldsymbol{\xi +\xi }_{\ast }\right) _{\perp _{\boldsymbol{n}}}}{2}+\mathbf{%
w}\right\vert ^{2}\right) d\mathbf{w} \\
&\leq &\int_{\left\vert \mathbf{w}\right\vert \leq 1}1+\left\vert \mathbf{w}%
\right\vert ^{\gamma -2}\,d\mathbf{w}+2\int_{\left( \mathbb{R}^{3}\right)
^{\perp _{\mathbf{n}}}}e^{-m_{\beta }\left\vert \widetilde{\mathbf{w}}%
\right\vert ^{2}/2}\,d\widetilde{\mathbf{w}} \\
&=&2\pi \left( \int_{0}^{1}R+R^{\gamma -1}\,dR+2\int_{\left( \mathbb{R}%
^{3}\right) ^{\perp _{\mathbf{n}}}}Re^{-m_{\beta }R^{2}/2}\,dR\right) =C%
\text{.}
\end{eqnarray*}

Then $k_{\alpha \beta ,ij}^{\left( \alpha \right) }(\boldsymbol{\xi },%
\boldsymbol{\xi }_{\ast })\mathbf{1}_{\mathfrak{h}_{N}}\in L^{2}\left( d%
\boldsymbol{\xi \,}d\boldsymbol{\xi }_{\ast }\right) $. Indeed, by changing
variables $\boldsymbol{\xi }\mathbf{_{\ast }}\rightarrow \mathbf{g}$, with $%
\mathbf{g}=\boldsymbol{\xi }-\boldsymbol{\xi }\mathbf{_{\ast }}$, and then
to spherical coordinates,

\begin{eqnarray*}
\int_{\mathfrak{h}_{N}}\left( k_{\alpha \beta ,ij}^{\left( \alpha \right) }(%
\boldsymbol{\xi },\boldsymbol{\xi }_{\ast })\right) ^{2}\,d\boldsymbol{\xi \,%
}d\boldsymbol{\xi }_{\ast } &\leq &\int_{\mathfrak{h}_{N}}\frac{C}{%
\left\vert \mathbf{g}\right\vert ^{2}}e^{-m_{\alpha }^{2}\left\vert \mathbf{g%
}\right\vert ^{2}/\left( 4m_{\beta }\right) }d\mathbf{g}\boldsymbol{\,}d%
\boldsymbol{\xi } \\
&=&C\int_{0}^{\infty }e^{-m_{\alpha }^{2}R^{2}/\left( 4m_{\beta }\right)
}dR\int_{0}^{N}\eta ^{2}d\eta \\
&=&CN^{3}\text{.}
\end{eqnarray*}%
Next we aim for proving that the integral of $k_{\alpha \beta ,ij}^{\left(
\alpha \right) }(\boldsymbol{\xi },\boldsymbol{\xi }_{\ast })$ with respect
to $\boldsymbol{\xi }$ over $\mathbb{R}^{3}$ is bounded in $\boldsymbol{\xi }%
_{\ast }$. Indeed, directly by the bound $\left( \ref{b2}\right) $ on $%
\left( k_{\alpha \beta ,ij}^{\left( \alpha \right) }\right) ^{2}$%
\begin{equation}
0\leq k_{\alpha \beta ,ij}^{\left( \alpha \right) }(\boldsymbol{\xi },%
\boldsymbol{\xi }_{\ast })\leq \frac{C}{\left\vert \mathbf{g}\right\vert }%
\sum\limits_{k,l=1}^{r_{\beta }}\exp \left( -\frac{m_{\beta }}{8}\left(
\left\vert \mathbf{g}\right\vert +2\left\vert \boldsymbol{\xi }\right\vert
\cos \varphi +2\chi _{ik,jl}^{\alpha \beta }\right) ^{2}-\frac{m_{\alpha
}^{2}}{8m_{\beta }}\left\vert \mathbf{g}\right\vert ^{2}\right) \text{.}
\label{b7}
\end{equation}%
Hence, due to the symmetry $k_{\alpha \beta ,ij}^{\left( \alpha \right) }(%
\boldsymbol{\xi },\boldsymbol{\xi }_{\ast })=k_{\alpha \beta ,ji}^{\left(
\alpha \right) }(\boldsymbol{\xi }_{\ast },\boldsymbol{\xi })$ $\left( \ref%
{sa1}\right) $, by a change of variables $\boldsymbol{\xi }\rightarrow 
\mathbf{g}=\boldsymbol{\xi }-\boldsymbol{\xi }\mathbf{_{\ast }}$, followed
by one to spherical coordinates, 
\begin{eqnarray*}
&&\int_{\mathbb{R}^{3}}k_{\alpha \beta ,ij}^{\left( \alpha \right) }(%
\boldsymbol{\xi },\boldsymbol{\xi }_{\ast })\,d\boldsymbol{\xi =}\int_{%
\mathbb{R}^{3}}k_{\alpha \beta ,ji}^{\left( \alpha \right) }(\boldsymbol{\xi 
}_{\ast },\boldsymbol{\xi })\,d\boldsymbol{\xi } \\
&\leq &\int_{\mathbb{R}^{3}}\frac{C}{\left\vert \mathbf{g}\right\vert }%
\sum\limits_{k,l=1}^{r_{\beta }}\exp \left( -\frac{m_{\alpha }^{2}}{%
8m_{\beta }}\left\vert \mathbf{g}\right\vert ^{2}\right) d\mathbf{g}%
=C\int_{0}^{\infty }Re^{-m_{\alpha }^{2}R^{2}/\left( 8m_{\beta }\right) }dR=C%
\text{.}
\end{eqnarray*}

Finally, heading for proving the uniform convergence of the integral of $%
k_{\alpha \beta ,ij}^{\left( \alpha \right) }$ with respect to $\boldsymbol{%
\xi }_{\ast }$ over the truncated domain $\mathfrak{h}_{N}$ to the one over
all of $\mathbb{R}^{3}$, the following bound on the integral over $\mathbb{R}%
^{3}$ can be obtained for $\left\vert \boldsymbol{\xi }\right\vert \neq 0$.
Indeed, by bound $\left( \ref{b7}\right) $, by changing variables $%
\boldsymbol{\xi }_{\ast }\rightarrow \mathbf{g}=\boldsymbol{\xi }-%
\boldsymbol{\xi }\mathbf{_{\ast }}$, then to (conventional) spherical
coordinates, with $\boldsymbol{\xi }$ as zenithal direction, and hence, $%
\varphi $ as polar angle, followed by the change of variables $\varphi
\rightarrow \eta =R+2\left\vert \boldsymbol{\xi }\right\vert \cos \varphi
+2\chi _{ik}^{jl}\left( R\right) $, with $d\eta =-2\left\vert \boldsymbol{%
\xi }\right\vert \sin \varphi \,d\varphi $,%
\begin{eqnarray}
&&\int_{\mathbb{R}^{3}}k_{\alpha \beta ,ij}^{\left( \alpha \right) }(%
\boldsymbol{\xi },\boldsymbol{\xi }_{\ast })\,d\boldsymbol{\xi }_{\ast } 
\notag \\
&\leq &\int_{\mathbb{R}^{3}}\frac{C}{\left\vert \mathbf{g}\right\vert }%
\sum\limits_{k,l=1}^{r_{\beta }}R\exp \left( -\frac{m_{\beta }}{8}\left(
R+2\left\vert \boldsymbol{\xi }\right\vert \cos \varphi +2\chi
_{ik,jl}^{\alpha \beta }(R)\right) ^{2}-\frac{m_{\alpha }^{2}}{8m_{\beta }}%
R^{2}\right) d\mathbf{g}  \notag \\
&=&C\sum\limits_{k,l=1}^{r_{\beta }}\int_{0}^{\infty }\int_{0}^{\pi }R\exp
\left( -\frac{m_{\beta }}{8}\left( R+2\left\vert \boldsymbol{\xi }%
\right\vert \cos \varphi +2\chi _{ik,jl}^{\alpha \beta }(R)\right) ^{2}-%
\frac{m_{\alpha }^{2}}{8m_{\beta }}R^{2}\right)   \notag \\
&&\times \sin \varphi \,d\varphi dR  \notag \\
&=&\frac{C}{\left\vert \boldsymbol{\xi }\right\vert }\sum%
\limits_{k,l=1}^{r_{\beta }}\int_{0}^{\infty }\int_{R+2\chi _{ik,jl}^{\alpha
\beta }(R)-2\left\vert \boldsymbol{\xi }\right\vert }^{R+2\chi
_{ik,jl}^{\alpha \beta }(R)+2\left\vert \boldsymbol{\xi }\right\vert
}Re^{-m_{\beta }\eta ^{2}/8}e^{-m_{\alpha }^{2}R^{2}/\left( 8m_{\beta
}\right) }d\eta dR  \notag \\
&\leq &\frac{C}{\left\vert \boldsymbol{\xi }\right\vert }\int_{0}^{\infty
}Re^{-m_{\alpha }^{2}R^{2}/\left( 8m_{\beta }\right) }\,dR\int_{-\infty
}^{\infty }e^{-m_{\beta }^{2}\eta /8}d\eta =\frac{C}{\left\vert \boldsymbol{%
\xi }\right\vert }\text{.}  \label{b7a}
\end{eqnarray}%
Then, by the bounds $\left( \ref{b7}\right) $ and $\left( \ref{b7a}\right) $%
, 
\begin{eqnarray*}
&&\sup_{\boldsymbol{\xi }\in \mathbb{R}^{3}}\int_{\mathbb{R}^{3}}k_{\alpha
\beta ,ij}^{\left( \alpha \right) }(\boldsymbol{\xi },\boldsymbol{\xi }%
_{\ast })-k_{\alpha \beta ,ij}^{\left( \alpha \right) }(\boldsymbol{\xi },%
\boldsymbol{\xi }_{\ast })\mathbf{1}_{\mathfrak{h}_{N}}\,d\boldsymbol{\xi }%
_{\ast } \\
&\leq &\sup_{\boldsymbol{\xi }\in \mathbb{R}^{3}}\int_{\left\vert \mathbf{g}%
\right\vert \leq \frac{1}{N}}k_{\alpha \beta ,ij}^{\left( \alpha \right) }(%
\boldsymbol{\xi },\boldsymbol{\xi }_{\ast })\,d\boldsymbol{\xi }_{\ast
}+\sup_{\left\vert \boldsymbol{\xi }\right\vert \geq N}\int_{\mathbb{R}%
^{3}}k_{\alpha \beta ,ij}^{\left( \alpha \right) }(\boldsymbol{\xi },%
\boldsymbol{\xi }_{\ast })\,d\boldsymbol{\xi }_{\ast } \\
&\leq &\int_{\left\vert \mathbf{g}\right\vert \leq \frac{1}{N}}\frac{C}{%
\left\vert \mathbf{g}\right\vert }\,d\mathbf{g}+\frac{C}{N}\leq C\left(
\int_{0}^{\frac{1}{N}}R\,dR+\frac{1}{N}\right)  \\
&=&C\left( \frac{1}{N^{2}}+\frac{1}{N}\right) \rightarrow 0\text{ as }%
N\rightarrow \infty \text{.}
\end{eqnarray*}%
Hence, by Lemma \ref{LGD} the operators 
\begin{equation*}
K_{\alpha \beta ,ij}^{(3)}=\int_{\mathbb{R}^{3}}k_{\alpha \beta ,ij}^{\left(
\alpha \right) }(\boldsymbol{\xi },\boldsymbol{\xi }_{\ast })\,h_{\alpha
,j\ast }\,d\boldsymbol{\xi }
\end{equation*}%
are compact on $L^{2}\left( d\boldsymbol{\xi }\right) $ for $\left\{
i,j\right\} \subseteq \left\{ 1,...,r_{\alpha }\right\} $ and $\left\{
\alpha ,\beta \right\} \subseteq \left\{ 1,...,s\right\} $.
\begin{figure}[h]
\centering
\includegraphics[width=0.5\textwidth]{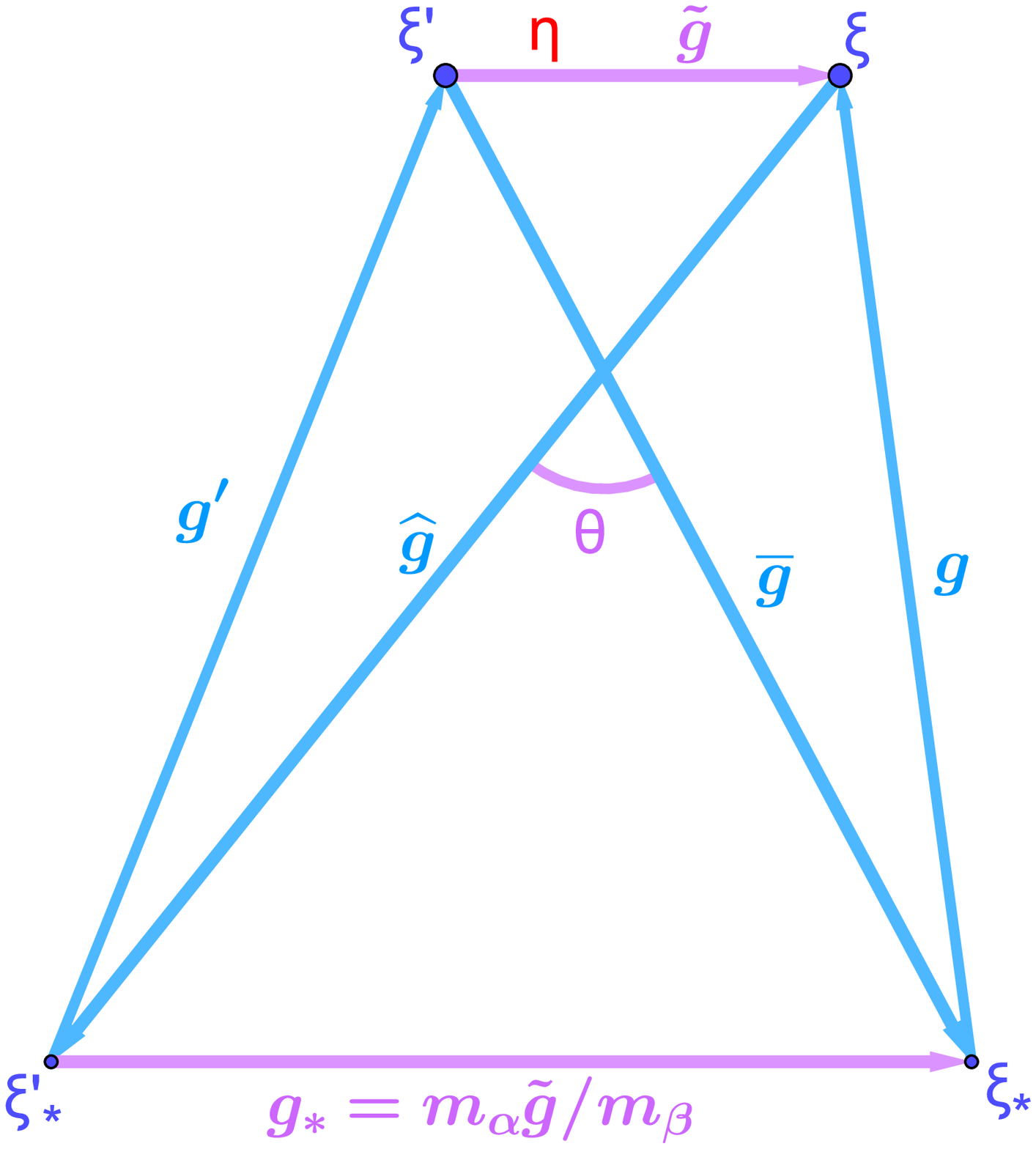}
\caption{Typical collision of $K_{\protect\alpha \protect\beta ,ij}^{(2)}$.}
\label{fig3}
\end{figure}

\textbf{III. Compactness of }$K_{\alpha \beta ,ij}^{(2)}=\int_{\mathbb{R}%
^{3}}k_{\alpha \beta ,ij}^{\left( \beta ,2\right) }(\boldsymbol{\xi },%
\boldsymbol{\xi }_{\ast })\,h_{\beta ,j\ast }\,d\boldsymbol{\xi }_{\ast }$
for $\left\{ \alpha ,\beta \right\} \subseteq \left\{ 1,...,s\right\} $ and $%
\left( i,j\right) \in \left\{ 1,...,r_{\alpha }\right\} \times \left\{
1,...,r_{\beta }\right\} $.

Firstly, assume that $m_{\alpha }\neq m_{\beta }$.

Assume that the internal energy gap $\Delta _{il,kj}^{\alpha \beta
}=I_{i}^{\alpha }+I_{l}^{\beta }-I_{k}^{\alpha }-I_{j}^{\beta }$ and the
velocities $\boldsymbol{\xi }$ and $\boldsymbol{\xi }_{\ast }$ are given.
Then a collision will be uniquely determined by a unit vector $\boldsymbol{%
\eta }=\left( \boldsymbol{\xi }-\boldsymbol{\xi }^{\prime }\right)
/\left\vert \boldsymbol{\xi }-\boldsymbol{\xi }^{\prime }\right\vert $, or, $%
\boldsymbol{\omega }=\left( \boldsymbol{\xi }^{\prime }-\boldsymbol{\xi }%
_{\ast }^{\prime }\right) /\left\vert \boldsymbol{\xi }^{\prime }-%
\boldsymbol{\xi }_{\ast }^{\prime }\right\vert $. This follows, since, by
conservation of momentum and total energy $\left( \ref{CI}\right) $
(reminding the relabeling of the velocities and internal energies), the
relation between $\left\vert \boldsymbol{\xi }-\boldsymbol{\xi }_{\ast
}^{\prime }\right\vert $ and $\left\vert \boldsymbol{\xi }^{\prime }-%
\boldsymbol{\xi }_{\ast }\right\vert $ (or, equivalently, between $%
\left\vert \boldsymbol{\xi }^{\prime }-\boldsymbol{\xi }_{\ast }^{\prime
}\right\vert $ and $\left\vert \boldsymbol{\xi }-\boldsymbol{\xi }_{\ast
}\right\vert $) can be obtained, while also $m_{\beta }\left( \boldsymbol{%
\xi }_{\ast }-\boldsymbol{\xi }_{\ast }^{\prime }\right) =m_{\alpha }\left( 
\boldsymbol{\xi }-\boldsymbol{\xi }^{\prime }\right) $, cf. Figure \ref{fig3}%
. Indeed, note that - with the aim to obtain expressions for $\left\vert 
\mathbf{g}^{\prime }\right\vert $ and $\mathbf{g}_{\alpha \beta }^{\prime }=%
\dfrac{m_{\alpha }\boldsymbol{\xi }_{\ast }^{\prime }-m_{\beta }\boldsymbol{%
\xi }^{\prime }}{m_{\alpha }-m_{\beta }}$ in the arguments of the
delta-functions,%
\begin{eqnarray*}
&&W_{\alpha \beta }(\boldsymbol{\xi },\boldsymbol{\xi }_{\ast }^{\prime
},I_{i}^{\alpha },I_{l}^{\beta }\left\vert \boldsymbol{\xi }^{\prime },%
\boldsymbol{\xi }_{\ast },I_{k}^{\alpha },I_{j}^{\beta }\right. ) \\
&=&\left( m_{\alpha }+m_{\beta }\right) ^{2}m_{\alpha }m_{\beta }\varphi
_{i}^{\alpha }\varphi _{l}^{\beta }\sigma _{il,kj}^{\alpha \beta }\frac{%
\left\vert \widehat{\mathbf{g}}\right\vert }{\left\vert \overline{\mathbf{g}}%
\right\vert }\delta _{3}\left( \left( m_{\alpha }-m_{\beta }\right) \left( 
\mathbf{g}_{\alpha \beta }-\mathbf{g}_{\alpha \beta }^{\prime }\right)
\right)  \\
&&\times \delta _{1}\left( \frac{m_{\alpha }m_{\beta }}{2\left( m_{\alpha
}-m_{\beta }\right) }\left( \left\vert \mathbf{g}^{\prime }\right\vert
^{2}-\left\vert \mathbf{g}\right\vert ^{2}\right) +\Delta _{il,kj}^{\alpha
\beta }\right)  \\
&=&\frac{\left( m_{\alpha }+m_{\beta }\right) ^{2}}{\left( m_{\alpha
}-m_{\beta }\right) ^{2}}\sigma _{il,kj}^{\alpha \beta }\mathbf{1}%
_{\left\vert \mathbf{g}\right\vert ^{2}>2\widehat{\Delta }I_{il,kj}^{\alpha
\beta }}\frac{\varphi _{i}^{\alpha }\varphi _{l}^{\beta }\left\vert \widehat{%
\mathbf{g}}\right\vert }{\left\vert \mathbf{g}^{\prime }\right\vert
\left\vert \overline{\mathbf{g}}\right\vert }\delta _{3}\left( \mathbf{g}%
_{\alpha \beta }-\mathbf{g}_{\alpha \beta }^{\prime }\right)  \\
&&\times \delta _{1}\left( \left\vert \mathbf{g}^{\prime }\right\vert -\sqrt{%
\left\vert \mathbf{g}\right\vert ^{2}-2\widehat{\Delta }_{il,kj}^{\alpha
\beta }}\right) \text{, with }\mathbf{g}=\boldsymbol{\xi }-\boldsymbol{\xi }%
_{\ast }\text{, }\mathbf{g}^{\prime }=\boldsymbol{\xi }^{\prime }-%
\boldsymbol{\xi }_{\ast }^{\prime }\text{, } \\
\text{ } &&\widehat{\mathbf{g}}=\boldsymbol{\xi }_{\ast }^{\prime }-%
\boldsymbol{\xi }\text{, }\overline{\mathbf{g}}=\boldsymbol{\xi }_{\ast }-%
\boldsymbol{\xi }^{\prime }\text{, }\mathbf{g}_{\alpha \beta }=\dfrac{%
m_{\alpha }\boldsymbol{\xi }-m_{\beta }\boldsymbol{\xi }_{\ast }}{m_{\alpha
}-m_{\beta }}\text{, }\mathbf{g}_{\alpha \beta }^{\prime }=\dfrac{m_{\alpha }%
\boldsymbol{\xi }^{\prime }-m_{\beta }\boldsymbol{\xi }_{\ast }^{\prime }}{%
m_{\alpha }-m_{\beta }}\text{,} \\
&&\text{ }\widehat{\Delta }_{il,kj}^{\alpha \beta }=\frac{m_{\alpha
}-m_{\beta }}{m_{\alpha }m_{\beta }}\Delta _{il,kj}^{\alpha \beta }\text{,
and }\Delta _{il,kj}^{\alpha \beta }=I_{i}^{\alpha }+I_{l}^{\beta
}-I_{k}^{\alpha }-I_{j}^{\beta }\text{.}
\end{eqnarray*}%
Then, by a change of variables $\left\{ \boldsymbol{\xi }^{\prime },%
\boldsymbol{\xi }_{\ast }^{\prime }\right\} \rightarrow \left\{ \!\mathbf{g}%
^{\prime }=\boldsymbol{\xi }^{\prime }-\boldsymbol{\xi }_{\ast }^{\prime },%
\mathbf{g}_{\alpha \beta }^{\prime }=\dfrac{m_{\alpha }\boldsymbol{\xi }%
^{\prime }-m_{\beta }\boldsymbol{\xi }_{\ast }^{\prime }}{m_{\alpha
}-m_{\beta }}\!\right\} $ and then to spherical coordinates, where%
\begin{equation*}
d\boldsymbol{\xi }^{\prime }d\boldsymbol{\xi }_{\ast }^{\prime }=d\mathbf{g}%
^{\prime }d\mathbf{g}_{\alpha \beta }^{\prime }=\left\vert \mathbf{g}%
^{\prime }\right\vert ^{2}d\left\vert \mathbf{g}^{\prime }\right\vert d%
\mathbf{g}_{\alpha \beta }^{\prime }d\boldsymbol{\omega }\text{, with\ }%
\boldsymbol{\omega }=\frac{\mathbf{g}^{\prime }}{\left\vert \mathbf{g}%
^{\prime }\right\vert }\text{,}
\end{equation*}%
the expression $\left( \ref{k1}\right) $ of $k_{\alpha \beta ,ij}^{\left(
\beta ,2\right) }$ may be transformed to%
\begin{eqnarray}
&&k_{\alpha \beta ,ij}^{\left( \beta ,2\right) }(\boldsymbol{\xi },\boldsymbol{%
\xi }_{\ast }) \notag \\
&=&\sum\limits_{k=1}^{r_{\alpha }}\sum\limits_{l=1}^{r_{\beta
}}\int_{\mathbb{R}^{3}\times \mathbb{R}_{+}\times \mathbb{S}^{2}}%
\frac{w_{\alpha \beta }(\boldsymbol{\xi },\boldsymbol{\xi }_{\ast }^{\prime
},I_{i}^{\alpha },I_{l}^{\beta }\left\vert \boldsymbol{\xi }^{\prime },%
\boldsymbol{\xi }_{\ast },I_{k}^{\alpha },I_{j}^{\beta }\right. )}{\left(
M_{\alpha ,i}M_{\beta ,j\ast }\right) ^{1/2}}\left\vert \mathbf{g}^{\prime
}\right\vert ^{2}d\left\vert \mathbf{g%
}^{\prime }\right\vert d\mathbf{g}_{\alpha \beta }^{\prime } d\boldsymbol{\omega }  \notag \\
&=&\frac{\left( m_{\alpha }+m_{\beta }\right) ^{2}}{\left( m_{\alpha
}-m_{\beta }\right) ^{2}}\sum\limits_{k=1}^{r_{\alpha
}}\sum\limits_{l=1}^{r_{\beta }}\int_{\mathbb{S}^{2}}\left( M_{\alpha
,k}^{\prime }M_{\beta ,l\ast }^{\prime }\right) ^{1/2}\frac{\left\vert 
\widehat{\mathbf{g}}\right\vert \left\vert \mathbf{g}^{\prime }\right\vert }{%
\left\vert \overline{\mathbf{g}}\right\vert }\left( \frac{\varphi
_{i}^{\alpha }\varphi _{l}^{\beta }}{\varphi _{k}^{\alpha }\varphi
_{j}^{\beta }}\right) ^{1/2}  \notag \\
&&\times \mathbf{1}_{\left\vert \mathbf{g}\right\vert ^{2}>2\widehat{\Delta }%
I_{kj,il}^{\alpha \beta }}\sigma _{il,kj}^{\alpha \beta }\left( \left\vert 
\widehat{\mathbf{g}}\right\vert ,\frac{\widehat{\mathbf{g}}\cdot \overline{%
\mathbf{g}}}{\left\vert \widehat{\mathbf{g}}\right\vert \left\vert \overline{%
\mathbf{g}}\right\vert }\right) \,d\boldsymbol{\omega }\text{.}  \label{exp1}
\end{eqnarray}%
Here, see Figure \ref{fig3},%
\begin{equation*}
\left\{ 
\begin{array}{l}
\boldsymbol{\xi }^{\prime }=\boldsymbol{\xi }-\left\vert \boldsymbol{\xi }-%
\boldsymbol{\xi }^{\prime }\right\vert \boldsymbol{\eta } \\ 
\boldsymbol{\xi }_{\ast }^{\prime }=\boldsymbol{\xi }_{\ast }-\dfrac{%
m_{\alpha }}{m_{\beta }}\left\vert \boldsymbol{\xi }-\boldsymbol{\xi }%
^{\prime }\right\vert \boldsymbol{\eta }%
\end{array}%
\right. \text{, with }\boldsymbol{\eta }=\frac{\boldsymbol{\xi }-\boldsymbol{%
\xi }^{\prime }}{\left\vert \boldsymbol{\xi }-\boldsymbol{\xi }^{\prime
}\right\vert }\in \mathbb{S}^{2}\text{.}
\end{equation*}%
Then, by Lemma $\ref{L2}$, since relation $\left( \ref{vrel3}\right) $
follows by energy conservation, we have the following relation between the
kinetic parts of \ the exponents of the products $\left( M_{\alpha }^{\prime
}M_{\beta \ast }^{\prime }\right) ^{2}$ and $\left( M_{\alpha }M_{\beta \ast
}\right) ^{2}$, respectively 
\begin{equation}
m_{\alpha }\left\vert \boldsymbol{\xi }^{\prime }\right\vert ^{2}+m_{\beta
}\left\vert \boldsymbol{\xi }_{\ast }^{\prime }\right\vert ^{2}\geq \rho
\left( m_{\alpha }\left\vert \boldsymbol{\xi }\right\vert ^{2}+m_{\beta
}\left\vert \boldsymbol{\xi }_{\ast }\right\vert ^{2}\right) -2\left\vert
\Delta _{kj,il}^{\alpha \beta }\right\vert ,  \label{ineq1}
\end{equation}%
for some positive number $\rho $, where $0<\rho <1$. However, also%
\begin{eqnarray}
&&k_{\alpha \beta ,ij}^{\left( \beta ,2\right) }(\boldsymbol{\xi },%
\boldsymbol{\xi }_{\ast })  \notag \\
&=&\sum\limits_{k=1}^{r_{\alpha }}\sum\limits_{l=1}^{r_{\beta }}\int_{%
\mathbb{S}^{2}}\left\vert \widehat{\mathbf{g}}\right\vert \left( M_{\alpha
,k}^{\prime }M_{\beta ,l\ast }^{\prime }\right) ^{1/2}\left( \frac{\varphi
_{i}^{\alpha }\varphi _{l}^{\beta }}{\varphi _{k}^{\alpha }\varphi
_{j}^{\beta }}\right) ^{1/2}\sigma _{il,kj}^{\alpha \beta }\left( \left\vert 
\widehat{\mathbf{g}}\right\vert ,\cos \theta \right) \mathbf{1}_{\left\vert 
\widehat{\mathbf{g}}\right\vert ^{2}>2\widetilde{\Delta }I_{kj,il}^{\alpha
\beta }}d\boldsymbol{\omega }\text{, }  \notag \\
&&\text{with }\cos \theta =\boldsymbol{\omega }\cdot \frac{\widehat{\mathbf{g%
}}}{\left\vert \widehat{\mathbf{g}}\right\vert }\text{ and }\widetilde{%
\Delta }I_{kj,il}^{\alpha \beta }=\frac{m_{\alpha }+m_{\beta }}{m_{\alpha
}m_{\beta }}\Delta I_{kj,il}^{\alpha \beta }\text{,}  \label{exp2}
\end{eqnarray}%
and%
\begin{eqnarray}
&&k_{\alpha \beta ,ij}^{\left( \beta ,2\right) }(\boldsymbol{\xi },%
\boldsymbol{\xi }_{\ast })  \notag \\
&=&\sum\limits_{k=1}^{r_{\alpha }}\sum\limits_{l=1}^{r_{\beta }}\int_{%
\mathbb{S}^{2}}\left\vert \overline{\mathbf{g}}\right\vert \left( M_{\alpha
,k}^{\prime }M_{\beta ,l\ast }^{\prime }\right) ^{1/2}\left( \frac{\varphi
_{k}^{\alpha }\varphi _{j}^{\beta }}{\varphi _{i}^{\alpha }\varphi
_{l}^{\beta }}\right) ^{1/2}\sigma _{kj,il}^{\alpha \beta }\left( \left\vert 
\overline{\mathbf{g}}\right\vert ,\cos \theta \right) \mathbf{1}_{\left\vert 
\overline{\mathbf{g}}\right\vert ^{2}>2\widetilde{\Delta }I_{il,kj}^{\alpha
\beta }}d\boldsymbol{\omega }\text{, }  \notag \\
&&\text{with }\cos \theta =\boldsymbol{\omega }\cdot \frac{\overline{\mathbf{%
g}}}{\left\vert \overline{\mathbf{g}}\right\vert }\text{ and }\widetilde{%
\Delta }I_{il,kj}^{\alpha \beta }=\frac{m_{\alpha }+m_{\beta }}{m_{\alpha
}m_{\beta }}\Delta I_{il,kj}^{\alpha \beta }\text{.}  \label{exp3}
\end{eqnarray}

\textbf{H1.} If $\min \left( \left\vert \widehat{\mathbf{g}}\right\vert
,\left\vert \overline{\mathbf{g}}\right\vert \right) \geq \left\vert \mathbf{%
g}\right\vert $ and $\max \left( \left\vert \widehat{\mathbf{g}}\right\vert
,\left\vert \overline{\mathbf{g}}\right\vert \right) \geq 1$, then $\Psi
_{il,kj}^{\alpha \beta }=\left\vert \widehat{\mathbf{g}}\right\vert
\left\vert \overline{\mathbf{g}}\right\vert \geq \left\vert \mathbf{g}%
\right\vert $, and, hence, since $\left\vert \mathbf{g}^{\prime }\right\vert
\leq C\left( 1+\left\vert \mathbf{g}\right\vert \right) $, by expression $%
\left( \ref{exp1}\right) $, assumption $\left( \ref{est1}\right) $, and
bound $\left( \ref{ineq1}\right) $, one may obtain the following bound%
\begin{eqnarray*}
&&\left( k_{\alpha \beta ,ij}^{\left( \beta ,2\right) }(\boldsymbol{\xi },%
\boldsymbol{\xi }_{\ast })\right) ^{2} \\
&\leq &Ce^{-\rho \left( m_{\alpha }\left\vert \boldsymbol{\xi }\right\vert
^{2}+m_{\beta }\left\vert \boldsymbol{\xi }_{\ast }\right\vert ^{2}\right)
/2}\sum\limits_{k=1}^{r_{\alpha }}\sum\limits_{l=1}^{r_{\beta }}\left(
1+\left\vert \mathbf{g}\right\vert \right) ^{2}\left( 1+\frac{1}{\left\vert 
\mathbf{g}\right\vert ^{1-\gamma /2}}\right) ^{2}\left( \int_{\mathbb{S}%
^{2}}d\boldsymbol{\omega }\right) ^{2} \\
&\leq &Ce^{-\rho \left( m_{\alpha }\left\vert \boldsymbol{\xi }\right\vert
^{2}+m_{\beta }\left\vert \boldsymbol{\xi }_{\ast }\right\vert ^{2}\right)
/2}\left( \left\vert \mathbf{g}\right\vert ^{2}+\frac{1}{\left\vert \mathbf{g%
}\right\vert ^{2}}\right)
\end{eqnarray*}

\textbf{H2.} If $\left\vert \widehat{\mathbf{g}}\right\vert \geq \left\vert 
\mathbf{g}\right\vert $ and $\left\vert \overline{\mathbf{g}}\right\vert
\leq $ $\max \left( 1,\left\vert \mathbf{g}\right\vert \right) \leq
1+\left\vert \mathbf{g}\right\vert $, then 
\begin{equation*}
\Psi _{kj,il}^{\alpha \beta }\mathbf{1}_{\left\vert \overline{\mathbf{g}}%
\right\vert ^{2}>2\widetilde{\Delta }I_{il,kj}^{\alpha \beta }}=\left\vert 
\overline{\mathbf{g}}\right\vert \sqrt{\left\vert \overline{\mathbf{g}}%
\right\vert ^{2}-2\widetilde{\Delta }I_{il,kj}^{\alpha \beta }}\mathbf{1}%
_{\left\vert \overline{\mathbf{g}}\right\vert ^{2}>2\widetilde{\Delta }%
I_{il,kj}^{\alpha \beta }}\leq C\left( 1+\left\vert \mathbf{g}\right\vert
\right) ^{2},
\end{equation*}%
and hence, by expression $\left( \ref{exp3}\right) $, assumption $\left( \ref%
{est1}\right) $, and inequality $\left( \ref{ineq1}\right) $, one may obtain
the following bound%
\begin{eqnarray*}
&&\left( k_{\alpha \beta ,ij}^{\left( \beta ,2\right) }(\boldsymbol{\xi },%
\boldsymbol{\xi }_{\ast })\right) ^{2} \\
&\leq &Ce^{-\rho \left( m_{\alpha }\left\vert \boldsymbol{\xi }\right\vert
^{2}+m_{\beta }\left\vert \boldsymbol{\xi }_{\ast }\right\vert ^{2}\right)
/2}\sum\limits_{k=1}^{r_{\alpha }}\sum\limits_{l=1}^{r_{\beta }}\frac{\left(
\left( 1+\left\vert \mathbf{g}\right\vert \right) ^{2}+\left( 1+\left\vert 
\mathbf{g}\right\vert \right) ^{\gamma }\right) ^{2}}{\left\vert \mathbf{g}%
\right\vert ^{2}}\left( \int_{\mathbb{S}^{2}}d\boldsymbol{\omega }\right)
^{2} \\
&\leq &Ce^{-\rho \left( m_{\alpha }\left\vert \boldsymbol{\xi }\right\vert
^{2}+m_{\beta }\left\vert \boldsymbol{\xi }_{\ast }\right\vert ^{2}\right)
/2}\left( \left\vert \mathbf{g}\right\vert ^{2}+\frac{1}{\left\vert \mathbf{g%
}\right\vert ^{2}}\right)
\end{eqnarray*}

\textbf{H3.} If $\left\vert \overline{\mathbf{g}}\right\vert \geq \left\vert 
\mathbf{g}\right\vert $ and $\left\vert \widehat{\mathbf{g}}\right\vert \leq 
$ $\max \left( 1,\left\vert \mathbf{g}\right\vert \right) \leq 1+\left\vert 
\mathbf{g}\right\vert $, then 
\begin{equation*}
\Psi _{il,kj}^{\alpha \beta }\mathbf{1}_{\left\vert \widehat{\mathbf{g}}%
\right\vert ^{2}>2\widetilde{\Delta }I_{kl,ij}^{\alpha \beta }}=\left\vert 
\widehat{\mathbf{g}}\right\vert \sqrt{\left\vert \widehat{\mathbf{g}}%
\right\vert ^{2}-2\widetilde{\Delta }I_{kj,il}^{\alpha \beta }}\mathbf{1}%
_{\left\vert \widehat{\mathbf{g}}\right\vert ^{2}>2\widetilde{\Delta }%
I_{kj,il}^{\alpha \beta }}\leq C\left( 1+\left\vert \mathbf{g}\right\vert
\right) ^{2},
\end{equation*}%
and hence, by expression $\left( \ref{exp2}\right) $, assumption $\left( \ref%
{est1}\right) $, and inequality $\left( \ref{ineq1}\right) $, the bound%
\begin{equation}
\left( k_{\alpha \beta ,ij}^{\left( \beta ,2\right) }(\boldsymbol{\xi },%
\boldsymbol{\xi }_{\ast })\right) ^{2}\leq Ce^{-\rho \left( m_{\alpha
}\left\vert \boldsymbol{\xi }\right\vert ^{2}+m_{\beta }\left\vert 
\boldsymbol{\xi }_{\ast }\right\vert ^{2}\right) /2}\left( \left\vert 
\mathbf{g}\right\vert ^{2}+\frac{1}{\left\vert \mathbf{g}\right\vert ^{2}}%
\right)   \label{b6}
\end{equation}%
may again be obtained. However, it is clear that $\left\vert \mathbf{g}%
\right\vert \leq \max \left( \left\vert \widehat{\mathbf{g}}\right\vert
,\left\vert \overline{\mathbf{g}}\right\vert \right) $, why all
possibilities are covered by the cases \textbf{H1-3} above. Therefore, we
have the general bound $\left( \ref{b6}\right) $ for $\left( k_{\alpha \beta
,ij}^{\left( \beta ,2\right) }(\boldsymbol{\xi },\boldsymbol{\xi }_{\ast
})\right) ^{2}.$

By applying the bound $\left( \ref{b6}\right) $ and then first changing
variables of integration $\left\{ \boldsymbol{\xi },\boldsymbol{\xi }_{\ast
}\right\} \rightarrow \left\{ \mathbf{g},\mathbf{G}_{\alpha \beta }=\dfrac{%
m_{\alpha }\boldsymbol{\xi }+m_{\beta }\boldsymbol{\xi }_{\ast }}{m_{\alpha
}+m_{\beta }}\right\} $, with unitary Jacobian, followed by a change to
spherical coordinates, reminding the expression $\left( \ref{m2}\right) $, 
\begin{eqnarray*}
&&\int_{\left( \mathbb{R}^{3}\right) ^{2}}\left( k_{\alpha \beta
,ij}^{\left( \beta ,2\right) }(\boldsymbol{\xi },\boldsymbol{\xi }_{\ast
})\right) ^{2}d\boldsymbol{\xi \,}d\boldsymbol{\xi }_{\ast } \\
&\leq &C\int_{\left( \mathbb{R}^{3}\right) ^{2}}e^{-\rho \left( m_{\alpha
}+m_{\beta }\right) \left\vert \mathbf{G}_{\alpha \beta }\right\vert
^{2}/2-\rho m_{\alpha }m_{\beta }\left\vert \mathbf{g}\right\vert
^{2}/\left( 2\left( m_{\alpha }+m_{\beta }\right) \right) }\left( \left\vert 
\mathbf{g}\right\vert ^{2}+\frac{1}{\left\vert \mathbf{g}\right\vert ^{2}}%
\right) d\mathbf{g}\boldsymbol{\,}d\mathbf{G}_{\alpha \beta } \\
&\leq &C\int_{0}^{\infty }R^{2}e^{-R^{2}}dR\int_{0}^{\infty }\left( 1+\eta
^{4}\right) e^{-\rho m_{\alpha }m_{\beta }\eta ^{2}/\left( 2\left( m_{\alpha
}+m_{\beta }\right) \right) }d\eta =C\text{.}
\end{eqnarray*}%
Hence,%
\begin{equation*}
K_{\alpha \beta ,ij}^{(2)}=\int_{\mathbb{R}^{3}}k_{\alpha \beta ,ij}^{\left(
\beta ,2\right) }(\boldsymbol{\xi },\boldsymbol{\xi }_{\ast })\,h_{\beta
,j\ast }\,d\boldsymbol{\xi }_{\ast }
\end{equation*}%
are Hilbert-Schmidt integral operators, and as such continuous and compact
on $L^{2}\left( d\boldsymbol{\xi }\right) $ \cite[Theorem 7.83]%
{RenardyRogers}, for $\left( i,j\right) \in \left\{ 1,...,r_{\alpha
}\right\} \times \left\{ 1,...,r_{\beta }\right\} $ and $\left\{ \alpha
,\beta \right\} \subseteq \left\{ 1,...,s\right\} $.

On the other hand, if $m_{\alpha }=m_{\beta }$, then 
\begin{eqnarray*}
&&k_{\alpha \beta ,ij}^{\left( \beta ,2\right) }(\boldsymbol{\xi },%
\boldsymbol{\xi }_{\ast })=\sum\limits_{k=1}^{r_{\alpha
}}\sum\limits_{l=1}^{r_{\beta }}\int_{\left( \mathbb{R}^{3}\right) ^{\perp _{%
\mathbf{n}}}}4\frac{\left\vert \widehat{\mathbf{g}}\right\vert \left(
M_{\alpha ,k}^{\prime }M_{\beta ,l\ast }^{\prime }\right) ^{1/2}}{\left\vert 
\overline{\mathbf{g}}\right\vert \left\vert \mathbf{g}\right\vert }\left( 
\frac{\varphi _{i}^{\alpha }\varphi _{l}^{\beta }}{\varphi _{k}^{\alpha
}\varphi _{j}^{\beta }}\right) ^{1/2} \\
&&\times \mathbf{1}_{\left\vert \widehat{\mathbf{g}}\right\vert ^{2}>4\Delta
I_{kj,il}^{\alpha \beta }/m_{\alpha }}\sigma _{il,kj}^{\alpha \beta }\left(
\left\vert \widehat{\mathbf{g}}\right\vert ,-\frac{\widehat{\mathbf{g}}\cdot 
\overline{\mathbf{g}}}{\left\vert \widehat{\mathbf{g}}\right\vert \left\vert 
\overline{\mathbf{g}}\right\vert }\right) d\mathbf{w}\text{, with }\widehat{%
\mathbf{g}}=\boldsymbol{\xi }-\boldsymbol{\xi }_{\ast }^{\prime }\text{ and }%
\overline{\mathbf{g}}=\boldsymbol{\xi }_{\ast }-\boldsymbol{\xi }.
\end{eqnarray*}%
Here%
\begin{eqnarray*}
&&\left\{ 
\begin{array}{l}
\boldsymbol{\xi }^{\prime }=\boldsymbol{\xi }+\mathbf{w}-\chi \mathbf{n} \\ 
\boldsymbol{\xi }_{\ast }^{\prime }=\boldsymbol{\xi }_{\ast }+\mathbf{w}%
-\chi \mathbf{n}%
\end{array}%
\right. \text{, where }\mathbf{w}\perp \mathbf{n}\text{ and }\chi =\chi
_{kj,il}^{\alpha \beta }=\frac{\Delta I_{kj,il}^{\alpha \beta }}{m\left\vert 
\mathbf{g}\right\vert }\text{,} \\
&&\text{ with }\mathbf{g}=\boldsymbol{\xi }-\boldsymbol{\xi }_{\ast }\text{
and }\mathbf{n=}\frac{\mathbf{g}}{\left\vert \mathbf{g}\right\vert }\text{.}
\end{eqnarray*}%
Then similar arguments to the ones for $k_{\alpha \beta ,ij}^{\left( \alpha
\right) }(\boldsymbol{\xi },\boldsymbol{\xi }_{\ast })$ (with $m_{\alpha
}=m_{\beta }$) above, can be applied.

Concluding, the operator 
\begin{eqnarray*}
&&K=(K_{1},...,K_{s})=\sum\limits_{\beta =1}^{s}\left( (K_{1\beta
}^{(3)},...,K_{s\beta }^{(3)})-(K_{1\beta }^{(1)},...,K_{s\beta
}^{(1)})+(K_{1\beta }^{(2)},...,K_{s\beta }^{(2)})\right) \text{,} \\
&&\text{with }K_{\alpha \beta }^{(i)}=\sum\limits_{j=1}^{r_{\beta
}}(K_{\alpha \beta ,1j}^{(i)},...,K_{\alpha \beta ,r_{\alpha }j}^{(i)})\text{
for }i\in \left\{ 1,2,3\right\} \text{and}\left\{ \alpha ,\beta \right\}
\subseteq \left\{ 1,...,s\right\} \!\text{,}
\end{eqnarray*}%
is a compact self-adjoint operator on $\left( L^{2}\left( d\boldsymbol{\xi }%
\right) \right) ^{r}$. Self-adjointness is due to the symmetry relations $%
\left( \ref{sa1}\right) ,\left( \ref{sa2}\right) $, cf. \cite[p.198]%
{Yoshida-65}.
\end{proof}

\section{Bounds on the collision frequency \label{PT2}}

This section concerns the proof of Theorem \ref{Thm2}. Note that throughout
the proof, $C$ will denote a generic positive constant.

\begin{proof}
Under assumption $\left( \ref{e1}\right) $ the collision frequencies $\nu
_{\alpha ,i}$, by the change of variables $\left\{ \boldsymbol{\xi }^{\prime
},\boldsymbol{\xi }_{\ast }^{\prime }\right\} \rightarrow \left\{ \left\vert 
\mathbf{g}^{\prime }\right\vert ,\boldsymbol{\omega }=\dfrac{\mathbf{g}%
^{\prime }}{\left\vert \mathbf{g}^{\prime }\right\vert },\mathbf{G}_{\alpha
\beta }^{\prime }=\dfrac{m_{\alpha }\boldsymbol{\xi }^{\prime }+m_{\beta }%
\boldsymbol{\xi }_{\ast }^{\prime }}{m_{\alpha }+m_{\beta }}\right\} $,
equal 
\begin{eqnarray*}
&&\nu _{\alpha ,i} \\
&=&\sum\limits_{\beta =1}^{s}\sum\limits_{k=1}^{r_{\alpha
}}\sum\limits_{j,l=1}^{r_{\beta }}\int_{\left( \mathbb{R}^{3}\right)
^{2}\times \mathbb{R}_{+}\times \mathbb{S}^{2}}\frac{M_{\beta ,j\ast }}{%
\varphi _{i}^{\alpha }\varphi _{j}^{\beta }}W_{\alpha \beta }\left\vert 
\mathbf{g}^{\prime }\right\vert ^{2}d\boldsymbol{\xi }_{\ast }d\left\vert \mathbf{g}^{\prime }\right\vert d\mathbf{G}%
_{\alpha \beta }^{\prime }d%
\boldsymbol{\omega } \\
&=&4\pi \sum\limits_{\beta =1}^{s}\dfrac{n_{\beta }\varphi _{j}^{\beta
}m_{\beta }^{3/2}}{\left( 2\pi \right) ^{3/2}q_{\alpha }}\sum%
\limits_{k=1}^{r_{\alpha }}\sum\limits_{j,l=1}^{r_{\beta }}\int_{\mathbb{R}%
^{3}}e^{-I_{j}^{\beta }-m_{\beta }\left\vert \boldsymbol{\xi }_{\ast
}\right\vert ^{2}/2}\sigma _{ij,kl}^{\alpha \beta }\mathbf{1}_{\left\vert 
\mathbf{g}\right\vert ^{2}>2\widetilde{\Delta }I_{kl,ij}^{\alpha \beta
}}\left\vert \mathbf{g}\right\vert d\boldsymbol{\xi }_{\ast } \\
&=&\sum\limits_{\beta =1}^{s}C_{\alpha \beta }\frac{\sqrt{2}n_{\beta
}m_{\beta }^{3/2}}{\varphi _{i}^{\alpha }q_{\alpha }\sqrt{\pi }}%
\sum\limits_{k=1}^{r_{\alpha }}\sum\limits_{j,l=1}^{r_{\beta }}\int\limits_{%
\mathbb{R}^{3}}e^{-I_{j}^{\beta }-m_{\beta }\left\vert \boldsymbol{\xi }%
_{\ast }\right\vert ^{2}/2}\sqrt{\left\vert \mathbf{g}\right\vert ^{2}-2%
\widetilde{\Delta }I_{kl,ij}^{\alpha \beta }}\mathbf{1}_{\left\vert \mathbf{g%
}\right\vert ^{2}>2\widetilde{\Delta }I_{kl,ij}^{\alpha \beta }}d%
\boldsymbol{\xi }_{\ast }\text{,} \\
&&\text{with }\widetilde{\Delta }I_{kl,ij}^{\alpha \beta }=\frac{m_{\alpha
}+m_{\beta }}{m_{\alpha }m_{\beta }}\Delta I_{ij,kl}^{\alpha \beta }\text{,}
\end{eqnarray*}%
for $i\in \left\{ 1,...,r_{\alpha }\right\} $ and $\alpha \in \left\{
1,...,s\right\} $.

Clearly, for any $i\in \left\{ 1,...,r_{\alpha }\right\} $ and $\alpha \in
\left\{ 1,...,s\right\} ,$ since $\Delta I_{ii,ii}^{\alpha \alpha }=0$,%
\begin{eqnarray*}
\nu _{\alpha ,i} &\geq &C_{\alpha \alpha }\frac{\sqrt{2}n_{\alpha }m_{\alpha
}^{3/2}}{\varphi _{i}^{\alpha }q_{\alpha }\sqrt{\pi }}e^{-I_{i}^{\alpha
}}\int_{\mathbb{R}^{3}}e^{-m_{\alpha }\left\vert \boldsymbol{\xi }_{\ast
}\right\vert ^{2}/2}\left\vert \mathbf{g}\right\vert \,d\boldsymbol{\xi }%
_{\ast } \\
&=&C\int_{\mathbb{R}^{3}}\left\vert \mathbf{g}\right\vert e^{-m_{\alpha
}\left\vert \boldsymbol{\xi }_{\ast }\right\vert ^{2}/2}\,d\boldsymbol{\xi }%
_{\ast }\text{.}
\end{eqnarray*}

Now consider the two different cases $\left\vert \boldsymbol{\xi }%
\right\vert \leq 1$ and $\left\vert \boldsymbol{\xi }\right\vert \geq 1$
separately. Firstly, if $\left\vert \boldsymbol{\xi }\right\vert \geq 1$,
then, for any $i\in \left\{ 1,...,r_{\alpha }\right\} $ and $\alpha \in
\left\{ 1,...,s\right\} $, by a trivial estimate 
\begin{eqnarray*}
\nu _{\alpha ,i} &\geq &C\int_{\mathbb{R}^{3}}e^{-m_{\alpha }\left\vert 
\boldsymbol{\xi }_{\ast }\right\vert ^{2}}\left\vert \left\vert \boldsymbol{%
\xi }\right\vert -\left\vert \boldsymbol{\xi }_{\ast }\right\vert
\right\vert \,d\boldsymbol{\xi }_{\ast } \\
&\geq &C\int_{\left\vert \boldsymbol{\xi }_{\ast }\right\vert \leq
1/2}e^{-m_{\alpha }\left\vert \boldsymbol{\xi }_{\ast }\right\vert
^{2}}\left( \left\vert \boldsymbol{\xi }\right\vert -\left\vert \boldsymbol{%
\xi }_{\ast }\right\vert \right) \,d\boldsymbol{\xi }_{\ast } \\
&\geq &Ce^{-m_{\alpha }/4}\frac{\left\vert \boldsymbol{\xi }\right\vert }{2}%
\int_{\left\vert \boldsymbol{\xi }_{\ast }\right\vert \leq 1/2}\,d%
\boldsymbol{\xi }_{\ast }\geq C\left\vert \boldsymbol{\xi }\right\vert \geq
C\left( 1+\left\vert \boldsymbol{\xi }\right\vert \right) \text{.}
\end{eqnarray*}%
Secondly, if $\left\vert \boldsymbol{\xi }\right\vert \leq 1$, then, for any 
$i\in \left\{ 1,...,r_{\alpha }\right\} $ and $\alpha \in \left\{
1,...,s\right\} $, by trivial estimates and a change to spherical
coordinates, 
\begin{eqnarray*}
\nu _{\alpha ,i} &\geq &C\int_{\mathbb{R}^{3}}e^{-m_{\alpha }\left\vert 
\boldsymbol{\xi }_{\ast }\right\vert ^{2}}\left\vert \left\vert \boldsymbol{%
\xi }\right\vert -\left\vert \boldsymbol{\xi }_{\ast }\right\vert
\right\vert \,d\boldsymbol{\xi }_{\ast } \\
&\geq &C\int_{\left\vert \boldsymbol{\xi }_{\ast }\right\vert \geq
2}e^{-m_{\alpha }\left\vert \boldsymbol{\xi }_{\ast }\right\vert ^{2}}\left(
\left\vert \boldsymbol{\xi }_{\ast }\right\vert -\left\vert \boldsymbol{\xi }%
\right\vert \right) \,d\boldsymbol{\xi }_{\ast } \\
&\geq &C\int_{\left\vert \boldsymbol{\xi }_{\ast }\right\vert \geq
2}e^{-m_{\alpha }\left\vert \boldsymbol{\xi }_{\ast }\right\vert ^{2}}\frac{%
\left\vert \boldsymbol{\xi }_{\ast }\right\vert }{2}\,d\boldsymbol{\xi }%
_{\ast } \\
&=&C\int_{2}^{\infty }R^{3}e^{-m_{\alpha }R^{2}}\,dR=C\geq C\left(
1+\left\vert \boldsymbol{\xi }\right\vert \right) \text{.}
\end{eqnarray*}%
Hence, there is a positive constant $\nu _{-}>0$, such that $\nu _{\alpha
,i}\geq \nu _{-}\left( 1+\left\vert \boldsymbol{\xi }\right\vert \right) $
for all $i\in \left\{ 1,...,r_{\alpha }\right\} $, $\alpha \in \left\{
1,...,s\right\} $, and $\boldsymbol{\xi }\in \mathbb{R}^{3}$.

On the other hand, by some estimates and a change to spherical coordinates, 
\begin{eqnarray*}
\nu _{\alpha ,i} &\leq &C\sum\limits_{\beta
=1}^{s}\sum\limits_{k=1}^{r_{\alpha }}\sum\limits_{j,l=1}^{r_{\beta }}\int_{%
\mathbb{R}^{3}}e^{-m_{\beta }\left\vert \boldsymbol{\xi }_{\ast }\right\vert
^{2}/2}\sqrt{\left\vert \mathbf{g}\right\vert ^{2}-2\widetilde{\Delta }%
I_{kl,ij}^{\alpha \beta }}\,d\boldsymbol{\xi }_{\ast } \\
&\leq &C\sum\limits_{\beta =1}^{s}\sum\limits_{k=1}^{r_{\alpha
}}\sum\limits_{j,l=1}^{r_{\beta }}\int_{\mathbb{R}^{3}}\left( 1+\left\vert 
\boldsymbol{\xi }\right\vert +\left\vert \boldsymbol{\xi }_{\ast
}\right\vert \right) e^{-m_{\beta }\left\vert \boldsymbol{\xi }_{\ast
}\right\vert ^{2}/2}\,d\boldsymbol{\xi }_{\ast } \\
&\leq &C\left\vert \boldsymbol{\xi }\right\vert \sum\limits_{\beta
=1}^{s}\int_{0}^{\infty }R^{2}e^{-m_{\beta }R^{2}/2}dR+C\sum\limits_{\beta
=1}^{s}\int_{0}^{\infty }e^{-m_{\beta }R^{2}/2}\left( R^{2}+R^{3}\right) \,dR
\\
&\leq &C\left( 1+\left\vert \boldsymbol{\xi }\right\vert \right) \text{.}
\end{eqnarray*}%
Hence, there is a positive constant $\nu _{+}>0$, such that $\nu _{\alpha
,i}\leq \nu _{+}\left( 1+\left\vert \boldsymbol{\xi }\right\vert \right) $
for all $i\in \left\{ 1,...,r_{\alpha }\right\} $, $\alpha \in \left\{
1,...,s\right\} $, and $\boldsymbol{\xi }\in \mathbb{R}^{3}$.
\end{proof}

\section{Appendix: Proof of Lemma $\ref{L3}$}\label{secA1}

This appendix concerns a proof of Lemma $\ref{L3}$.

\begin{proof}
Denote $q:=\left\vert \boldsymbol{\xi }-\boldsymbol{\xi }^{\prime
}\right\vert $. For $\boldsymbol{\eta =}\left( \boldsymbol{\boldsymbol{\xi }-%
\boldsymbol{\xi }^{\prime }}\right) \boldsymbol{/}q$ the following (unique)
decomposition can be made, by the relations $\left( \ref{vrel2}\right) $,
cf. Figure $\ref{fig3}$, 
\begin{equation*}
\left\{ 
\begin{array}{l}
\boldsymbol{\xi }=\mathbf{w}+r\boldsymbol{\eta } \\ 
\boldsymbol{\xi }^{\prime }=\mathbf{w}+\left( r-q\right) \boldsymbol{\eta }%
\end{array}%
\right. \text{, with }\mathbf{w}\perp \boldsymbol{\eta }\text{, \ }
\end{equation*}%
while%
\begin{equation*}
\left\{ 
\begin{array}{l}
\boldsymbol{\xi }_{\ast }=\widetilde{\mathbf{w}}+r_{\ast }\boldsymbol{\eta }
\\ 
\boldsymbol{\xi }_{\ast }^{\prime }=\widetilde{\mathbf{w}}+\left( r_{\ast }-%
\dfrac{m_{\alpha }}{m_{\beta }}q\right) \boldsymbol{\eta }%
\end{array}%
\right. \text{, with }\widetilde{\mathbf{w}}\perp \boldsymbol{\eta }\text{,}
\end{equation*}%
where it follows by relation $\left( \ref{vrel3}\right) $ that 
\begin{equation*}
r_{\ast }=r-\chi +\frac{m_{\alpha }-m_{\beta }}{2m_{\beta }}q\text{, with }%
\chi =\chi _{kj,il}^{\alpha \beta }=\frac{\Delta I_{kj,il}^{\alpha \beta }}{%
m_{\alpha }q}\text{,}
\end{equation*}%
implying that%
\begin{equation*}
\left\{ 
\begin{array}{l}
\boldsymbol{\xi }_{\ast }=\widetilde{\mathbf{w}}+\left( r-\chi +\dfrac{%
m_{\alpha }-m_{\beta }}{2m_{\beta }}q\right) \boldsymbol{\eta } \\ 
\boldsymbol{\xi }_{\ast }^{\prime }=\widetilde{\mathbf{w}}+\left( r-\chi -%
\dfrac{m_{\alpha }+m_{\beta }}{2m_{\beta }}q\right) \boldsymbol{\eta }%
\end{array}%
\right. \text{, with }\widetilde{\mathbf{w}}\perp \boldsymbol{\eta }\text{. }
\end{equation*}%
Then%
\begin{eqnarray*}
&&m_{\alpha }\left\vert \boldsymbol{\xi }^{\prime }\right\vert ^{2}+m_{\beta
}\left\vert \boldsymbol{\xi }_{\ast }^{\prime }\right\vert ^{2} \\
&=&m_{\alpha }\left\vert \mathbf{w}\right\vert ^{2}+m_{\beta }\left\vert 
\widetilde{\mathbf{w}}\right\vert ^{2}+\left( m_{\alpha }+m_{\beta }\right)
\left( r^{2}+\chi q\right) -\left( 3m_{\alpha }+m_{\beta }\right) qr \\
&&+\frac{m_{\alpha }^{2}+6m_{\alpha }m_{\beta }+m_{\beta }^{2}}{4m_{\beta }}%
q^{2}+m_{\beta }\left( \chi ^{2}-2\chi r\right)  \\
&=&m_{\alpha }\left\vert \mathbf{w}\right\vert ^{2}+m_{\beta }\left\vert 
\widetilde{\mathbf{w}}\right\vert ^{2}+\left( m_{\alpha }+m_{\beta }\right)
\left( r-\frac{2m_{\beta }\chi +\left( 3m_{\alpha }+m_{\beta }\right) q}{%
2\left( m_{\alpha }+m_{\beta }\right) }\right) ^{2} \\
&&+\frac{m_{\alpha }m_{\beta }}{m_{\alpha }+m_{\beta }}\left( \chi +\frac{%
m_{\alpha }-m_{\beta }}{2m_{\beta }}q\right) ^{2}\text{,}
\end{eqnarray*}%
while%
\begin{eqnarray*}
&&m_{\alpha }\left\vert \boldsymbol{\xi }\right\vert ^{2}+m_{\beta
}\left\vert \boldsymbol{\xi }_{\ast }\right\vert ^{2} \\
&=&m_{\alpha }\left\vert \mathbf{w}\right\vert ^{2}+m_{\beta }\left\vert 
\widetilde{\mathbf{w}}\right\vert ^{2}+\left( m_{\alpha }+m_{\beta }\right)
r^{2}+\left( m_{\alpha }-m_{\beta }\right) q\left( r-\chi \right)  \\
&&+\frac{\left( m_{\alpha }-m_{\beta }\right) ^{2}}{4m_{\beta }}%
q^{2}+m_{\beta }\left( \chi ^{2}-2\chi r\right)  \\
&=&m_{\alpha }\left\vert \mathbf{w}\right\vert ^{2}+m_{\beta }\left\vert 
\widetilde{\mathbf{w}}\right\vert ^{2}+\left( m_{\alpha }+m_{\beta }\right)
\left( r+\frac{\left( m_{\alpha }-m_{\beta }\right) q-2m_{\beta }\chi }{%
2\left( m_{\alpha }+m_{\beta }\right) }\right) ^{2} \\
&&+\frac{m_{\alpha }m_{\beta }}{m_{\alpha }+m_{\beta }}\left( \chi -\frac{%
m_{\alpha }-m_{\beta }}{2m_{\beta }}q\right) ^{2}\text{.}
\end{eqnarray*}%
Denote 
\begin{eqnarray*}
a &:&=-\frac{2m_{\beta }\chi +\left( 3m_{\alpha }+m_{\beta }\right) q}{%
2\left( m_{\alpha }+m_{\beta }\right) }\text{, }b:=\frac{\left( m_{\alpha
}-m_{\beta }\right) q-2m_{\beta }\chi }{2\left( m_{\alpha }+m_{\beta
}\right) }\text{, and } \\
c^{2} &:&=\dfrac{\left( m_{\alpha }-m_{\beta }\right) ^{2}}{4\left(
m_{\alpha }+m_{\beta }\right) ^{2}}\dfrac{m_{\alpha }}{m_{\beta }}q^{2}.
\end{eqnarray*}%
Then 
\begin{align*}
& m_{\alpha }\left\vert \boldsymbol{\xi }^{\prime }\right\vert ^{2}+m_{\beta
}\left\vert \boldsymbol{\xi }_{\ast }^{\prime }\right\vert ^{2}-\rho \left(
m_{\alpha }\left\vert \boldsymbol{\xi }\right\vert ^{2}+m_{\beta }\left\vert 
\boldsymbol{\xi }_{\ast }\right\vert ^{2}\right)  \\
\geq & \left( m_{\alpha }+m_{\beta }\right) \left( \left( r+a\right)
^{2}-\rho \left( r+b\right) ^{2}+\left( 1-\rho \right) c^{2}\right) +\left(
1+\rho \right) \frac{m_{\alpha }-m_{\beta }}{m_{\alpha }+m_{\beta }}%
m_{\alpha }\chi q \\
=& \left( m_{\alpha }+m_{\beta }\right) \left( 1-\rho \right) \left( r^{2}+2%
\frac{a-\rho b}{1-\rho }r+\frac{a^{2}-\rho b^{2}}{1-\rho }+c^{2}\right)
+\left( 1+\rho \right) \frac{m_{\alpha }-m_{\beta }}{m_{\alpha }+m_{\beta }}%
\Delta I_{kj,il}^{\alpha \beta } \\
=& \left( m_{\alpha }+m_{\beta }\right) \left( \left( 1-\rho \right) \left(
r+\frac{a-\rho b}{1-\rho }\right) ^{2}+\frac{c^{2}}{1-\rho }\left( \rho
^{2}-2\rho \left( 1+\frac{\left( a-b\right) ^{2}}{2c^{2}}\right) +1\right)
\right)  \\
& +\left( 1+\rho \right) \frac{m_{\alpha }-m_{\beta }}{m_{\alpha }+m_{\beta }%
}\Delta I_{kj,il}^{\alpha \beta } \\
\geq & c^{2}\frac{m_{\alpha }+m_{\beta }}{1-\rho }\left( \rho ^{2}-2\rho
\left( 1+\frac{\left( a-b\right) ^{2}}{2c^{2}}\right) +1\right) +\left(
1+\rho \right) \frac{m_{\alpha }-m_{\beta }}{m_{\alpha }+m_{\beta }}\Delta
I_{kj,il}^{\alpha \beta } \\
\geq & \left( 1+\rho \right) \frac{m_{\alpha }-m_{\beta }}{m_{\alpha
}+m_{\beta }}\Delta I_{kj,il}^{\alpha \beta }\geq -2\left\vert \Delta
I_{kj,il}^{\alpha \beta }\right\vert  \\
& \text{if }0\leq \rho \leq 1+\frac{\left( a-b\right) ^{2}}{2c^{2}}-\frac{1}{%
2c^{2}}\sqrt{\left( a-b\right) ^{4}+4c^{2}\left( a-b\right) ^{2}}<1\text{.}
\end{align*}%
Note that the second last inequality follows by the assumption on $\rho $.
Now let%
\begin{eqnarray*}
\rho  &=&1+\frac{\left( a-b\right) ^{2}}{2c^{2}}-\frac{1}{2c^{2}}\sqrt{%
\left( a-b\right) ^{4}+4c^{2}\left( a-b\right) ^{2}} \\
&=&1+\frac{1-\sqrt{1+\dfrac{4c^{2}}{\left( a-b\right) ^{2}}}}{\dfrac{2c^{2}}{%
\left( a-b\right) ^{2}}}=1-\frac{2}{1+\sqrt{1+\dfrac{4c^{2}}{\left(
a-b\right) ^{2}}}} \\
&=&1-\frac{2}{1+\sqrt{1+\dfrac{\left( m_{\alpha }-m_{\beta }\right) ^{2}}{%
4m_{\alpha }m_{\beta }}}}=\left( \frac{\sqrt{m_{\alpha }}-\sqrt{m_{\beta }}}{%
\sqrt{m_{\alpha }}+\sqrt{m_{\beta }}}\right) ^{2}>0\text{ if }m_{\alpha
}\neq m_{\beta }\text{.}
\end{eqnarray*}
\end{proof}



\end{document}